%% file: MiniBatchQNPMODTR2.tex
\documentclass[onefignum,onetabnum]{siamonline250211}

\input{ex_shared}
\input{defs.tex}
\usepackage{color}
\usepackage{cleveref}
\crefname{assumption}{assumption}{assumptions}
\Crefname{assumption}{Assumption}{Assumptions}
\usepackage{silence}
\usepackage{cuted}
\usepackage{soul}

\usepackage{mathtools, cuted}
\usepackage{relsize}
\usepackage[export]{adjustbox}
\usepackage[dvipsnames]{xcolor}
\usepackage[skins]{tcolorbox}

\definecolor{darkred}{rgb}{0.7,0,0}
\definecolor{darkgreen}{rgb}{0,0.5,0}
\definecolor{darkblue}{rgb}{0,0,0.5}
\definecolor{goldenrod}{rgb}{0.85, 0.65, 0.13}
\usepackage{tikz,pgfplots}
\pgfplotsset{compat=1.5.1}
\usepgfplotslibrary{groupplots}
\usetikzlibrary{intersections,calc,arrows,matrix,spy,math,calc,fit,positioning,decorations, decorations.pathmorphing,plotmarks}
\tikzset{snake it/.style={decorate, decoration=snake}}
 \pgfplotsset{
        table/search path={figs/,figs/sim3D},
    }
\usetikzlibrary{arrows.meta,shapes.geometric}
\newcommand{\cb}{\color{black}}
\newcommand{\sumBt}{\umB^k}

\newcommand{\MajRev}[1]{{\color{black} #1}}
\newcommand{\MinRev}[1]{{\color{black} #1}}

\usepackage{float}

\usepackage{multicol}
\usepackage{slashbox}
\usepackage{diagbox}
\usepackage{mathrsfs}
\usepackage{booktabs}

\makeatletter

\usepackage{enumitem}
\setlist[enumerate]{leftmargin=.5in}
\setlist[itemize]{leftmargin=.5in}

\ifpdf
\hypersetup{
  pdftitle={An Example Article},
  pdfauthor={D. Doe, P. T. Frank, and J. E. Smith}
}
\fi

\begin{document}

\maketitle

\begin{abstract}
Over the years, computational imaging with accurate nonlinear physical models has garnered considerable interest due to its ability to achieve high-quality reconstructions.
However, using such nonlinear models for reconstruction is computationally demanding.
A popular choice for solving the corresponding inverse problems is the accelerated stochastic proximal method (ASPM), with the caveat that each iteration is still expensive.
To overcome this issue, we propose a \increm{} quasi-Newton proximal method (\ours{}) tailored to image reconstruction problems with constrained total variation regularization. Compared to ASPM, \ours{} requires fewer iterations to converge. Moreover, we propose an efficient approach to compute a weighted proximal mapping at a cost similar to that of the proximal mapping in ASPM. We also analyze the convergence of \ours{} \MajRev{in the nonconvex setting}. We assess the performance of \ours{} on three-dimensional inverse-scattering problems with linear and nonlinear physical models. Our results on simulated and real data demonstrate the effectiveness and efficiency of \ours{}, while also validating our theoretical analysis.
\end{abstract}



\section{Introduction}
The reconstruction of an image of interest from noisy measurements is a necessary step in many applications such as geophysical, medical, and optical imaging~\cite{neto2012introduction}. \MajRev{From some measurements, for example a set of $L$ acquired images of $M$ pixels $\{\uvy_l\in \mathbb{C}^M\}_{l=1}^L$, such a reconstruction is generally achieved by solving a composite minimization problem of the form}
\begin{equation}
\label{eq:defMainPro}
\min_{\uvx\in\mathcal C} \Phi(\uvx)\equiv \digamma(\uvx)+\lambda\,  h(\uvx),
\end{equation}
where $\digamma(\uvx)=\frac{1}{L}\sum_{l=1}^L f_l(\uvx)$ with $f_l(\uvx)=\frac{1}{2}\|\mathcal{H}_{l}(\mathrm{\uvx})- \uvy_l\|_2^2$, $\uvx\in\mathbb{R}^N$ is the vectorized image, and $\mathcal{C} \subset \mathbb{R}^{N}$ is a closed convex set. 
The data-fidelity terms $\{f_l\}_{l=1}^L$ ensure consistency with the measurements. \MajRev{Here, $N$ refers to the total volumes of the reconstructed image.} The (nonsmooth) regularization term~\MajRev{$h$} imposes some prior knowledge on the reconstructed image. The tradeoff parameter $\lambda>0$ is used to balance these two terms.
The operator~$\mathcal{H}_l:\mathbb{R}^N \rightarrow \mathbb{C}^M$ models the physical mapping from $\uvx$ to the measurements~$\uvy_l$. 

There is a growing interest in accurate physical models, motivated by the expectation that they can substantially improve reconstruction quality. Several imaging modalities have benefited from such refinements; for instance, optical diffraction tomography~(ODT)~\cite{pham2020three} or full waveform inversion~(FWI)~\cite{metivier2013full}.
However, accurate operators~$\mathcal{H}_l$ are usually nonlinear and require solving an additional system of equations iteratively when evaluating $\mathcal{H}_l(\uvx)$ (\emph{e.g.}, solving wave equations in ODT and FWI), which incurs high computational cost. Moreover, these operators result in the nonconvexity of $\{f_l\}_{l=1}^L$, introducing additional challenges in solving \eqref{eq:defMainPro}.

The regularization term \MajRev{$h$} incorporates prior information about the images to stabilize the reconstruction process. There \MajRev{exists} a plethora of options, such as total variation~(TV)~\cite{rudin1992nonlinear,hong2024complex}, the Hessian-Schatten norm~\cite{lefkimmiatis2013hessian}, deep-learning-based techniques~\cite{ulyanov2018deep}, and plug-and-play (PnP)/regularization by denoising (RED)~\cite{venkatakrishnan2013plug,romano2017little,hong2019acceleration,hong2024provable}, to name a few. Although recent priors may outperform TV, the latter is still widely used in 3D ODT~\cite{lim2019high,chowdhury2019high}. This observation motivates us to consider the constrained TV-based reconstruction, \textit{i.e.},
\begin{equation}
\label{eq:defMainProTV}
\min_{\uvx\in\mathcal C} \digamma(\uvx) +\lambda\,  \mrm{TV}(\uvx).
\end{equation} 
\MajRev{Note that we consider both isotropic and anisotropic TV throughout the paper; the explicit TV formulations are presented in \Cref{sec:pre:definationTV}.} Many iterative methods have been developed to handle the nonsmoothness of TV~\cite{chambolle2011upwind, chambolle2004algorithm,goldfarb2005second,chan2006total, beck2009fastTV}. In particular, Beck and Teboulle proposed the accelerated proximal method (APM)~\cite{beck2009fastTV}, which is one of the most popular first-order methods due to its low computational cost and fast convergence in many practical applications. \MajRev{If multiple machines are available, one can also adopt distributed optimization 
methods to solve~\eqref{eq:defMainProTV} \cite{alghunaim2020decentralized,boyd2011distributed}. 
However, this paper mainly focuses on optimization methods for solving 
\eqref{eq:defMainProTV} on a single machine.}

Quasi-Newton and Newton methods typically require fewer iterations than first-order methods in convex smooth optimization problems~\cite{nocedal2006numerical,schmidt2009optimizing}, owing to their use of second-order information. The quasi-Newton proximal methods (QNPMs) are variants adapted to composite problems~\cite{kim2010tackling,lee2014proximal,karimi2017imro,becker2019quasi,hong2020solving}. Ge \textit{et al.} \cite{ge2020proximal} and Hong \textit{et al.} \cite{hong2024complex} applied QNPMs to solve convex inverse problems with $L=1$ in X-ray imaging and magnetic resonance imaging, respectively. Kadu \textit{et al.} \cite{kadu2020high} used QNPMs for a nonlinear and nonconvex inverse-scattering problem with $L>1$. 
In their work, the authors observed faster convergence than APMs. Moreover, QNPMs can be seen as first-order methods with a variable metric. This perspective has led to another class of algorithms called variable metric operator splitting methods (VMOSMs) \cite{chouzenoux2014variable,bonettini2016variable,repetti2021variable}. We refer the reader to the prior work section in \cite{becker2019quasi}, where Becker \emph{et al.} discussed the relations between QNPMs and VMOSMs. However, these deterministic methods require the computation of the full gradient at each iteration, which can be prohibitive for $L\gg1$.

Stochastic methods are efficient iterative algorithms that mitigate the computational burden {{when $L\gg 1$}}.
These methods estimate the gradient from a (varying) subset of $\{f_l\}_{l=1}^L$ at each iteration~\cite{bottou2010large,johnson2013accelerating,defazio2014saga,schmidt2017minimizing}, making the computational cost independent of $L$.
The stochastic counterpart of APMs has been used in many instances of image reconstruction~\cite{chouzenoux2017stochastic,soubies2017efficient,pham2020three}. Thus, stochastic or incremental second-order methods, such as SLBFGS~\cite{moritz2016linearly}, IQN~\cite{mokhtari2018iqn}, and  SdLBFGS-VR~\cite{agarwal2017second} have been proposed to address $\min_{\uvx} \sum_l f_l(\uvx)$. Note that SLBFGS and IQN  assume that $\{f_l\}_{l=1}^L$ are strongly convex, while SdLBFGS-VR does not require a convexity assumption. 

The most challenging aspect of stochastic second-order methods is the reliable estimation of the Hessian from noisy gradient information. To address this difficulty, variance-reduction techniques have proven to be effective~\cite{moritz2016linearly,agarwal2017second,wang2017stochastic,goldfarb2020practical}. Other methods~\cite{curtis2016self,goldfarb2020practical,yang2022stochastic} were proposed to address the nonconvex setting. Wang \textit{et al.} \cite{wang2019stochastic} extended variance-reduced stochastic quasi-Newton methods to solve composite problems with \MajRev{$h=\|\cdot\|_1$} and nonconvex functions $\{f_l\}_{l=1}^L$. Using the first-order optimality conditions of \eqref{eq:defMainPro}, Yang \textit{et al.} \cite{yang2022stochastic} proposed a stochastic extra-step quasi-Newton method to find the solution of \eqref{eq:defMainPro} by solving a related nonlinear and nonsmooth equation. Wang \textit{et al.} \cite{zhang2025proximal} introduced a proximal stochastic quasi-Newton proximal method with an adaptive sampling scheme and a novel stochastic line search.  However, these methods either require the evaluation of the full gradient at regular intervals \cite{wang2019stochastic,yang2022stochastic} or involve extra step for computing the gradient and function value \cite{zhang2025proximal} during optimization. These can hinder the deployment of quasi-Newton proximal methods to large-scale imaging modalities such as 3D ODT, where $L$ is large and the physical model is nonlinear, making the computation of the gradient or function value expensive even for a single measurement.

\section{Contributions and Roadmap}\label{sec:contributions}
In this work, we derive a \emph{\increm{} quasi-Newton proximal method} (\ours{}) that never requires the evaluation of the full gradient. Moreover, our experiments demonstrate that \ours{} converges faster than the accelerated stochastic proximal method (ASPM) and the variance-reduced quasi-Newton proximal method \cite{wang2019stochastic}, both in terms of iterations and wall time. Compared to first-order proximal methods, QNPMs require computing a weighted proximal mapping~(WPM)\footnote{The WPM is defined in \Cref{sec:subsectionWPMProperties}.} at each iteration, which can be as challenging as the original problem. When \MajRev{$h=\mathrm{TV}$}, the authors in \cite{ge2020proximal,kadu2020high} computed the WPM using first-order methods such as FISTA or primal-dual methods.
Their algorithm involves inner and outer iterations, which adds to the global complexity (\textit{i.e.,} a three-layered iterative optimization). Leveraging the dual formulation of TV in a manner similar to the seminal work of Beck and Teboulle~\cite{beck2009fastTV}, we adapt the \MajRev{fast dual projected-gradient method (FDPGM) \cite{polyak2013dual}} to compute the WPM. This avoids the embedding of additional iterative algorithms and ensures fast convergence. Although the methodology for the proposed computation of WPM is similar to that described in~\cite{hong2024complex}, 
working under a set of \MinRev{constraints} $\vx\in\mathcal C$ requires computing an additional WPM to obtain the gradient in FDPGM at each iteration. By using the structure of the estimated Hessian matrices in \ours{}, we show that the additional WPM can be computed with negligible cost. Note that the images of interest are 3D. Therefore, to reduce memory usage when estimating the Hessian matrices, we employ a memory-efficient \MajRev{symmetric rank-$1$ (SR1) method \cite{becker2019quasi}}. Moreover, we analyze the convergence of \ours{} in the \emph{nonconvex} setting. Our 3D ODT experiments demonstrate that the proposed method requires fewer iterations and less computational time than first-order approaches to achieve high-quality reconstructions. These results indicate that our method is well suited for large-scale and nonlinear inverse problems. We further provide empirical evidence supporting our theoretical analysis through comprehensive numerical experiments. Although we only discussed TV regularization in this paper, \ours{} can be extended to broader regularizers, \textit{e.g.}, the Hessian-Schatten norm~\cite{lefkimmiatis2013hessian}.

In summary, the main contributions of our paper are given as follows:
\begin{itemize}
    \item We propose a \emph{\increm{} quasi-Newton proximal method}, in which the computation at each iteration is independent of the number of measurements. Moreover, our method does not require evaluating the full gradient at any iteration.
    \item We introduce an efficient approach to compute the WPM when considering a constrained TV regularizer. Furthermore, we adapt a memory-efficient SR1 method for Hessian estimation to reduce memory usage.
    \item We analyze the convergence properties of \ours{} \MajRev{in the nonconvex setting} and extensively test its performance on simulated and real data, as well as empirically validate our theoretical results.
\end{itemize}

The paper is organized as follows: \Cref{sec:preliminaries} introduces the notation and relevant preliminaries. \Cref{sec:IncrementalWPM} derives the proposed \ours{} and presents some implementation details. The convergence analysis of \ours{} is summarized in \Cref{sec:ConvAnalysis}. \Cref{sec:NumericalExperiments} studies the performance of \ours{} on three-dimensional inverse-scattering problems with simulated and real data. \MajRev{The conclusion is provided in \Cref{sec:conclusionAndfuture}.}
 
 \section{Preliminaries}
\label{sec:preliminaries}
In this section, we introduce the necessary \MinRev{notations} and present the discretized form of TV along with its dual representation. We then define the WPM, outline its key properties, and introduce a useful theorem.

\subsection{Notation}
Throughout the paper, vectors and matrices are represented in upright bold font.  \MajRev{Given $\umX\in\mathbb R^{N\times N}$, the notation $\umX\succ 0$ denotes that $\umX$ is a symmetric positive definite matrix.} For $\umX_1,\umX_2\in\mathbb R^{N\times N}$, $\umX_1\succeq\umX_2$ implies that $\umX_1-\umX_2$ is symmetric positive semidefinite.
 The \MajRev{$n$-th} element of a vector $\uvx\in \mathbb R^N$ is represented as $x_n$. \MajRev{The $n$-th column of a matrix $\umX\in \mathbb R^{D\times N}\,(D\geq 1)$ is denoted by $\uvx_n$.} The $(N\times N)$ identity matrix is denoted by $\umI_{N}$.
The notation $\langle \cdot,\cdot \rangle$ stands for the inner product. \MajRev{Let $\uvx \in \mathbb R^N$ be the vectorization of a $D$-dimensional array, so that $N=\Pi_{d=1}^D R_{d}$, where $R_d\geq 1$ denotes the number of elements for $d$-th dimension. For any $d\in[1,\cdots,D]$ we define the finite-difference matrix $\umD^{d}\in \mathbb R^{N\times N}$ such that the $(i,j)$-th component corresponds to $-\delta[i-j]+\delta[i-j-R_{d}]$, where $r\in\mathbb Z,\,\delta[r]=1$ if $r=0,$ and $0$ otherwise. The expectation is denoted by $\mathbb{E}[\cdot]$. The associated probability distributions will be specified when needed. \MinRev{The conditional expectation is denoted by $\mathbb{E}[\cdot \mid \uvx]$, where $\uvx$ is the conditioning variable.}
}



\subsection{Discretized Total Variation}
\label{sec:pre:definationTV}
We present two popular variants of TV: isotropic and anisotropic \cite{rudin1992nonlinear}.
The isotropic discretized total variation of ${\mathbf{x}}$ is defined as
\MajRev{\begin{equation}
\label{eq:TViso}\mrm{TV}_{{\mathrm{iso}}}({\mathbf{x}})={\mathrm{tr}}\left(\sqrt{\sum_{d=1}^{D}\,\left({\mathbf{D}}^{d}\,{\mathbf{x}}\right)\,\left({\mathbf{D}}^{d}\,{\mathbf{x}}\right)^{{\Trans}}}\right),
\end{equation}}
while the anisotropic version is defined as
\begin{equation}
\label{eq:TVl1}\mrm{TV}_{\ell_{1}}({\mathbf{x}})={\mathrm{tr}}\left(\sum_{d=1}^{D}\,\sqrt{\left({\mathbf{D}}^{d}\,{\mathbf{x}}\right)\,\left({\mathbf{D}}^{d}\,{\mathbf{x}}\right)^{{\Trans}}}\right),
\end{equation}
where $\Trans$ represents the transpose operator and $D$ represents the image dimension. In~\eqref{eq:TViso} and~\eqref{eq:TVl1}, the square root is applied component-wise.

\subsection{Dual Representation of Total Variation}
Using
$
\left\|{\mathbf{x}}\right\|=\max_{\uvz \in{\mathbb{R}}^{N},\left\|\uvz\right\|_* \leq1}\,\uvz^{{\Trans}}\,{\mathbf{x}},
$
where $\|\cdot\|_*$ denotes the dual norm of $\|\cdot\|$,
we rewrite~\eqref{eq:TViso} and~\eqref{eq:TVl1} as \cite{chambolle2004algorithm}
\begin{equation}
\label{eq:B&Tiso}
\mrm{TV}_{{\mathrm{iso}}}({\mathbf{x}})~=\max_{\substack{ {\mathbf{P}}\in{\mathbb{R}}^{D\times N} \\ \left\{\left\|{\mathbf{p}}_{n}\right\|_{\cb 2}\leq1\right\}_{n=1}^{N}}} \,{\mathbf{d}(\umP)}^{{\Trans}}\,{\mathbf{x}}
\end{equation}
and
\begin{equation}
\label{eq:B&Tl1}
\mrm{TV}_{\ell_{1}}({\mathbf{x}})~=\max_{\substack{{\mathbf{P}}\in{\mathbb{R}}^{D\times N}\\ \left\{\left\|{\mathbf{p}}_{n}\right\|_{\cb \infty}\leq1\right\}_{n=1}^{N}}}\,{\mathbf{d}(\umP)}^{{\Trans}}\,{\mathbf{x}},
\end{equation}
respectively. Further, the $\left(D\times N\right)$ matrix
$
{\mathbf{P}}=\left[{\mathbf{p}}_{1}\cdots{\mathbf{p}}_{N}\right]=\left[{\mathbf{q}}_{1}\cdots{\mathbf{q}}_{D}\right]^{{\Trans}}
$
contains the variables over which the maximization is performed. Furthermore, the vector-valued function ${\mathbf{d}}:\mathbb{R}^{D\times N}\rightarrow{\mathbb{R}}^{N}$ is given by
\MajRev{$
{\mathbf{d}(\umP)}=\sum_{d=1}^{D}\,\left({\mathbf{D}}^{d}\right)^{{\Trans}}\,{\mathbf{q}}_{d}.
$ Meanwhile, we define the adjoint function of $\uvd$ as $\uvd^\Trans$, 
which satisfies $\langle \uvx,\, \uvd(\umP) \rangle
= \langle \uvd^\Trans(\uvx),\, \umP \rangle.$}

\subsection{Weighted Proximal Mapping (WPM)}
\label{sec:subsectionWPMProperties}
In this part, we introduce the definition of  WPM and then discuss some key properties. 
\begin{definition}[Weighted proximal mapping]
Given a proper closed convex function \MajRev{$h$} and a symmetric positive definite matrix $\umW\in\mathbb{R}^{N\times N},~\umW\succ0$, the WPM associated with $h$ is defined as
\begin{equation}
\mrm{prox}_{h}^{\umW}(\uvx)=\arg\min\limits_{\uvu\in\mathbb R^N}\left(h(\uvu)+\frac{1}{2}\|\uvu-\uvx\|^2_{\umW}\right),\,\MajRev{\forall \uvx\in
\mathbb R^N,}
	\label{eq:def:WeightedProximal}
\end{equation} 
where $\|\uvx\|_{\umW}\triangleq{\sqrt{\uvx^\Trans\umW\uvx}}$ denotes the $\umW$-norm \MajRev{of $\uvx$}.
\end{definition}
Next, we outline some properties of \eqref{eq:def:WeightedProximal}:
\begin{enumerate}[label=\arabic*)]
\item The $\mrm{prox}_{h}^{\umW}(\uvx)$ exists and is unique for $\uvx\in\mathbb{R}^N$ since $h(\uvu)+\frac{1}{2}\|\uvu-\uvx\|^2_{\umW}$ is strongly convex.
	\item \MinRev{The generalized Moreau envelope of $h$ with respect to $\umW$ is defined as}
	$$
	\MinRev{\mathcal M_h^{\umW}}(\uvx)= \inf_{\vu\in\mathbb R^N} \left(h(\vu)+\frac{1}{2}\|\vu-\uvx\|_{\umW}^2\right).
	$$
	The function \MinRev{$\mathcal M_h^{\umW}$} on $\uvx\in \mathbb R^N$ is continuously differentiable with \MajRev{gradient~\cite[Theorem~5.30]{beck2017first}}
	\begin{equation}
	    {\pmb{ \nabla}_{\uvx}} \MinRev{\mathcal M_h^{\umW}} (\uvx)=\umW \left(\uvx-\mrm{prox}_{h}^{\umW}(\uvx)\right),\,\MajRev{\forall \uvx\in\mathbb R^N,}
	    \label{eq:WPM:gradient}
	\end{equation}
	and Lipschitz constant $\sigma_{\umW}$, which is the largest eigenvalue of $\umW$.
\end{enumerate}
See~\cite{lee2014proximal,becker2019quasi} for further details on the WPM.


For $\umW=\umI_{N}$, WPM becomes the proximal mapping~\cite{parikh2014proximal} which has a closed-form solution for many popular~$h$ \cite[Chapter 6]{beck2017first}. 
Although this does not necessarily carry over to $\mrm{prox}_{h}^{\umW}(\uvx)$ with a generic~$\umW$,
the computation of $\mrm{prox}_{h}^{\umW}(\uvx)$ can be simplified by using  \Cref{them:structuredWPM:evaluation}
if $\umW={\cb \bm \Sigma}\pm\umU\umU^\Trans$, where ${\cb \bm \Sigma}\in\mathbb R^{N\times N}$ is a diagonal matrix and $\umU\in\mathbb R^{N\times r}$ is $\mbox{rank-r}$ matrix with $r\ll N$. 
\begin{theorem}\cite[Theorem 3.4]{becker2019quasi}
\label{them:structuredWPM:evaluation}
	 Let $\umW={\cb \bm \Sigma}\pm\umU\umU^\Trans$, $\umW\in\mathbb R^{N\times N},\,\umW \succ 0$, and $\umU\in\mathbb R^{N\times r}$. Then, it holds that 
	 \begin{equation}
	 \label{eq:WPM:structure:equiv}
	 	 \mrm{prox}_{h}^{\umW}(\uvx)=\mrm{prox}_{h}^{{\cb \bm \Sigma}}(\uvx\mp{\cb \bm \Sigma}^{-1}\umU\bm \beta^*),\,\MajRev{\forall \uvx\in
\mathbb R^N,}
	 \end{equation}
	 where $\bm\beta^*\in\mathbb R^r$ is the unique solution of the nonlinear system of equation
	 \begin{equation}
\label{eq:WPM:structure:equiv:nonlinearSystems}
	 	 \underbrace{\umU^\Trans\left(\uvx-\mrm{prox}_{h}^{{\cb \bm \Sigma}}\left(\uvx\mp{\cb \bm \Sigma}^{-1}\umU\bm\beta\right)\right)+\bm\beta}_{\varphi(\bm\beta)} =\bm 0.
	 \end{equation}
\end{theorem}
Since ${\cb \bm \Sigma}$ is a diagonal matrix, computing $\mrm{prox}_{h}^{{\cb \bm \Sigma}}(\uvx)$ is as straightforward as the proximal mapping associated with $h$. 
To solve \eqref{eq:WPM:structure:equiv:nonlinearSystems}, we employ a semi-smooth Newton method~\cite{qi1999survey} because $r$ is small. In practice, we find that a few iterations are sufficient to obtain an accurate solution. \Cref{sec:proposed:sub:impDetails} provides more details about the implementation of the semi-smooth Newton method.

\section{Proposed Mini-Batch Quasi-Newton Proximal Method}
\label{sec:IncrementalWPM}
In this section, we first review the full batch quasi-Newton proximal method (\oursF{}) for solving \eqref{eq:defMainPro} and then present \ours{}. Splitting the index set $\{1,2,\ldots,L\}$ into $\Ksubset$ non-overlapping subsets $\{\mathscr{S}_s\}_{s=1}^\Ksubset$\footnote{\MajRev{In our current setting, we require each batch to contain non-overlapping subsets. However, the subsequent convergence analysis under Strategy~II does not rely on this assumption. It would therefore be interesting to investigate how the algorithm behaves when batches are allowed to contain overlapping subsets. We leave this direction for future work.}}, \MajRev{we rewrite \eqref{eq:defMainPro}} as
\begin{equation}
\min_{\uvx\in\mathcal C} \left(\frac{1}{ \Ksubset}\sum_{s=1}^\Ksubset F_s(\uvx)+ \bar{h}(\uvx)\right),\label{eq:NewModelPro}
\end{equation}
where $\bar{h}(\uvx)=\lambda\, h(\uvx)$ and \MajRev{$F_{s}(\uvx)= \frac{S}{L}\sum_{l\in\mathscr{S}_s} f_l(\uvx)$}. For the sake of brevity, we write $\sum_{s=1}^{\Ksubset}$ as $\sum_s$. At the \MajRev{$k$-th} iteration, \oursF{} obtains the next image iterate by solving a WPM:
\begin{equation}
\label{eq:oursF:WPM}
\uvx_k=\mathrm{prox}^{{\bar{\umH}_k}^{-1}}_{\stepsize \bar{h}+\iota_{\mathcal C}}\left(\uvx_{k-1}- \frac{\stepsize}{S}\bar{\umH}_k \sum_s \pmb{\nabla} F_s(\uvx_{k-1})\right),
\end{equation}
where $\bar{\umH}_k \in\mathbb R^{N\times N},~ \bar{\umH}_k\succ 0$ is the inversion of the estimated Hessian matrix at the \MajRev{$k$-th} iteration, $a_k$ is the stepsize, and $\iota_{\mathcal C}$ represents the characteristic function such that $\iota_{\mathcal C}(\uvx)=0,\,\uvx\in\mathcal C; +\infty,\,\uvx\notin \mathcal C$.
The techniques used in quasi-Newton methods for estimating Hessian matrices can be adapted here to estimate $\bar{\umH}_k$ \cite{nocedal2006numerical}.

Note that \eqref{eq:oursF:WPM} requires computing the full gradient, which can be prohibitively expensive for a large $L$. Indeed, in ODT, even computing $\pmb{\nabla} f_l$ is computationally expensive. To address this issue, we propose \ours{},  which computes the gradient of a single subset $\mathscr S_s$ at each iteration and estimates the Hessian matrices based on partial gradients. Moreover, \ours{} does not require computing the full gradient throughout the entire iteration.

\begin{algorithm}[t]        
\caption{Proposed \increm{} quasi-Newton proximal method (\ours{})}
\label{alg:IncrementalWPMs}                  
\begin{algorithmic}[1]
\REQUIRE Initial guess $\uvx_0\in\mathbb{R}^N$; tradeoff parameter $\lambda$; $\Ksubset$ subsets $\{\mathscr S_s\}_{s=1}^\Ksubset$; stepsize $\stepsize$; Lipschitz constants $\alpha_{s}$ of $F_{s} ~\text{for all}~{s}$; maximal number of iterations Max\_Iter
\lastcon $\uvx^*$
\STATE $k\leftarrow 1$
\FORALL{$k\leq\mrm{Max}\_\mrm{Iter}$}
\IF{$k\leq \Ksubset$} \label{alg:IncrementalWPMs:firstk_begin}
\STATE $s\leftarrow k$
\STATE Set $\uvx_s^k\leftarrow \uvx_{k-1},\,\uvg_s^k\leftarrow {\pmb{\nabla}} F_s (\uvx_{k-1}),\, \umB_s^k\leftarrow \alpha_s \umI_{N}$ \label{alg:IncrementalWPMs:firstk_setting}\\[5pt]
\STATE  {$\uvx_k \leftarrow \mathrm{prox}_{\stepsize\lambda\,\mathrm{TV}+\iota_{\mathcal C}}^{\umB_s^k}\left(\uvx_{k-1}-\stepsize (\umB_s^k)^{-1}\uvg_s^k\right)$
\label{alg:IncrementalWPMs:firstk_end}}
\ELSE

\STATE Pick $\{\uvg_s^k,\,\uvx_s^k,\,\umB_s^k\}_{s,k}$ (See \Cref{sec:IncrementalWPM:sub:choiceIterGradHess}) \label{alg:IncrementalWPMs:enterQN} \\[5pt]
\STATE $\sumBt\leftarrow \sum_s \umB_s^k$\\[5pt]
\STATE $\uvv_k\leftarrow \left(\sumBt\right)^{-1}\sum_s\left({\umB}_s^k \uvx_s^k - a_k\uvg_s^k \right)$\\[5pt]
\STATE $\uvx_k \leftarrow \mathrm{prox}_{\stepsize S\lambda\,\mathrm{TV}+\iota_{\mathcal C}}^{\umB^k} (\uvv_k)$
\label{alg:IncrementalWPMs:endQN}
\ENDIF 
\STATE {$k\leftarrow k+1$}
\ENDFOR
\RETURN $\uvx^*\leftarrow \uvx_{\mrm{Max}\_\mrm{Iter}}$
\end{algorithmic}
\end{algorithm}

For given $\uvx_s^k,\,\uvg_s^k,\,\umB_s^k\succ 0$, we define 
\begin{equation}
\label{eq:TaylorAppSubFunc}
\bar F_s^k(\uvx)=F_s(\uvx_s^k)+\left\langle \uvg_s^k, \uvx-\uvx_s^k\right\rangle +\frac{1}{2\stepsize}\|\uvx-\uvx_s^k\|^2_{\umB_s^k},
\end{equation}
as the local second-order Taylor approximation of $F_s(\uvx)$ at the \MajRev{$k$-th} iteration. Then, at iteration~$k > \Ksubset$, \ours{} computes $\uvx_{k}$ by solving the following minimization problem: 
\begin{equation}
\label{eq:SecondProximal_TV} 
\uvx_{k} = \arg\min\limits_{\uvx\in\mathcal C} \left(\frac{1}{\Ksubset}\sum_s \bar F_s^k(\uvx)+ \bar{h}(\uvx)\right).
\end{equation}
Rewriting the quadratic and linear terms in $\uvx-\uvx_s^k$ of \eqref{eq:TaylorAppSubFunc}, we recast \eqref{eq:SecondProximal_TV} as a WPM\footnote{\MajRev{\Cref{app:deduce:aylorAppSubFuncToSecondProximal_TV} presents the detailed derivation from \eqref{eq:SecondProximal_TV} to \eqref{eq:SecondProximal_WPMTV}}.}
\begin{equation}
\uvx_{k} =
\arg\min\limits_{\uvx\in\mathcal C} \left(\frac{1}{2}\|\uvx-\uvv_k\|_{\sumBt}^2+a_k\Ksubset\lambda\, \mrm{TV}(\uvx)\right),
\label{eq:SecondProximal_WPMTV}
\end{equation}
where $\sumBt=\sum_s \umB_s^k$ and $\uvv_k=\left(\sumBt\right)^{-1}\sum_s\left({\umB}_s^k \uvx_s^k - a_k\uvg_s^k \right)$. Since this paper mainly focuses on the TV regularizer, we replace $h(\uvx)$ with $\mathrm{TV}(\uvx)$.  \Cref{alg:IncrementalWPMs} summarizes the detailed steps of \ours{}.  For $k\leq S$, we set $\umB_s^k=\alpha_s\umI_N$. So computing $\uvx_k$ at step \ref{alg:IncrementalWPMs:firstk_end} of \Cref{alg:IncrementalWPMs} reduces to the proximal mapping, which can be efficiently solved using the FDPGM~\cite{beck2009fastTV}. For $k>S$, in general, $\umB_s^k\neq \umI_N$, since we will use second-order information. Therefore, it is crucial to efficiently compute a nontrivial WPM -- specifically step \ref{alg:IncrementalWPMs:endQN} of \Cref{alg:IncrementalWPMs} -- to further reduce the overall computational cost. \MajRev{In the} following, we propose an efficient approach to compute the related WPM for a special class of $\{\umB_s^k\}_{s,k}$. 

\subsection{Efficient Computation of the WPM}
\label{sec:IncrementalWPM:sub:eff:WPM}
Inspired by~\cite{beck2009fastTV}, we compute the WPM at step \ref{alg:IncrementalWPMs:endQN} of \Cref{alg:IncrementalWPMs} using its dual formulation, with computational complexity comparable to that of the proximal mapping when  $\{\umB_s^k\}_{s,k}$ shares the same structure as $\umW$ in \Cref{them:structuredWPM:evaluation}. Invoking \eqref{eq:B&Tiso} or \eqref{eq:B&Tl1}, we recast \eqref{eq:SecondProximal_WPMTV} as
\begin{equation}
\min\limits_{\uvx\in\mathcal C}\left( \max\limits_{\umP\in\mathcal P} \frac{1}{2}\|\uvx-\uvv_k\|_{\sumBt}^2 + a_k\Ksubset\lambda \uvd(\umP)^\Trans\uvx\right),
\label{eq:SecondProximal_NewTAPDual} 
\end{equation}
where $\mathcal{P}=\left\{\umP\in\mathbb{R}^{D\times N}:\left\{\left\|{\mathbf{p}}_{n}\right\|_2\leq1\right\}_{n=1}^N\right\}$ for the isotropic TV.\footnote{For the anisotropic TV, we have $\mathcal{P}=\left\{\umP\in\mathbb{R}^{D\times N}:\left\{\left\|{\mathbf{p}}_{n}\right\|_{\cb \infty}\leq 1\right\}_{n=1}^N\right\}$.} 
Reorganizing \eqref{eq:SecondProximal_NewTAPDual}, we obtain
\begin{equation}
\begin{array}{cl}
\min\limits_{\uvx\in\mathcal C} \max\limits_{\umP\in\mathcal P}&\left\|\uvx-\uvw_k(\umP)\right\|_{\sumBt}^2-\left\|\uvw_k(\umP)
\right\|_{\sumBt}^2,
\label{eq:SecondProximalMatrixnormV1}
\end{array}
\end{equation}
where $\uvw_k(\umP)=\uvv_k-a_k \Ksubset\lambda\left(\sumBt\right)^{-1}\uvd(\umP)$.
Since \eqref{eq:SecondProximalMatrixnormV1} is convex in $\uvx$ and concave in $\umP$,\footnote{\MajRev{By considering only $\umP$, the cost function in \eqref{eq:SecondProximalMatrixnormV1} 
can be written as $\langle \umB^k \uvx,\, \uvw_k(\umP) \rangle$. 
Since the algebraic dependence of $\uvw_k(\cdot)$ on $\umP$ is linear, the resulting cost 
function with respect to $\umP$ is linear, and therefore both convex and concave.}} we interchange the $\min$ and $\max$ \MajRev{using the minimax theorem \cite{boyd2004convex}} and then rewrite it as:
\begin{equation}
\begin{array}{cl}
 \max\limits_{\umP\in\mathcal P}\min\limits_{\uvx\in\mathcal C}&\left\|\uvx-\uvw_k(\umP)\right\|_{\sumBt}^2-\left\|\uvw_k(\umP)
\right\|_{\sumBt}^2.
\label{eq:SecondProximalMatrixnormV1:interchange} 
\end{array}
\end{equation}
Note that $\uvx$ only appears in the first term of \eqref{eq:SecondProximalMatrixnormV1:interchange}. So the optimal solution of $\uvx$ in \eqref{eq:SecondProximalMatrixnormV1:interchange} is 
\begin{equation}
	\mrm{prox}^{\sumBt}_{\iota_{\mathcal C}}\left(\uvw_k(\umP)\right).
	\label{eq:SecondProximalMatrixnorm:OptSolutionInnerPro}
\end{equation}

By substituting \eqref{eq:SecondProximalMatrixnorm:OptSolutionInnerPro} into \eqref{eq:SecondProximalMatrixnormV1:interchange}, we derive \eqref{eq:SecondProximalMatrixnormV2_New}, which depends only on $\umP$:
\begin{equation}
\label{eq:SecondProximalMatrixnormV2_New} 
 \umP^*=\arg\min\limits_{\umP\in\mathcal P} \MajRev{T(\umP),\,\,T(\umP)\triangleq}\left\|\uvw_k(\umP)\right\|^2_{\sumBt}-\|\uvw_k(\umP)-\mrm{prox}^{\sumBt}_{\iota_{\mathcal C}}\left(\uvw_k(\umP)\right)\|_{\sumBt}^2.
\end{equation}
\MajRev{By invoking \Cref{lemma:gradient:LipschitzConstant:dualPro}, which provides the gradient and 
Lipschitz constant of $T(\umP)$, we directly apply APM to solve 
\eqref{eq:SecondProximalMatrixnormV2_New}. After solving \eqref{eq:SecondProximalMatrixnormV2_New}, we get $\uvx_{k}=\mrm{prox}^{\sumBt}_{\iota_{\mathcal C}}\left(\uvw_k(\umP^*)\right).$} The proof of \Cref{lemma:gradient:LipschitzConstant:dualPro} is given in \Cref{app:proof:gradient:LipschitzConstant:dualPro}.
\begin{lemma}
	\label{lemma:gradient:LipschitzConstant:dualPro}
\MinRev{ 
For all $\umP \in \mathbb{R}^{D \times N}$, the gradient of $T$ in  \eqref{eq:SecondProximalMatrixnormV2_New}  with respect to $\umP$ is given by
\begin{equation}
	\label{eq:SecondProximalMatrixnormV2_New:gradient}
\vnabla T(\umP) =-2a_k\Ksubset\lambda
	\uvd^\Trans\Big(\mrm{prox}^{\sumBt}_{\iota_{\mathcal C}}\left(\uvw_k(\umP)\right)\Big),
\end{equation}
}
with Lipschitz constant \MajRev{$8Da_k^2\Ksubset^2
\lambda^2(\omega_{\mrm{min}}^k)^{-1}$}, where $\omega_{\mrm{min}}^k$ is the smallest eigenvalue of $\sumBt$ and $D$ is the dimension of the image.
\end{lemma}
\
\begin{remark}
	In view of \eqref{eq:SecondProximalMatrixnormV2_New:gradient}, computing~$\uvw_k(\umP)$ and $\mrm{prox}^{\sumBt}_{\iota_{\mathcal C}}\left(\uvw_k(\umP)\right)$ are the most computationally expensive parts.  However, by choosing $\umB_s^k$ to have the same structure as $\umW$ in \Cref{them:structuredWPM:evaluation}, we can compute $\uvw_k(\umP)$ and $\mrm{prox}^{\sumBt}_{\iota_{\mathcal C}}\left(\uvw_k(\umP)\right)$ efficiently, as discussed in \Cref{sec:proposed:sub:impDetails}. 
\end{remark}

\subsection{Setting \texorpdfstring{$\{\uvx_s^k,\,\uvg_s^k,\,\umB_s^k\}_{s,k}$}{TEXT}}
\label{sec:IncrementalWPM:sub:choiceIterGradHess}
In this section, we discuss the choice of $\{\uvx_s^k,\,\uvg_s^k,\,\umB_s^k\}_{s,k}$ when $k>S$ such that, at each iteration, we only compute $\vnabla F_s(\uvx)$ for one selected subset $s$. Choosing $\uvx_s^k =\uvx_{k-1}$, $\uvg_s^k=\vnabla F_s(\uvx_{k-1}),\,\text{\MajRev{for all}}~s$, and $\umB_1^k=\umB_2^k=\cdots=\umB_S^k$ at the \MajRev{$k$-th} iteration, we simply recover \oursF{}. \MajRev{In the following, we present two strategies for choosing $\{\uvx_s^k,\uvg_s^k\}_{s,k}$ when $k>S$, such that only one subset gradient is computed at each iteration.}

{\emph{Strategy I:}} At the \MajRev{$k$-th} iteration, we compute the gradient for \MajRev{the $s'$-th} subset such that $s'=\mathrm{mod}(k,S)$ and then set $\uvx_s^k=\uvx_{k-1}$ and $\uvg_s^k=\vnabla F_{s'}(\uvx_{k-1})$ \MinRev{for all $s$}.

{\emph{Strategy II:}} At the \MajRev{$k$-th} iteration, we uniformly sample a subset $s'$ among the $S$ subsets to compute $\vnabla F_{s'}(\uvx_{k-1})$ and then assign $\uvx_s^k=\uvx_{k-1}$ and $\uvg_s^k=\vnabla F_{s'}(\uvx_{k-1})$ \MinRev{for all $s$}.


%
%
%
\begin{algorithm}[t]        
\caption{SR1 estimation: $\umB_s^k$} 
\label{alg:WeightingSR1}                  
\begin{algorithmic}[1]
\REQUIRE $\uvs_s^k$, $\uvm_s^k$, $\gamma\in(0,1)$, $\alpha_s>0$, \MinRev{$\delta_1=10^{-8}$, $\delta_2=10^{-12}$,} \MajRev{and $0<\tau_{\max}<\infty$}
\lastcon $\umB_s^k,\,\tau_s^k,\,\uvu_s^k$
\STATE {$\tau_s^k \leftarrow \frac{\langle\uvm_s^k,\uvm_s^k\rangle}{\gamma\langle\uvs_s^k,\uvm_s^k\rangle}$} \label{alg:WeightingSR1:initialization} \\[5pt]
\STATE \MajRev{{$\tau_s^k\leftarrow \min(\tau_s^k,\tau_{\max})$}\label{alg:WeightingSR1:project:tau}}
\IF {\MajRev{$\tau_s^k\leq 0$}}
\STATE $\tau_s^k\leftarrow \alpha_s$
\STATE $\uvu_s^k\leftarrow \bm 0$
\ELSE
\IF{$\langle\uvm_s^k-\tau_s^k\, \uvs_s^k,\uvs_s^k \rangle\leq \MinRev{\delta_1}\|\uvs_s^k\|_2\|\uvm_s^k-\tau_s^k\, \uvs_s^k\|_2$}
\STATE $\uvu_s^k\leftarrow \bm 0$
\ELSE
\MajRev{
\WHILE{$\langle\uvm_s^k-\tau_s^k\, \uvs_s^k,\uvs_s^k \rangle\leq \MinRev{\delta_2} \|\uvs_s^k\|_2^2$} \label{alg:WeightingSR1:damp:tau}
\IF {$\tau_s^k<\MinRev{\delta_2}$}
\STATE $\tau_s^k\leftarrow \alpha_s$
\STATE $\uvu_s^k \leftarrow\bm 0$
\STATE break
\ELSE
\STATE $\tau_s^k\leftarrow \tau_s^k/2$
\STATE $\uvu_s^k \leftarrow \frac{\uvm_s^k-\tau_s^k\,\uvs_s^k}{\sqrt{\langle\uvm_s^k-\tau_s^k\, \uvs_s^k,\uvs_s^k\rangle}}$
\ENDIF
\ENDWHILE
}
\ENDIF
\ENDIF
\STATE $\umB_{s,k}^0 \leftarrow \tau_s^k\umI_N$
\STATE {$\umB_s^k \leftarrow \umB_{s,k}^0 +\uvu_s^k(\uvu_s^k)^\Trans$}
\end{algorithmic}
\end{algorithm}

Along with $\{\uvx_s^k,\uvg_s^k\}_{s,k}$, we define another pair,  $\{\bar{\uvx}_s^k,\bar{\uvg}_s^k\}_{s,k}$. For $k\leq S$, we set $\{\bar{\uvx}_s^k,\bar{\uvg}_s^k\}_{s,k}$ to be identical to $\{\uvx_s^k,\uvg_s^k\}_{s,k}$. For $k>S$, at the \MajRev{$k$-th} iteration, we assign $\bar{\uvx}_{s'}^k=\uvx_{k-1}$ and $\bar{\uvg}_{s'}^k=\vnabla F_{s'}(\uvx_{k-1})$ for the chosen \MajRev{$s'$-th} subset. For $s\neq s'$, we set $\bar{\uvx}_s^k=\bar{\uvx}_s^{\,k-1}$ and $\bar{\uvg}_s^k=\bar{\uvg}_s^{\,k-1}$. Denote by $\uvs_s^k = \bar{\uvx}_s^k-\bar{\uvx}_s^{\,i_k^*}$ and $\uvm_s^k =\bar{\uvg}_s^k-\bar{\uvg}_s^{\,j_k^*}$ where $i_k^* = \max_{i<k} \{i\,|\,\bar{\uvx}_s^i\neq \bar{\uvx}_s^k\}$ and $j_k^* = \max_{j<k} \{j\,|\,\bar{\uvg}_s^j\neq \bar{\uvg}_s^k\}$. \MajRev{In our construction of $\bar{\uvx}_s^k$ and $\bar{\uvg}_s^k$, we always have $i_k^* = j_k^*$ under both strategies for selecting $s'$.} With $\{\uvm_s^k,\uvs_s^k\}_{s,k}$, we deploy the symmetric-rank-$1$ (SR1) method~\cite{nocedal2006numerical} to estimate $\{\umB_s^k\}_{s,k}$ such that they hold the same structure as $\umW$ in~\Cref{them:structuredWPM:evaluation}. \Cref{alg:WeightingSR1} summarizes the steps of estimating $\umB_s^k$. The parameter $\alpha_s>0$ acts as the Lipschitz constant of $F_s$. \MajRev{We use $\tau_{\max}$ to prevent $\tau_s^k $ from becoming infinite. The condition at step~\ref{alg:WeightingSR1:damp:tau} in~\Cref{them:structuredWPM:evaluation} enforces a uniform curvature that enables an upper bound on the estimated Hessian (cf.~\Cref{App:sec:Proof:lemma:boundedHessSR1}).} The classical SR1 method uses the previously estimated Hessian matrix with a \mbox{rank-1} correction. Here, by contrast, we enforce that~$\umB_s^k = \tau_s^k \umI_N + \uvu_s^k(\uvu_s^k)^{\Trans}$ to save memory usage. Note that $\tau_s^k$ is a scalar.

\subsection{Implementation Details}
\label{sec:proposed:sub:impDetails}
This part discusses how to compute  $\mrm{prox}^{\sumBt}_{\iota_{\mathcal C}}\left(\uvw_k(\umP)\right)$ and $\uvw_k(\umP)$ efficiently. To compute $\uvw_k(\umP)$, we have to invert~$\sumBt$.
Since $\{\umB_s^k\}_s$ is a set of rank-$1$ corrected matrices, we have that
\begin{equation*}
    \sumBt={\cb \bm \Sigma}_k+\umU_k\umU_k^\Trans,
\end{equation*}
where $\umU_k=[\uvu_1^k\;\uvu_2^k \cdots\uvu_{\Ksubset}^k]\in \mathbb R^{N\times \Ksubset}$ and ${\cb \bm \Sigma}_k= \tau_k^*\umI_{N}$ with \mbox{$ \tau_k^*=\sum_s \tau_s^k$}.
Using the Woodbury matrix identity, we derive
$$\left(\sumBt\right)^{-1}=(\tau_k^*)^{-1}\umI_{N}- (\tau_k^*)^{-2}\umU_k\left(\umI_{\Ksubset}+\frac{\umU_k^{\Trans}\umU_k}{ \tau_k^*}\right)^{-1}\umU_k^{\Trans},$$ so that $\left(\sumBt\right)^{-1}$ is easily applied.

Using the structure of $\umB_s^k$ and~\Cref{them:structuredWPM:evaluation}, we can compute~$\mrm{prox}^{\sumBt}_{\iota_{\mathcal C}}\left(\uvx\right)$ efficiently. Note that computing~$\mrm{prox}^{\sumBt}_{\iota_{\mathcal C}}\left(\uvx\right)$ requires to solve a nonsmooth and nonlinear equation, \emph{i.e.}~\eqref{eq:WPM:structure:equiv:nonlinearSystems}. Compared to the size of image~$N$, $\Ksubset$ is small. Therefore, we adopt the semi-smooth Newton method~\cite{qi1999survey}. Let 
$
\mrm{dom}_\varphi=\left\{\bm\beta\in\mathbb R^\Ksubset~|~\varphi(\bm\beta) ~\text{is differentiable at} ~\bm\beta\right\}.
$
Then, the generalized Jacobian of $\varphi$ at $\bm\beta$ is defined by 
$
\partial \varphi(\bm\beta) = \mrm{conv}\,\partial_{\mrm{dom}_\varphi} \varphi(\bm\beta), 
$
where 
\MinRev{$\partial_{\mrm{dom}_\varphi} \varphi(\bm\beta)=\left\{\lim_ {\substack{\bm\beta_i\rightarrow \bm\beta\\ \bm\beta_i\in {\mrm{dom}_\varphi}}}\umJ \varphi(\bm\beta_i)\right\}$}, conv denotes the convex hull, and \MinRev{$\umJ \varphi(\bm\beta_i)$ denotes the corresponding Jacobian matrix}. With these definitions, at the \MajRev{$i$-th} iteration, the semi-smooth Newton method \cite{qi1999survey} updates $\bm\beta_i$ through
$
\bm\beta_i = \bm\beta_{i-1}-\umH_{i-1}^{-1}\varphi(\bm\beta_{i-1}),
$
where $\umH_{i-1}\in \partial \varphi(\bm\beta_{i-1})$. \Cref{alg:semismoothNewton} presents the implementation details of the semi-smooth Newton method.
In our experiments, 
\Cref{alg:semismoothNewton} reaches a small error tolerance~(\textit{e.g.}, $10^{-6}$) after few iterations. 

\begin{algorithm}[t]    
\caption{Semi-smooth Newton to solve $\varphi(\bm\beta)=\bm 0$} 
\label{alg:semismoothNewton}                  
\begin{algorithmic}[1]
\REQUIRE ~\\
Initial guess $\bm\beta_0$, tolerance $\epsilon$ (\textit{e.g.}, $10^{-6}$), maximal number of iterations Max\_Iter
\lastcon $\bm\beta^*$
\STATE $i \leftarrow 1$
\FORALL{$i \leq \mrm{Max\_Iter}$}
\IF {$\|\varphi(\bm\beta_{i -1})\|_2\leq \epsilon$}
\RETURN
\ELSE
\STATE Pick $\umH_{i -1}\in \partial \varphi(\bm\beta_{i -1})$
\STATE 
$
\bm\beta_i  \leftarrow \bm\beta_{i -1}-\umH_{i -1}^{-1}\varphi(\bm\beta_{i-1})
$
\ENDIF
\STATE $i \leftarrow i +1$
\ENDFOR
\RETURN $\bm\beta^* \leftarrow \bm\beta_i$
\end{algorithmic}
\end{algorithm}
%
\subsection{Discussion}
\label{sec:IncrementalWPM:sub:discussion}
Note that, in \Cref{alg:IncrementalWPMs}, the dominant computation of \ours{} at the \MajRev{$k$-th} iteration is the computation of $\vnabla F_{s'}(\uvx_{k-1})$ for the selected \MajRev{$s'$-th} subset and the related WPM. By using \Cref{alg:WeightingSR1}, the estimated Hessian shares the same structure as $\umW$ in \Cref{them:structuredWPM:evaluation}, enabling efficient solutions to the WPM as discussed in \Cref{sec:IncrementalWPM:sub:eff:WPM,sec:proposed:sub:impDetails}. Thus, computing $\vnabla F_{s'}(\uvx_{k-1})$ dominates the computational complexity in practice.

Next, we discuss the memory usage of \ours{}. Compared to ASPM, \ours{} requires storing $\{\uvx_s^k,\uvg_s^k,\umB_s^k\}_{s,k}$ and $\{\bar{\uvx}_s^k,\bar{\uvg}_s^k\}_{s,k}$. However, regardless of the strategy used to choose $s'$, $\{\uvx_s^k,\uvg_s^k\}_{s,k}$ can always be retrieved from $\{\bar{\uvx}_s^k,\bar{\uvg}_s^k\}_{s,k}$, meaning only $\{\bar{\uvx}_s^k,\bar{\uvg}_s^k\}_{s,k}$ needs to be stored. According to \Cref{alg:WeightingSR1}, it is sufficient to store $\{\uvu_s^k\}_s$ instead of $\{\umB_s^k\}_s$, which requires storing only $S$ additional images. Notice that we use $\{\bar{\uvx}_s^k,\bar{\uvg}_s^k\}_{s,k}$ to estimate the Hessian matrices which involves only the current $\bar{\uvx}_s^k,\bar{\uvg}_s^k$ and its most recent previous one for each subset. Once the Hessian is estimated, we only need to save the most recent $\bar{\uvx}_s^k,\bar{\uvg}_s^k$ for each subset, requiring an additional $2S$ images. In total, \ours{} requires storing an additional $3S$ images, which scales linearly with $S$ and is independent of the number of iterations.

\section{Convergence Analysis}
\label{sec:ConvAnalysis}
This section presents the convergence analysis of \ours{} without assuming convexity of $\{f_l\}_{l=1}^L$. Our analysis encompasses strategies I and II for selecting $\{\uvx_s^k,\uvg_s^k\}_{s,k}$. Before presenting our main convergence results, we introduce three assumptions used in our analysis, as stated in~\Cref{assum:LipFunGradHessian,assum:PLCond,assum:expectGrad}.

\begin{assumption}
\label{assum:LipFunGradHessian} \MinRev{Assume that $\Phi$ is bounded below, the regularizer term $h$ is convex but possibly nonsmooth, and  $F_s$ is twice continuously differentiable for all $s$}. Furthermore, we assume that $F_s$ satisfies the following properties for \MajRev{all $s$}.
\begin{itemize}
	\item[(a)] \label{assum:LipFunGradHessian:itemLipFun} \MajRev{$F_s$} is $\xi$-Lipschitz continuous, \MajRev{\emph{i.e.}, there exists a constant $\xi >0$ such that for all $\uvx_1,\,\uvx_2\in\mathbb R^N$},
	\begin{equation}
	\label{eq:funLip:l}
		\|F_s(\uvx_1)-F_s(\uvx_2)\|\leq \xi \, \|\uvx_1-\uvx_2\|.
	\end{equation}
	\item[(b)] \label{assum:LipFunGradHessian:itemLipGrad} The gradient of \MajRev{$F_s$} is $\kappa$-Lipschitz continuous, \MajRev{\emph{i.e.}, there exists a constant $\kappa>0$ such that for all $\uvx_1,\,\uvx_2\in\mathbb R^N$},
	\begin{equation}
	\label{eq:gradLip:l}
		\|\vnabla F_s(\uvx_1)-\vnabla F_s(\uvx_2)\|\leq \kappa \, \|\uvx_1-\uvx_2\|.
	\end{equation}
\end{itemize}
\end{assumption}
A direct conclusion of \Cref{assum:LipFunGradHessian} (a) and (b) is
	\begin{equation}
	\label{eq:BoundedHessiangrad:l}
		\|\vnabla F_s(\uvx)\|\leq \xi \text{~~and~~}\|\vnabla ^2 F_s(\uvx)\|\leq \kappa,\text{~for all~} s.
	\end{equation}
Since $\digamma(\uvx)=\frac{1}{S}\sum_s F_s(\uvx)$, it is easy to verify that \MajRev{$\digamma$ and $\vnabla \digamma$} are $\xi$- and  $\kappa$-Lipschitz continuous such that we have 
	\begin{equation}
	\label{eq:BoundedHessiangrad:Sum}
		\|\vnabla \digamma(\uvx)\|\leq \xi \text{~and~}\|\vnabla ^2 \digamma(\uvx)\|\leq \kappa.
\end{equation}
\begin{assumption}[\cite{karimi2016linear,wang2019stochastic}]
	\label{assum:PLCond}\MajRev{Assume that $\bar{h}$ is lower-bounded over set $\mathcal C$.
	Define 
\begin{equation}
\label{eq:def:D_h}
	\mathcal D_{\bar h}^{\mathcal C}(\uvx,\uvg,\umB,a)=-\frac{2}{a} \min_{\uvx'\in \mathcal C} \mathcal B_{\bar h}(\uvx',\uvx,\uvg,\umB,a),
\end{equation}
where $\mathcal B_{\bar h}(\uvx',\uvx,\uvg,\umB,a)=\langle \uvg,\uvx'-\uvx \rangle +\frac{1}{2a}\|\uvx'-\uvx\|_{\umB}^2+{\bar h}(\uvx')-{\bar h}(\uvx)$ and $a>0$.} Then \MajRev{$\Phi$ satisfies the Polyak-\L{}ojasiewicz inequality, in the sense   there exists a constant $\varrho>0$ such that}
\begin{equation}
\label{eq:PLIneq}
	\mathcal D_{\bar h}^{\mathcal C}(\uvx,\vnabla \digamma(\uvx),\umI_N,a)\geq 2\varrho (\Phi(\uvx)-\Phi^*), \forall \uvx \in \mathcal C,
\end{equation}
where $\Phi^*$ is \MajRev{the minimal value of \eqref{eq:defMainPro} over $\mathcal C$}.
\end{assumption}
Similar to \cite{karimi2016linear,wang2019stochastic}, we use \Cref{assum:PLCond} in our analysis, as it encompasses certain nonconvex settings. In general, the nonlinear Lippmann--Schwinger data-fidelity terms 
$\{f_l\}_{l=1}^L$ in ODT lead to an optimization problem that may contain 
spurious local minima and therefore do not satisfy the 
Polyak--\L{}ojasiewicz (PL) condition. Nevertheless, in our analysis we 
assume that the overall objective $\Phi$ satisfies the PL condition, 
which is a stronger geometric assumption and may still hold even when 
the individual terms $\{f_l\}_{l=1}^L$ do not satisfy it. Moreover, our 
experiments demonstrate that the reconstruction converges to 
high-quality solutions, suggesting that these spurious local minima do 
not pose a significant issue in our experimental settings. \MinRev{A more 
realistic and weaker assumption would be to adopt the 
Kurdyka--\L{}ojasiewicz (KL) property~\cite{attouch2010proximal}, which 
applies to a broader class of nonconvex problems. Nevertheless, the PL 
condition guarantees convergence to a global minimizer, whereas the general KL framework typically ensures convergence only to a critical point. In fact, the PL condition can be  viewed as a stronger special case of the KL property in the smooth setting. Extending our analysis to the broader KL framework is therefore an interesting direction for future work.} 
\MinRev{
\begin{assumption}
	\label{assum:expectGrad}
	The variable $s'$ is sampled uniformly from the $S$ subsets at each iteration. Then for all $k$, we assume
\begin{equation}
\label{eq:ExpGrad}
\mathbb{E}\left[\nabla F_{s'}(\mathbf{x}_k)\mid \mathbf{x}_k\right]
= \nabla \digamma(\mathbf{x}_k),~\text{and}~ \mathbb{E}\left[\|\nabla F_{s'}(\mathbf{x}_k)- \nabla \digamma(\mathbf{x}_k)\|^2\mid \mathbf{x}_k \right] \leq \sigma^2.
\end{equation}
\end{assumption}
}
Next, we present two lemmas to simplify the presentation of our convergence analysis.
\begin{lemma}
	\label{lemma:OptWPMIneq}
	If $\uvx_k$ is obtained by \eqref{eq:SecondProximal_TV} and $\stepsize>0$, then we have
	\begin{equation}
		\label{eq:WPMOptCond}
					\left \langle \stepsize\sum_s \uvg_s^k, \iterDiffk \right \rangle \leq \left \langle \sum_s \umB_s^k(\uvx_k-\uvx_s^k), -\iterDiffk \right\rangle  +\stepsize S \left(\bar{h}(\uvx_{k-1})-\bar{h}(\uvx_k) \right).
	\end{equation}
	where $\iterDiffk=\uvx_k-\uvx_{k-1}$.
\end{lemma}

\begin{lemma}
	\label{lemma:boundedHessSR1}
	Under \Cref{assum:LipFunGradHessian}, if $\{\umB_s^k\}_{s,k}$ are generated by \Cref{alg:WeightingSR1}, then there exist two positive constants $\underline{\kappa},\,\overline{\kappa}$ such that 
	$\underline{\kappa}\,\umI_N \preceq \umB_s^k \preceq\overline{\kappa}\,\umI_N$ \MajRev{for all $s,k$}.
	\end{lemma}
	The proofs of \Cref{lemma:OptWPMIneq,lemma:boundedHessSR1} are presented in \Cref{App:sec:Proof:lemma:OptWPMIneq,App:sec:Proof:lemma:boundedHessSR1}.
A direct conclusion from \Cref{lemma:boundedHessSR1} is 
$$\underline{\kappa}\,\umI_N \preceq \frac{1}{S}\umB^k \preceq  \overline{\kappa}\,\umI_N,$$
since $\umB^k = \sum_s \umB_s^k$.

\begin{theorem}
\label{them:ConvResults} 
Denote by $e^*=\frac{2(S-1)^2}{S^2}\xi^2$ and $c_k=1+\frac{\stepsize (1+\kappa)-\underline \kappa}{2\underline \kappa- (1+\kappa)\stepsize}$. Then we can establish the following convergence results for \ours{}:
	\begin{itemize}
		\item[(1)] Under \Cref{assum:LipFunGradHessian}, $\stepsize \in(0, \frac{2\underline{\kappa}}{1+\kappa})$, and running \ours{} $K$ iterations with strategy I, we have 
		$$
		\MajRev{\Delta_K^*} \leq \frac{\Phi(\uvx_0)-\Phi^*+Ke^*}{\sum_{k=1}^K \frac{2\underline{\kappa}-(1+\kappa)\stepsize}{2\stepsize}},		
		$$
where \MajRev{$\Delta_K^*=\min_{k\leq K} \|\iterDiffk\|_2^2$} with $\iterDiffk=\uvx_k-\uvx_{k-1}$ and $\Phi^*$ is \MinRev{the} minimal value of $\Phi$ over $\mathcal C$.

\item[(2)] Under \Cref{assum:LipFunGradHessian,assum:expectGrad}, $\stepsize \in(0, \frac{2\underline{\kappa}}{1+\kappa})$, and running \ours{} $K$ iterations with strategy II, we have 
\MinRev{
$$
\Delta^*_{K,\mathbb E}\leq \frac{\mathbb E\left[ \Phi(\uvx_0)-\Phi^*\right]+K\sigma^2/2}{\sum_{k=1}^K \frac{2\underline{\kappa}-(1+\kappa)\stepsize}{2\stepsize}},
$$
}
where \MinRev{$\Delta^*_{K,\mathbb E}=\min_{k\leq K} \mathbb E[\|\iterDiffk\|_2^2\mid \uvx_{k-1}]$}.

\item[(3)] Under \Cref{assum:LipFunGradHessian,assum:PLCond,assum:expectGrad}, 
let $\stepsize \in \big(\frac{\underline{\kappa}}{1+\kappa}, \frac{2\underline{\kappa}}{1+\kappa}\big)$. 
Running \ours{} for $K$ iterations with strategy II and sampling $k^*$ according to 
$\mathrm{Prob}\{k^*=k\}=\frac{\stepsize}{2\overline{\kappa}c_k K}$ yields
\MinRev{
\[
\mathbb{E}\!\left[\Phi(\uvx_{k^*})-\Phi^*\right]
\le 
\frac{\mathbb{E}\!\left[\Phi(\uvx_0)-\Phi^*\right]}{2\varrho K}
+ \frac{\sigma^2}{4\varrho}.
\]
}
	\end{itemize}
\end{theorem}
The proof of \Cref{them:ConvResults} is summarized in \Cref{App:sec:Proof:them:ConvResults}.  From \Cref{them:ConvResults} (1) and (2), if we choose the stepsize $a_k$ such that $\sum_{k=1}^K \frac{2\underline{\kappa}-(1+\kappa)\stepsize}{2\stepsize}\rightarrow \infty$ and $\frac{K}{\sum_{k=1}^K \frac{2\underline{\kappa}-(1+\kappa)\stepsize}{2\stepsize}} \leq \mathrm{Constant}$ as $K \rightarrow \infty$, then \MajRev{$\Delta_K^*$} and \MinRev{$\Delta^*_{K,\mathbb E}$} approach zero plus a constant. Therefore, \MajRev{$\Delta_K^*$} and \MinRev{$\Delta^*_{K,\mathbb E}$} are upper bounded, which implies the stability of the algorithm. Note that a simply constant stepsize policy can satisfy the requirement. \MajRev{In our numerical experiments, we observed that $e_k^2$ tended to zero, implying the value of $\Delta_K^*$ is always well bounded.} \MinRev{The bound of $\Delta^*_{K,\mathbb{E}}$ depends on the variance of the gradient estimation.} Note that \Cref{them:ConvResults} (3) demonstrates that, when running \ours{} under strategy~II for a sufficiently large predetermined number of iterations, the expected function value converges to the optimal value \MinRev{up to a term proportional to the variance of the gradient estimator}. In our experiments, we simply choose the last iterate as the output.

\section{Numerical Experiments}
\label{sec:NumericalExperiments}
ODT is a noninvasive and label-free technique that allows one to obtain a refractive-index~(RI) map of the sample~\cite{wolf1969three}. 
In ODT, the sample is sequentially illuminated from different angles. The outgoing complex wave field of each illumination is recorded through a digital-holography microscope \cite{kim2010principles}.
Finally, the RI map is recovered by solving an inverse-scattering problem.
\Cref{fig:ODT} displays a scheme of the acquisition principle. ODT is ideal for studying the performance of \ours{}.
Indeed, inverse-scattering problems are composite minimization problems that can be either convex or nonconvex, depending on the choice of the physical model--whether a linear model, such as the \MajRev{Born equation~\cite{born1920volumen}}, or a nonlinear model, such as the \MajRev{Lippmann-Schwinger (LippS) equation~\cite{schwinger1951green}}. The nonlinear model is accurate for strongly scattering samples. For completeness, we provide a brief introduction to the continuous model of ODT in the supplementary material. Following~\cite{pham2020three}, we discretize the continuous ODT model using finite differences. When LippS is used as the forward model, executing $\mathcal H_l(\uvx)$ once and computing the gradient of \MajRev{$f_l$ at point $\uvx$} require solving the associated LippS equation once and twice, respectively. \MajRev{Here, $\mathcal H_l$ models the physical mapping from $\uvx$ to the measurements~$\uvy_l$.}  In the following experiments, we use the BiCGSTAB algorithm \cite{van1992bi} to solve the discretized LippS equations. See~\cite{pham2020three} and the references therein for further details about ODT. 
\input{figs/scheme3DODT/scheme3DODT}

We studied the performance of \ours{} for reconstructing the RI map using simulated and real data, incorporating an isotropic TV regularizer and a nonnegativity constraint. Specifically,
we solved
\begin{equation}
\uvx^\ast\in\arg\min_{\uvx\in\mathbb R^N_+} ~~\Phi(\uvx).\label{eq:ODTMinProb}
\end{equation}
For both simulated and real data, we compared \ours{} with \oursF{}, ASPM, and the variance-reduced based stochastic quasi-Newton proximal method (\wangs{})~\cite{wang2019stochastic}. For completeness, we present the details of ASPM in the supplementary material. We use \ours{}-I and \ours{}-II to denote \ours{} with strategy I and strategy II, respectively. Note that Wang \emph{et al.} \cite{wang2019stochastic} only considered $h=\|\cdot\|_1$, but, for the sake of fairness, we considered a constrained TV regularizer instead. We then deployed our method in \Cref{sec:IncrementalWPM:sub:eff:WPM} to efficiently solve the related WPM. In consequence, we only compared the Hessian estimation approach of \cite{wang2019stochastic} with ours.

For the simulated data, we first recovered RI maps using the first-order Born approximation~\cite{chen1998validity}.
The corresponding physical model is then linear, which makes~\eqref{eq:ODTMinProb} a convex optimization problem. \MajRev{However, we will see that the linear Born model is deficient for strongly scattering samples, illustrating the importance of using more accurate models.} We then recovered RI maps using LippS on simulated and real data, which means that~\eqref{eq:ODTMinProb} now corresponds to a nonconvex optimization problem. In our experiments, we demonstrated the advantages of using LippS for strongly scattering samples. \MajRev{For clarity of presentation, the experiments based on the linear Born model were included in the supplementary material.}

All experiments were run on a workstation with 3.3GHz AMD EPYC 7402 and NVIDIA GeForce RTX 3090. For a fair comparison, all reconstruction algorithms were run on the same GPU platform.
Our implementation is based on the GlobalBioIm library~\cite{soubies2018pocket} and is publicly available at \url{https://github.com/hongtao-argmin/MiniBatch-QNP-NonlinearReco}. 

\input{figs/LS_Simulated.tex}
\subsection{Simulated Data}
\noindent{\emph {Simulation Settings:}}
We mainly used two phantoms as the ground-truth volumes: one weakly scattering sample (maximal RI $1.363$) and one strongly scattering sample (maximal RI $1.43$). The strongly scattering sample with RI~$\eta_\mathrm{strong}(\vr)$ (\MajRev{where $\vr$ denotes the three-dimensional spatial coordinate}) was immersed in a medium with RI $\eta_\mathrm{m}=1.333$ and was illuminated by plane waves of wavelength $\lambda_\mathrm{in}=406$nm.
The domain~$\Omega$ is a cube of edge length $3.2\mu$m and fully contains the sample.
To obtain the complex-valued measurements, we used the Lipps on a grid with a resolution of $50$nm, yielding a total of $64^3$ voxels to discretize $\Omega$ and $\eta_\mathrm{strong}$.
The sample was probed by $L=60$ tilted plane waves~$u_\mathrm{in}^l(\uvr)=\mathrm{exp}(\mathrm{j}\langle\uvk_l^{\mrm{in}},\uvr\rangle)$ for $l=1,\ldots,L$.
The wavevectors~$\{\uvk_l^\mathrm{in}\in\mathbb{R}^3\}_{l=1}^L$ were embedded in a cone with half-angle $42^{\circ}$~(see \Cref{fig:ODT}). We then obtained a total of $60\times512^2$ measurements.
Without the regularization, this setting makes our inverse problem ill-posed. The measurements are lacking information on the frequency along the optical axis, \textit{i.e.}, the so-called missing cone problem.
For the weakly scattering sample, we proceeded similarly ({same medium and wavelength}) but simulate with the linear Born model to generate the measurements instead. 

\subsubsection{Nonlinear LippS Model--Strongly Scattering Sample}
\label{Sec:Exp:subsec:2Dnonlinear}
In this part, we studied the performance of \ours{}-I/II to recover the RI maps using the LippS model. The regularization $\lambda$ and the stepsize were set as~$1/64^3$ and $1/20$, respectively.
A total of $100$ iterations were performed for all competing methods.

\Cref{fig:LS:CostSNR:Simulated} shows the evolution of the full cost and SNR versus the iterations and wall time for all competing methods. Although \wangs{} converged faster than ASPM and \ours{}-I/II in terms of iterations at the beginning, it became slower than \ours{}-I/II at the later iterations. Moreover, \ours{}-I/II required fewer iterations than ASPM to reach a lower full cost. \oursF{} is the fastest algorithm in terms of iterations, but it loses its advantage in running time due to the need to compute the full gradient at each iteration. From the perspective of running time, \ours{}-I is the fastest algorithm, demonstrating the superiority of our method. Moreover, we also observed that \ours{}-I converged faster than \ours{}-II, \MajRev{which is consistent with the observation made in the linear Born model results (cf. the supplementary material)}. Compared to the reconstruction under the Born model (cf. the experiments in the supplementary material), \Cref{fig:LS:CostSNR:Simulated} shows that the LippS model required almost three times more wall time than the linear Born model to perform the same number of iterations for reconstruction. This observation highlights the importance of reducing the number of gradient computations at each iteration in the LippS model, illustrating the merits of \ours{}.

The second row of \Cref{fig:LS:CostSNR:Simulated} shows ASPM achieved \MajRev{its highest SNR} at the $100$th iteration ($776.6$ seconds) while \ours{}-I only required $38$ iterations ($264.3$ seconds) to achieve a similar SNR, demonstrating the superiority of our approach and the benefits of utilizing second-order information. \Cref{fig:LS:Ortho:Simulated} displays the orthoviews of the RI maps recovered by ASPM, and \ours{}-I/II. We saw that all these results for a nonconvex composite-optimization problem corroborated the observations we got on the convex counterpart.

 \input{figs/ChoiceLGamma.tex}

\subsubsection{On the Choice of \texorpdfstring{$\Ksubset$}{TEXT} and \texorpdfstring{$\gamma$}{TEXT}}
In this part, we investigated the effect of $\Ksubset$ and \MinRev{$\gamma$ (which is used to scale the initial Hessian matrix; cf. \Cref{alg:WeightingSR1})} on the convergence behavior of \ours{}. \Cref{fig:LSSimLSReco:Subset} presents the full cost versus iterations and wall time for \ours{}-I/II across different values of $\Ksubset$. Clearly, we saw the convergence of \ours{}-I/II were influenced by $\Ksubset$. In particular, a smaller $\Ksubset$ led to faster convergence in terms of iterations because it resulted in a more accurate gradient and Hessian estimation. Indeed, $\Ksubset=1$ is the fastest one in terms of iterations since it used the \emph{full} gradient at each iteration. However, \ours{} with $\Ksubset=1$ required more computation at each iteration and thus lost its efficiency in terms of wall time.  Indeed, \Cref{fig:LSSimLSReco:Subset:Time:I,fig:LSSimLSReco:Subset:Time:II} show that \ours{} with $\Ksubset=1$ converged slower than $\Ksubset>1$ in terms of wall time. Moreover, \Cref{fig:LSSimLSReco:Subset:Time:I} indicates \ours{}-I with $\Ksubset=12$ converged faster than the other methods in terms of wall time, while \Cref{fig:LSSimLSReco:Subset:Time:II} shows \ours{}-II with $\Ksubset=6$ was the fastest one in terms of wall time. However, the difference in wall time was not significant for $S\geq 4$, therefore we simply set $S = 4$ in our experiments.  \Cref{fig:LSSimLSReco:Gamma} describes the effect of $\gamma$ on the convergence behavior of \ours{}-I/II. \MajRev{For \ours{}-I, we observed that using $\gamma<1$ led to faster convergence during the early iterations. Eventually, different $\gamma$ reached similar final costs, indicating that \ours{}-I is not particularly sensitive to the choice of $\gamma\leq 1$.} For \ours{}-II, we observed that $\gamma = 1$ converged faster than the others at the early iterations, but it eventually yielded a slightly higher final full cost. \MajRev{However, we observed that \ours{}-II is not sensitive to $\gamma < 1$ because they converged to similar final costs}. Since values of $\gamma < 1$ consistently produced slightly better results, we chose $\gamma = 0.8$ in our experiments.

\input{figs/LSReal.tex}

\subsubsection{Convergence Validation}  
\label{sec:NumericalExp:Simulated:ConvVal}
\begin{figure}[t]
\centering
\subfigure[\ours{}-I]{\includegraphics[width=0.45\textwidth]{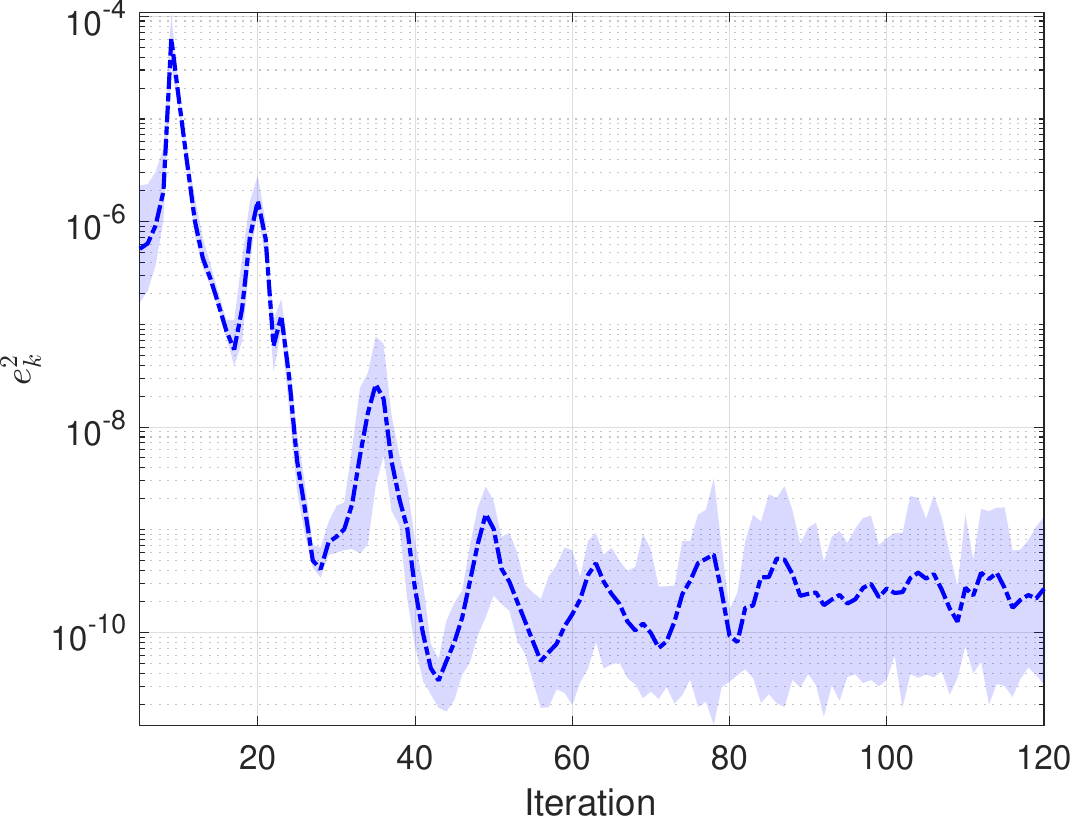}
\label{fig:ConvVal:biaserr:I}}
\subfigure[\ours{}-I]{\includegraphics[width=0.45\textwidth]{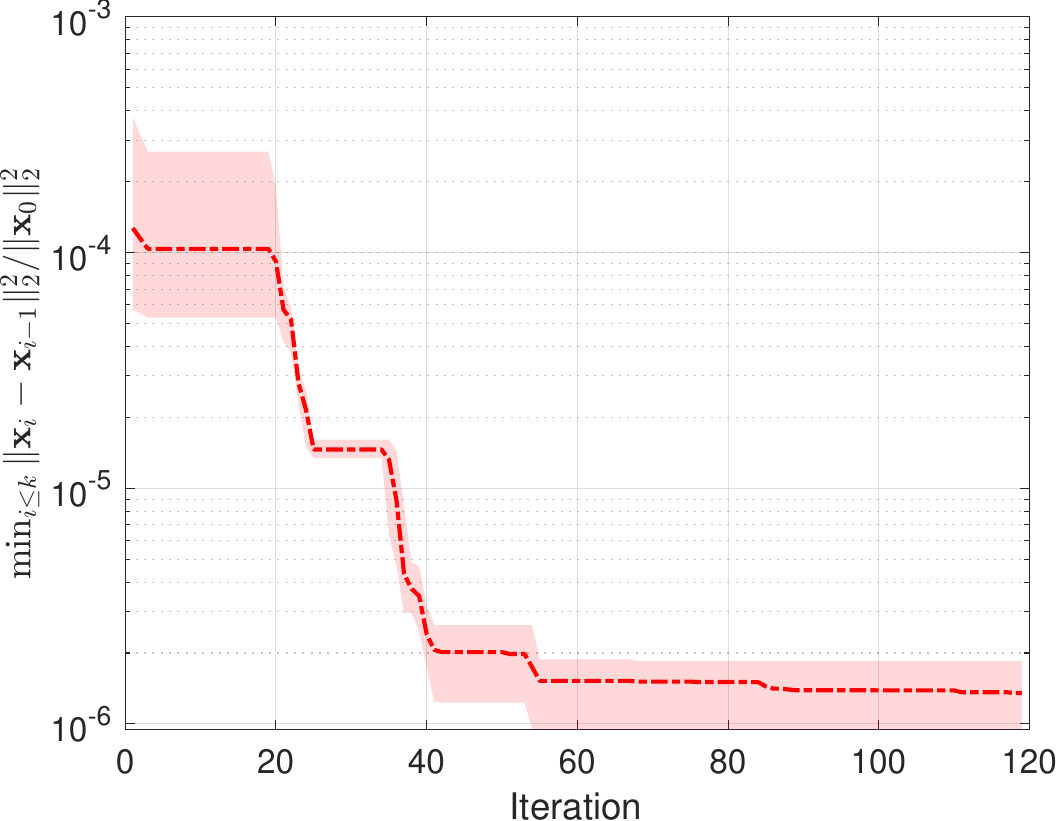}
\label{fig:ConvVal:iterdiff:I}}

\subfigure[\ours{}-II]{\includegraphics[width=0.45\textwidth]{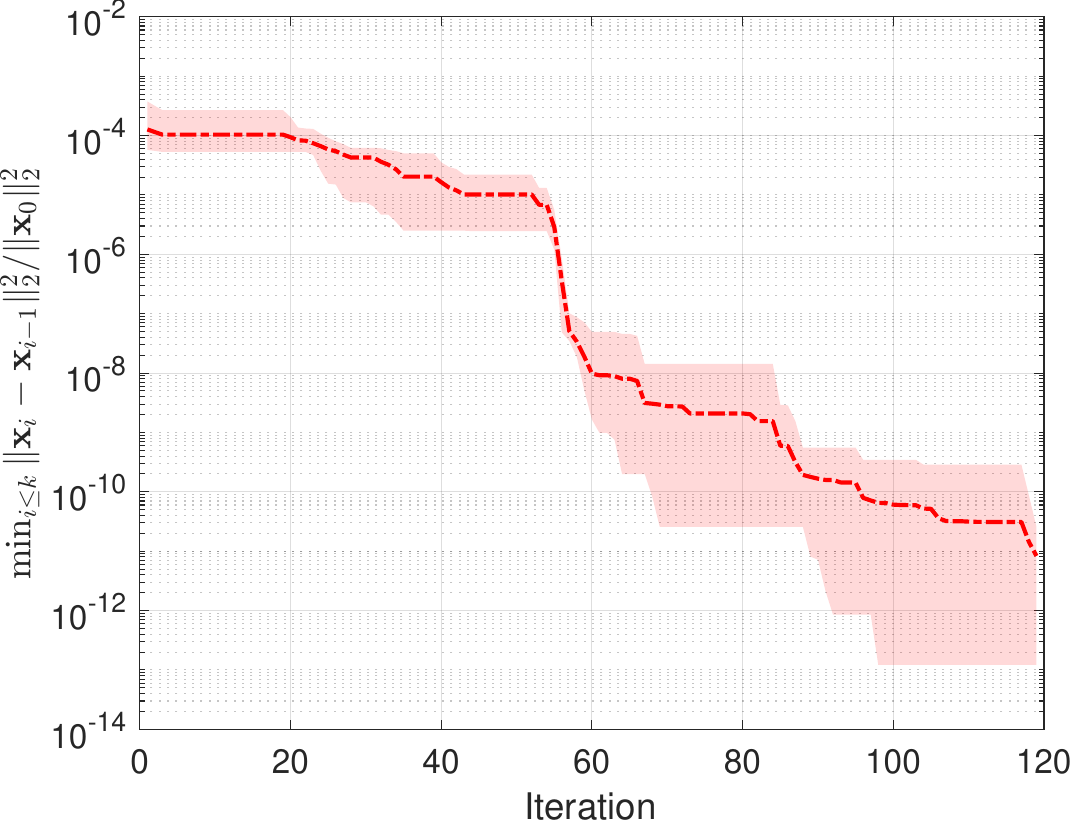}
\label{fig:ConvVal:iterdiff:II}}
\subfigure[\ours{}-I/II]{\includegraphics[width=0.45\textwidth]{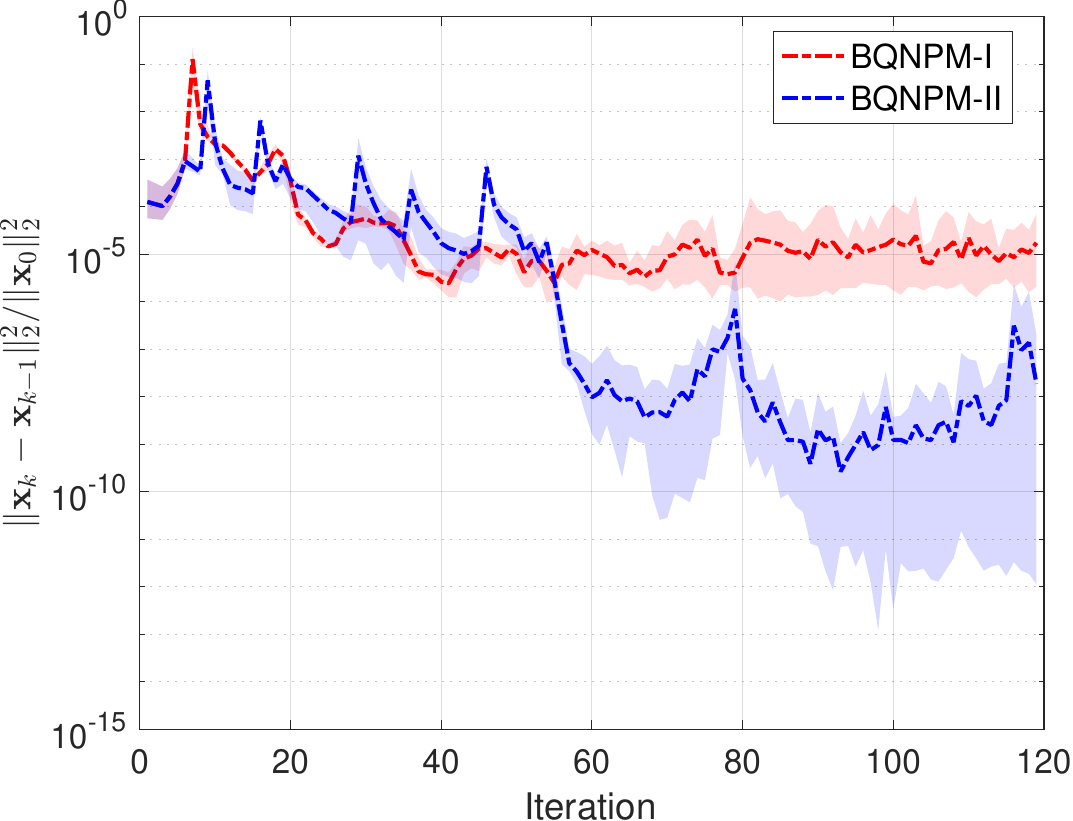}\label{fig:ConvVal:delta:IandII}}
\caption{(a): Averaged $e^2_k$ values versus iterations for \ours{}-I on the reconstruction of strongly scattering samples; \MajRev{(b), (c), and (d): Averaged $\min_{i\leq k}\|\uvx_{i}-\uvx_{i-1}\|_2^2/\|\uvx_0\|_2^2$ and $\|\uvx_k-\uvx_{k-1}\|_2^2/\|\uvx_0\|_2^2$ values versus iterations for \ours{}-I/II on the reconstruction of strongly scattering samples}. The shaded region of each curve represents the range of
the evaluation criterion across samples with different refractive-index values.}
\label{fig:ConvVal:IandII}	
\end{figure}
In this section, we empirically validate the theoretical results presented in \Cref{sec:ConvAnalysis}. The algorithmic setting used here was identical to \Cref{fig:LS:CostSNR:Simulated}. We reconstructed various samples with different maximal RI values ranging from $1.41$ to $1.53$ in intervals of $0.01$. \Cref{fig:ConvVal:biaserr:I} shows the average squared error $e_k^2 = \left\| \frac{1}{S} \sum_{s \neq s'} \left( \nabla F_s(\uvx_{k-1}) - \uvg_s^k \right) \right\|_2^2$ versus iteration for \ours{}-I. This quantity eventually tends to zero, indicating that $e^*$ in the first part of \Cref{them:ConvResults} is negligible in practice. Indeed, \Cref{fig:ConvVal:iterdiff:I} depicts the averaged $\min_{i\leq k}\|\uvx_i-\uvx_{i-1}\|_2^2/\|\uvx_0\|_2^2$ values versus iteration for \ours{}-I, showing a reduction by an order of two as the iterations progressed. This demonstrated that the values were well-controlled, thereby validating our results in \Cref{them:ConvResults}. \Cref{fig:ConvVal:iterdiff:II} presents the result of \ours{}-II, which clearly shows that $\min_{i\leq k}\|\uvx_{i}-\uvx_{i-1}\|_2^2/\|\uvx_0\|_2^2$ tended to zero as the iterations progressed. \MajRev{\Cref{fig:ConvVal:delta:IandII} shows the values of $\|\uvx_k-\uvx_{k-1}\|_2^2/\|\uvx_0\|_2^2$, which indicate that the difference between consecutive iterates becomes smaller as the iterations progress.}

\subsection{Real Data}
Finally, we assessed the performance of \ours{}-I/II on real data of a yeast cell immersed in water~($\eta_m=1.338$). The sample was illuminated by $60$ incident plane waves ($\lambda=532$nm) embedded in a cone of illumination whose half-angle is $35$\degree~\cite{ayoub2019method,lim2019high}.
The discretized volume has a total of $96^3$ voxels of size $99^3\mathrm{nm}^3$. See \cite{ayoub2019method, lim2019high,pham2020three} for the detailed description of the acquisition settings.
The stepsize and regularization parameter were set as $0.01$ and $2/96^3$, respectively.
Moreover, $60\times 150^2$ measurements were used for the reconstruction and~$60$ iterations were run for ASPM, \wangs{}, \ours{}-I/II, \MajRev{and \oursF}.
\Cref{fig:LS:CostSNR:real} displays the evolution of the full cost for both algorithms. Similar to the simulated cases, we saw that \MajRev{\oursF{} was the fastest algorithm in terms of iterations. In this case, we also observed that \oursF{} was faster than ASPM, \wangs{}, and \ours{}-II in terms of wall time. Overall, \ours{}-I was the fastest algorithm in terms of wall time.} ASPM achieved a lower cost than \ours{}-II at later iterations, while \ours{}-I remained the fastest. \Cref{fig:LS:recon:real} presents the orthoviews of the RI maps obtained at the $10$th, $30$th, $40$th, and $50$th iteration for each method. We saw that \ours{}-I recovered a qualitatively good RI map in $30$ iterations~($659.05$ seconds), while ASPM achieved a similar quality only after $50$ iterations~($1344$ seconds). The reconstructed images of the same yeast cell presented in~\cite{lim2019high} were visually similar to those obtained by our method.

\section{Conclusion}
\label{sec:conclusionAndfuture}
We propose a \increm{} quasi-Newton proximal method (\ours{}) for solving constrained total variation-based nonlinear image reconstruction problems. The computational cost of \ours{} is independent of the number of measurements, making it well-suited for composite minimization problems involving large sets of measurements. Additionally, our method avoids the need to compute the full gradient of the data-fidelity term, thereby eliminating the costly traversal of the entire measurements. This represents a significant departure from existing stochastic proximal quasi-Newton methods.  

We have also introduced an efficient approach to compute the weighted proximal mapping required by \ours{}. Furthermore, we have provided a convergence analysis of \ours{} in the nonconvex setting. Our numerical experiments on 3D optical diffraction tomography, conducted with both simulated and real data, demonstrate that \ours{} converges more rapidly than a stochastic accelerated first-order proximal method, both in terms of iterations and wall time. These results highlight how the proposed method can substantially reduce the computational cost of solving composite inverse problems.

\MajRev{Although \ours{} is mainly illustrated for ODT reconstruction with a TV regularizer, it can be naturally extended to other regularizers such as wavelets~\cite{chambolle1998nonlinear} and the Hessian–Schatten norm~\cite{lefkimmiatis2013hessian}. Moreover, the weighted proximal mappings for wavelets and for the Hessian–Schatten norm can be computed in a manner similar to that of the TV regularizer. Note that our convergence analysis is developed in a nonconvex setting, making it particularly interesting to extend the framework to inverse problems with massive measurements by incorporating learned priors~\cite{cohen2021has} with convergence guarantees. Furthermore, our experiments on the linear Born model also demonstrate the fast convergence of \ours{} in the convex setting, suggesting that it is also natural to extend \ours{} to positron emission tomography (PET), computed tomography (CT), and multi-coil magnetic resonance imaging (MRI) applications.
}

%




\appendix

\section{Derivation from \texorpdfstring{\eqref{eq:SecondProximal_TV}}{TEXT} to \texorpdfstring{\eqref{eq:SecondProximal_WPMTV}}{TEXT}}
\label{app:deduce:aylorAppSubFuncToSecondProximal_TV}
\MajRev{
By substituting \eqref{eq:TaylorAppSubFunc} into \eqref{eq:SecondProximal_TV} for all $s$ and ignoring the constant terms, we obtain 
$$
\begin{array}{rrl}
	\uvx_k=&\arg\min_{\uvx\in\mathcal C}&\left\{ \frac{1}{S}\sum_s \MinRev{\left(\frac{1}{2a_k}\uvx^\Trans\umB_s^k\uvx +\langle \uvg_s^k-\frac{1}{a_k}\umB_s^k\uvx_s^k,\uvx\rangle\right)}\right\} + \bar h(\uvx)\\[6pt]
	&& =\frac{1}{a_kS} \left\{\frac{1}{2}\uvx^\Trans \umB^k\uvx -\langle \sum_s (\umB_s^k\uvx_s^k-a_k\uvg_s^k),\uvx\rangle \right\} + \bar h(\uvx)\\[6pt]
	&&=\frac{1}{2a_kS}\|\uvx-\uvv_k\|_{\umB^k}^2 + \bar h(\uvx) - \frac{1}{2a_kS}\|\uvv_k\|_{\umB^k}^2.
\end{array}
$$
By ignoring the constant term $\frac{1}{2a_kS}\|\uvv_k\|_{\umB^k}^2$ and multiplying the remaining cost function by $a_kS$, we reach \eqref{eq:SecondProximal_WPMTV}.
}

\section{Proof of \texorpdfstring{\Cref{lemma:gradient:LipschitzConstant:dualPro}}{TEXT}}
\label{app:proof:gradient:LipschitzConstant:dualPro}
Denote by $\uvh_{\mathcal C}(\uvx)=\uvx-\mrm{prox}^{\sumBt}_{\iota_{\mathcal C}}\left(\uvx\right)$ and $T(\umP)= \left(-\left\|\uvh_{\mathcal C}\left(\uvw_k(\umP)\right)\right\|_{\sumBt}^2+\left\|\uvw_k(\umP)\right\|^2_{\sumBt}\right)$. \MinRev{The proof consists of differentiating the function $T$, which can be rewritten as
$$
\umP \mapsto T(\umP)
= \|\uvw_k(\umP)\|_{\umB^k}^2
 - \|\uvw_k(\umP) - \mrm{prox}_{\iota_{\mathcal C}}^{\umB^k} ( \uvw_k(\umP))\|_{\umB^k}^2
= \|\uvw_k(\umP)\|_{\umB^k}^2
 - 2\,\mathcal{M}_{\iota_{\mathcal C}}^{\umB^k}(\uvw_k(\umP)).
$$
Let $\umJ {\uvw_k}(\umP) = -a_k S \lambda (\umB^k)^{-1} \uvd$ denote the Jacobian matrix of $\uvw_k$ at point $\umP$. Then the chain rule thus naturally leads to
$$
\begin{array}{rcl}
   {\pmb{\nabla}} T(\umP) &=& [\umJ\uvw_k(\umP)]^\Trans \Big( 2\umB^k \uvw_k(\umP)-2\pmb{\nabla} \mathcal M_{\iota_{\mathcal C}}^{\umB^k} (\uvw_k(\umP))  \Big) \\
   &=&-a_k\Ksubset\lambda\mathcal \uvd^\Trans {(\umB^k)}^{-1} \Big(2\umB^k \uvw_k(\umP)- 2\umB^k\big(\uvw_k(\umP)-\mrm{prox}_{\iota_{\mathcal C}}^{\umB^k} ( \uvw_k(\umP))\big) \Big)	\\
   &=&-2a_k\Ksubset\lambda\mathcal \uvd^\Trans \Big(\mrm{prox}_{\iota_{\mathcal C}}^{\umB^k} ( \uvw_k(\umP))\Big)
\end{array}
$$}

Now, we compute the Lipschitz constant of $\nabla T$. 
For every two pairs of $\umP_1$ and $\umP_2$, we have that
\begin{equation}
\begin{array}{rcl}
	\|{\pmb{\nabla}} T(\umP_1)-{\pmb{\nabla}} T(\umP_2)\|&=&\left\|2a_k\Ksubset\lambda\mathcal \uvd^\Trans\left(\mrm{prox}^{\sumBt}_{\iota_{\mathcal C}}\left(\uvw_k(\umP_1)\right)
	-\mrm{prox}^{\sumBt}_{\iota_{\mathcal C}}\left(\uvw_k(\umP_2)\right)
\right)\right\|\\[6pt]
&\leq& 2a_k\Ksubset\lambda\sqrt{(\omega_{\mrm{min}}^k)^{-1}}\|\uvd^\Trans\|\\[6pt]
&&~~\left\|\mrm{prox}^{\sumBt}_{\iota_{\mathcal C}}\left(\uvw_k(\umP_1)\right)
	-\mrm{prox}^{\sumBt}_{\iota_{\mathcal C}}\left(\uvw_k(\umP_2)\right)
\right\|_{\sumBt}\\[6pt]
&\overset{(*)}{\leq}& 2a_k\Ksubset\lambda\sqrt{(\omega_{\mrm{min}}^k)^{-1}}\|\uvd^\Trans\|\cdot\|Sa_k\lambda\left(\sumBt\right)^{-1}\left(\uvd(\umP_1)-\uvd(\umP_2)\right)\|_{\sumBt}\\[6pt]
&\overset{(**)}{\leq}&  2a_k^2S^2\lambda^2(\omega_{\mrm{min}}^k)^{-1}\|\mathcal \uvd^{\Trans}\|\cdot \|\uvd\|\cdot\|\umP_1-\umP_2\|\\[6pt]
&=& 2a_k^2S^2\lambda^2(\omega_{\mrm{min}}^k)^{-1}\|\uvd\|^2\cdot\|\umP_1-\umP_2\|,
\end{array}
\end{equation}
where the transition $(*)$ follows from the non-expansiveness property of the WPM \cite{lee2014proximal}, while $(**)$ follows from the fact that $\|\left(\sumBt\right)^{-1}\uvx\|_{\sumBt} \leq \sqrt{(\omega_{\mrm{min}}^k)^{-1}}\|\uvx\|$, with $\omega_{\mrm{min}}^k$ the smallest eigenvalue of $\sumBt$. Following the proof in \cite[Lemma 4.2]{beck2009fastTV}, we have $\|\uvd\|=\sqrt{4D}$, where $D$ is the dimension of the image $\uvx$. The Lipschitz constant of $T(\umP)$ is then \MajRev{$8Da_k^2S^2
\lambda^2(\omega_{\mrm{min}}^k)^{-1}$}. We note that $\omega_{\mrm{min}}^k$ can be obtained through the power method since $\left(\sumBt\right)^{-1}\uvx$ can be applied cheaply in our case. Alternatively, one could adopt a backtracking strategy to set the stepsize at each iteration \cite{beck2017first}.


\section{Proof of \texorpdfstring{\Cref{lemma:OptWPMIneq}}{TEXT}}
\label{App:sec:Proof:lemma:OptWPMIneq}
Notice that $\uvx_k$ is the optimal solution in  \eqref{eq:SecondProximal_TV} \MinRev{and the cost function in  \eqref{eq:SecondProximal_TV} is convex. There exists $\uvg_{\bar h}\in\partial {\bar h}(\uvx_k)$  such that, for any $\uvx'\in\mathcal C$, we have
$$
\left\langle \frac{1}{\Ksubset}\sum_s \vnabla \bar F_s^k(\uvx_k)+\uvg_{\bar h},\uvx'-\uvx_k\right \rangle \geq 0,
$$
where $\partial \bar h (\uvx_k)$ refers to the subgradient of $\bar h$ at point $\uvx_k$. Letting $\uvx'=\uvx_{k-1}$ and using the fact that $\bar F_s^k(\uvx_k)=F_s(\uvx_s^k)+\left\langle \uvg_s^k, \uvx_k-\uvx_s^k\right\rangle +\frac{1}{2\stepsize}\|\uvx_k-\uvx_s^k\|^2_{\umB_s^k}$ for all $s$ (cf.~\eqref{eq:TaylorAppSubFunc}), we obtain 
$$
\begin{array}{rcl}
	\left\langle \frac{1}{S}\sum_s \uvg_s^k,\iterDiffk\right\rangle &\leq& \Big \langle \left[\frac{1}{\stepsize\,S}\sum_s\umB_s^k(\uvx_k-\uvx_s^k)\right] +\uvg_{\bar h}, -\iterDiffk \Big \rangle\\[6pt]
	&\leq & \Big \langle \left[\frac{1}{\stepsize\,S}\sum_s\umB_s^k(\uvx_k-\uvx_s^k)\right], -\iterDiffk \Big \rangle + \bar h(\uvx_{k-1}) - \bar h(\uvx_k),
\end{array}
$$
where $\iterDiffk=\uvx_k-\uvx_{k-1}$. The second inequality follows from the fact that $\bar h$ is convex and $\uvg_{\bar h}\in\partial \bar{h}(\uvx_k)$.} Multiplying both sides by $\stepsize\,S$, we get the desired result.

\section{Proof of \texorpdfstring{\Cref{lemma:boundedHessSR1}}{TEXT}}
\label{App:sec:Proof:lemma:boundedHessSR1}
\MajRev{\MinRev{In Algorithm~4.2, we enforce the sequence $(\tau_s^k)_k$ to be positively lower-bounded ($\tau_s^k \ge \min(\delta_2/2 , \alpha_s)$ for all $k$ at every output), which implies that we can derive $\umB_{s,k}^0 \succeq \underline{\kappa} \umI_N$ with $\underline{\kappa} > 0$.} Since $\umB_s^k$ is either $\tau_s^k\umI_N$ or $\tau_s^k\umI_N+\uvu_s^k{(\uvu_s^k)}^\Trans$, we obtain $\umB_s^k \succeq \underline{\kappa} \umI_N$. Here we use the fact that the eigenvalues of $\uvu_s^k (\uvu_s^k)^\Trans$ are always nonnegative. Now we show that $\umB_s^k$ is also upper-bounded. \MinRev{If $\uvu_s^k=\bm 0$, we have $\umB_s^k=\tau_s^k \umI_N$. From step \ref{alg:WeightingSR1:project:tau} in \Cref{alg:WeightingSR1}, we have $\tau_s^k\leq \tau_{\max}<\infty$.  In the following, we consider the branch  $\uvu_s^k\neq \bm 0$.}

Note that $\uvm_s^k=\bar \uvg_s^k-\bar \uvg_s^{j_k^*}=\int_0^1 \frac{d \vnabla F_s(\bar \uvx_s^{i_k^*}+t\,\uvs_s^k)}{dt} dt=\int_0^1 \vnabla^2 F_s(\bar \uvx_s^{i_k^*}+t\,\uvs_s^k) \uvs_s^k dt$ since $\bar \uvg_s^k = \vnabla F_s(\bar \uvx_s^k) $, $\bar\uvg_s^{j_k^*} = \vnabla F_s(\bar \uvx_s^{i_k^*})$, and $\uvs_s^k=\bar \uvx_s^k-\bar \uvx_s^{i_k^*}$. With these, we can derive 
\begin{equation}
\label{eq:AppProof:m_ks_kRela}
	\uvm_s^k=\overline {\vnabla^2 F_s^k} \uvs_s^k,\quad \overline {\vnabla^2 F_s^k}=\int_0^1 \vnabla^2 F_s (\bar \uvx_s^{i_k^*}+t\,\uvs_s^k)dt,
\end{equation}
from the mean-value theorem~\cite{comenetz2002calculus}. \MinRev{By using the condition at step \ref{alg:WeightingSR1:damp:tau} in \Cref{alg:WeightingSR1}, we get 
\begin{equation}
\label{eq:App:sec:Proof:lemma:boundedHessSR1:Bound}
	\langle\uvm_s^k-\tau_s^k\, \uvs_s^k,\uvs_s^k \rangle \geq \delta_2\|\uvs_s^k\|^2_2>0.
\end{equation}
Using  $\uvm_s^k=\overline {\vnabla^2 F_s^k} \uvs_s^k$, the definition of $\uvu_s^k$, and \eqref{eq:App:sec:Proof:lemma:boundedHessSR1:Bound}, we can upper-bound $\|\umB_s^k\|$ as
$$
\begin{array}{rcl}
	\|\umB_s^k\|\leq \tau_s^k+(\uvu_s^k)^\Trans \uvu_s^k&=&\tau_s^k+\frac{(\uvs_s^k)^\Trans (\overline {\vnabla^2 F_s^k}-\tau_s^k \umI_N)^2 \uvs_s^k }{\langle\uvm_s^k-\tau_s^k\, \uvs_s^k,\uvs_s^k \rangle}\\[6pt]
	&\leq& \tau_s^k + \frac{\|\overline {\vnabla^2 F_s^k}-\tau_s^k \umI_N\|_2^2}{\delta_2}\\[6pt]
	&\leq& \tau_{\max}+ (\kappa+\tau_{\max})^2/\delta_2 < \infty,
\end{array}
$$
where the second inequality is due to \Cref{assum:LipFunGradHessian:itemLipFun} (b). In summary, we have $\overline{\kappa}=\max\left(\alpha^*_s,\tau_{\max}+ (\kappa+\tau_{\max})^2\right)/\delta_2<\infty$ where $\alpha^*_s=\max_s \alpha_s$.
}
}


\section{Proof of \texorpdfstring{\Cref{them:ConvResults}}{TEXT}}
\label{App:sec:Proof:them:ConvResults}
	By applying \Cref{assum:LipFunGradHessian:itemLipFun} (b),  we have the following descent inequality \cite[Lemma 5.7]{beck2017first}
\begin{equation}
	\label{eq:Proof:LipIneq}
			\digamma(\uvx_k)\leq \digamma(\uvx_{k-1})+\langle \vnabla \digamma (\uvx_{k-1}),\iterDiffk\rangle+ \frac{\kappa}{2}\|\iterDiffk\|_2^2,
\end{equation}
where $\iterDiffk=\uvx_k-\uvx_{k-1}$ and $\digamma(\cdot)=\frac{1}{S}\sum_s F_s(\cdot)$.
Invoking \Cref{lemma:OptWPMIneq} with \eqref{eq:Proof:LipIneq}, we reach
\begin{equation*}
		\begin{array}{rcl}
			\digamma(\uvx_k)&\leq& \digamma(\uvx_{k-1}) + \frac{\kappa}{2}\|\iterDiffk\|_2^2+ \langle \frac{1}{S}\sum_s \uvg_s^k,\iterDiffk\rangle+\langle \frac{1}{S}\sum_s (\vnabla  F_s(\uvx_{k-1})-\uvg_s^k),\iterDiffk\rangle \\[5pt]
			&\leq &  \digamma(\uvx_{k-1})+ \frac{\kappa}{2}\|\iterDiffk\|_2^2 -\langle \frac{1}{\stepsize S} \sum_s\umB_s^k(\uvx_k-\uvx_s^k),\iterDiffk \rangle + \bar h(\uvx_{k-1})-\bar h(\uvx_k) \\[5pt]
			&&~ +\langle \frac{1}{S}\sum_s (\vnabla  F_s(\uvx_{k-1})-\uvg_s^k),\iterDiffk\rangle
	\end{array}
\end{equation*}
Moving $\bar h(\uvx_k)$ to the left hand side and using the fact that $\uvx_s^k=\uvx_{k-1}$ \MajRev{for all $s$}, we get 
\begin{equation}
\label{eq:Proof:Conv:MainIneq}
\begin{array}{rcl}
\Phi(\uvx_k)&\leq & \Phi(\uvx_{k-1})+\frac{\kappa}{2}\|\iterDiffk\|_2^2 -\langle \frac{1}{\stepsize S} \sum_s\umB_s^k \iterDiffk,\iterDiffk \rangle \\
&&~~+\langle \frac{1}{S}\sum_s (\vnabla F_s(\uvx_{k-1})-\uvg_s^k),\iterDiffk\rangle\\[5pt]
			&\leq &\Phi(\uvx_{k-1})-(\frac{2\underline{\kappa}-\stepsize \kappa}{2\stepsize})\|\iterDiffk\|_2^2+\langle \frac{1}{S}\sum_s (\vnabla F_s(\uvx_{k-1})-\uvg_s^k),\iterDiffk\rangle,
\end{array}	
\end{equation}
where the second inequality comes from \Cref{lemma:boundedHessSR1}. 

Now we derive the convergence result of \ours{} for the selection of $\uvg_s^k$ with strategies I or II. For strategy I, substituting  \MajRev{$\uvg_s^k=\nabla F_{s'}(\uvx_{k-1})$} into \eqref{eq:Proof:Conv:MainIneq}, we obtain
\begin{equation}
\label{eq:Proof:Conv:StrI}
\begin{array}{rcl}
\Phi(\uvx_k)  &\leq &\Phi(\uvx_{k-1})-(\frac{2\underline{\kappa}-\stepsize \kappa}{2\stepsize})\|\iterDiffk\|_2^2 +\langle \frac{1}{S}\sum_{s\neq s'} (\vnabla  F_s(\uvx_{k-1})-\uvg_s^k),\iterDiffk\rangle\\[5pt]
			&\leq & \Phi(\uvx_{k-1})-(\frac{2\underline{\kappa}-(1+\kappa)\stepsize}{2\stepsize})\|\iterDiffk\|_2^2 +\frac{1}{2}\underbrace{\|\frac{1}{S}\sum_{s\neq s'} (\vnabla F_s(\uvx_{k-1})-\uvg_s^k)\|_2^2}_{e_k^2},
\end{array}	
\end{equation}
where  the second inequality comes from \MajRev{Cauchy-Schwarz inequality and} $ab \leq \frac{a^2+b^2}{2}$. By   reorganizing \eqref{eq:Proof:Conv:StrI}, we get
\begin{equation}
	\label{eq:Proof:Conv:StrI:Cons}
	\frac{2\underline{\kappa}-(1+\kappa)\stepsize}{2\stepsize}\|\iterDiffk\|_2^2\leq \Phi(\uvx_{k-1})-\Phi(\uvx_k)+\frac{1}{2}e_k^2
\end{equation}
Letting $0<\stepsize < \frac{2\underline{\kappa}}{1+\kappa}$ and 
summing up \eqref{eq:Proof:Conv:StrI:Cons} from $k=1$ to $K$, we obtain
\begin{equation}
\label{eq:Proof:Str:delta:summation}
	\sum_{k=1}^K \frac{2\underline{\kappa}-(1+\kappa)\stepsize}{2\stepsize} \|\iterDiffk\|_2^2\leq \Phi(\uvx_0)-\Phi(\uvx_K)+\frac{1}{2}Ke_K^*\leq \Phi(\uvx_0)-\Phi^*+\frac{1}{2}Ke_K^*,
\end{equation}
\MajRev{where $\Phi^*$ represents the minimal value of $\Phi$ over $\mathcal C$ and $e_K^*=\max_{k\leq K} e_k^2$. Note that $\Phi$ is assumed to be lower bounded}. Dividing $\sum_{k=1}^K \frac{2\underline{\kappa}-(1+\kappa)\stepsize}{2\stepsize}$ to both sides of \eqref{eq:Proof:Str:delta:summation}, we obtain
\begin{equation}
\label{eq:Proof:UB:StrI:e_K}	
\MajRev{\Delta_K^*} \leq \frac{\Phi(\uvx_0)-\Phi^*+\frac{1}{2}Ke_K^*}{\sum_{k=1}^K \frac{2\underline{\kappa}-(1+\kappa)\stepsize}{2\stepsize}},
\end{equation}
\MajRev{where $\Delta_K^*=\min_{k\leq K} \|\iterDiffk\|_2^2$}.
By using a constant stepsize policy, we have 
\begin{equation}
\label{eq:Proof:UB:StrI:e_K:Const}	
\MajRev{\Delta_K^*} \leq \frac{2a^*(\Phi(\uvx_0)-\Phi^*)}{K(2\underline{\kappa}-(1+\kappa)a^*)}+\frac{a^*e_K^*}{2\underline{\kappa}-(1+\kappa)a^*},
\end{equation}
where $a^*$ denotes the constant stepsize. In the following, we theoretically establish an upper bound for $e_K^*$. Notice that
$$
|e_k| \leq   \frac{1}{S}\sum_{s\neq s'} \|\vnabla  F_s(\uvx_{k-1})-\uvg_s^k\|\leq\frac{1}{S}\sum_{s\neq s'} (\|\vnabla  F_s(\uvx_{k-1})\|+\|\uvg_s^k\|)\leq \frac{2(S-1)}{S}\xi.
$$
Clearly, we have $e_K^*\leq \frac{4(S-1)^2}{S^2}\xi^2$. Substituting this bound into \eqref{eq:Proof:UB:StrI:e_K}, we get \MinRev{(1)}. 

For strategy II, we uniformly sample $s'$ such that~\eqref{eq:ExpGrad} is satisfied.  \MinRev{Using the same inequality as in the second inequality of \eqref{eq:Proof:Conv:StrI}, we obtain
\begin{equation}
\label{eq:Proof:StrII:VarBound}
	\langle \frac{1}{S}\sum_s (\vnabla  F_s(\uvx_{k-1})-\uvg_{s'}^k),\iterDiffk \big\rangle\leq \frac{1}{2}\|\frac{1}{S}\sum_s (\vnabla  F_s(\uvx_{k-1})-\uvg_{s'}^k)\|_2^2+\frac{1}{2}\|\iterDiffk\|_2^2.
\end{equation}
Taking conditional expectation \MinRev{relative to the variable $\uvx_{k-1}$} for both sides of \eqref{eq:Proof:Conv:MainIneq}, using  \eqref{eq:Proof:StrII:VarBound}, \Cref{assum:expectGrad},  and letting $\stepsize< \frac{2\underline{\kappa}}{1+\kappa}$, we obtain
\begin{equation}
\label{eq:Proof:StrII:bound:iterdiff}
	\left(\frac{2\underline{\kappa}-\stepsize(1+\kappa)}{2\stepsize}\right) \mathbb E \left[\|\iterDiffk\|_2^2 \mid \uvx_{k-1}\right] \leq \mathbb E\left[ \Phi(\uvx_{k-1})-\Phi(\uvx_k) \mid \uvx_{k-1}\right]+\sigma^2/2.
\end{equation}
Summing up \eqref{eq:Proof:StrII:bound:iterdiff} from $k=1$ to $K$, we obtain
$$
\MajRev{\Delta^*_{K,\mathbb E}}\leq \frac{\mathbb E\left[ \Phi(\uvx_0)-\Phi^*\right]+K\sigma^2/2}{\sum_{k=1}^K \frac{2\underline{\kappa}-\stepsize (1+\kappa)}{2\stepsize}},
$$  
where \MinRev{$\Delta^*_{K,\mathbb E}=\min_{k\leq K} \mathbb E[\|\iterDiffk\|_2^2\mid \uvx_{k-1}]$}.}

If \MajRev{$\Phi$ satisfies the Polyak-\L{}ojasiewicz inequality described in  \Cref{assum:PLCond}}, we can further establish the convergence rate of the function values under strategy II. \MajRev{Using the fact that $\uvg_s^k=\uvg_{s'}^k$ for all $s$ and \MinRev{$\vnabla \digamma =\frac{1}{S}\sum_s \vnabla F_s$}, the definition of $\mathcal D_{\bar h}^{\mathcal C}(\uvx_{k-1}, \uvg_{s}^k,\umB^k,\stepsize)$ (cf. \eqref{eq:def:D_h}), the smoothness inequality (cf.  \eqref{eq:Proof:LipIneq}), and $\uvx_k$ is the optimal solution of~\eqref{eq:SecondProximal_TV}, we have
\begin{equation}
	\label{eq:Proof:LipIneqPL}
	\begin{array}{rcl}
			\digamma(\uvx_k) &\leq& \digamma(\uvx_{k-1})+\langle \vnabla \digamma (\uvx_{k-1}),\iterDiffk\rangle+ \frac{\kappa}{2}\|\iterDiffk\|_2^2\\[5pt]
&=& \underbrace{\langle \frac{1}{S}\sum_s \uvg_{s}^k,\iterDiffk \rangle +\frac{1}{2\stepsize}\|\iterDiffk\|_{\umB^k}^2+\bar h(\uvx_k)-\bar h(\uvx_{k-1})}_{\mathcal B_{\bar h}(\uvx_k,\uvx_{k-1},\frac{1}{S}\sum_s \uvg_{s}^k,\umB^k,\stepsize)}\\[8pt]
&&+\digamma(\uvx_{k-1})+\langle \frac{1}{S}\sum_s (\vnabla  F_s(\uvx_{k-1})-\uvg_{s}^k),\iterDiffk\rangle\\[7pt]
&&+\frac{\kappa}{2}\|\iterDiffk\|_2^2-\frac{1}{2\stepsize}\|\iterDiffk\|_{\umB^k}^2+\bar h(\uvx_{k-1})-\bar h(\uvx_k)\\[5pt]
&=& -\frac{\stepsize}{2} \mathcal D_{\bar h}^{\mathcal C}(\uvx_{k-1},\uvg_{s}^k,\umB^k,\stepsize)+\digamma(\uvx_{k-1})+\langle \frac{1}{S}\sum_s (\vnabla  F_s(\uvx_{k-1})-\uvg_{s}^k),\iterDiffk\rangle \\[5pt]
			&&+\frac{\kappa}{2}\|\iterDiffk\|_2^2-\frac{1}{2\stepsize}\|\iterDiffk\|_{\umB^k}^2+\bar h(\uvx_{k-1})-\bar h(\uvx_k).
	\end{array}
\end{equation}
} By reorganizing \eqref{eq:Proof:LipIneqPL} \MajRev{and invoking \Cref{lemma:boundedHessSR1}}, we get
\begin{equation}
\label{eq:Proof:PLIneqStrII}
\begin{array}{rcl}
		\frac{\stepsize}{2}\mathcal D_{\bar h}^{\mathcal C}(\uvx_{k-1},\uvg_{s'}^k,\umB^k,\stepsize) &\leq & \left\langle \frac{1}{S}\sum_s (\vnabla F_s(\uvx_{k-1})-\uvg_{s'}^k),\iterDiffk \right \rangle  +\frac{\stepsize \kappa-\underline \kappa}{2\stepsize}\|\iterDiffk\|_2^2\\[5pt]
		&&~~+\Phi(\uvx_{k-1})-\Phi(\uvx_k).
\end{array}
\end{equation}
\MajRev{Note that $\uvg_s^k = \uvg_{s'}^k$ for all $s$, and we use $\uvg_{s'}^k$ here since $s'$ denotes the subset selected at iteration~$k$.} Taking the conditional expectation \MinRev{relative to the variable $\uvx_{k-1}$} on both sides and letting \MinRev{$\frac{\underline{\kappa}}{1+\kappa}<\stepsize<\frac{2\underline{\kappa}}{1+\kappa}$, we get
\begin{equation}
\label{eq:Proof:ExpPLStrII}
	\begin{array}{rcl}
	\mathbb E\left[\frac{\stepsize}{2}\mathcal D_{\bar h}^{\mathcal C}(\uvx_{k-1},\uvg_{s'}^k,\umB^k,\stepsize) \mid \uvx_{k-1} \right] &\leq &\mathbb E\left[\Phi(\uvx_{k-1})-\Phi(\uvx_k)\mid \uvx_{k-1}\right] \\
	&&~+\frac{\stepsize (1+\kappa)-\underline \kappa}{2\stepsize}\mathbb E[\|\iterDiffk\|_2^2\mid \uvx_{k-1}]+\sigma^2/2\\[5pt]
	&\leq & c_k\big(\mathbb E\left[\Phi(\uvx_{k-1})-\Phi(\uvx_k) \mid \uvx_{k-1}\right]+\sigma^2/2\big),
\end{array}
\end{equation}
where $c_k=1+\frac{\stepsize (\kappa+1)-\underline \kappa}{2\underline \kappa- \stepsize(1+\kappa)}$. The first and second inequalities come from  \eqref{eq:Proof:StrII:VarBound} and \eqref{eq:Proof:StrII:bound:iterdiff}, respectively.}


Now, we construct a lower bound for the left hand side of \eqref{eq:Proof:ExpPLStrII}. \MinRev{By using \Cref{lemma:boundedHessSR1}, we have the following inequality
$$
 \mathcal B_{\bar h}(\uvx',\uvx_{k-1},\uvg_{s'}^k,\umB^k,\stepsize)
\leq \mathcal B_{\bar h}(\uvx',\uvx_{k-1},\uvg_{s'}^k,\umI_N,\frac{\stepsize}{\overline{\kappa}}).
$$
Taking the conditional expectation relative to the variable $\uvx_{k-1}$ on both sides of the above inequality, we obtain
$$
\mathbb E\left[\mathcal B_{\bar h}(\uvx',\uvx_{k-1},\uvg_{s'}^k,\umB^k,\stepsize) \mid \uvx_{k-1}\right]\leq \mathcal B_{\bar h}(\uvx',\uvx_{k-1},\vnabla \digamma(\uvx_{k-1}),\umI_N,\frac{\stepsize}{\overline{\kappa}}).
$$}
Minimizing both sides of the above inequality with respect to $\uvx'$, we reach
\MinRev{
\begin{equation}
\label{eq:Proof:ExpEstiToGradLP}
		\min\limits_{\uvx'\in\mathcal C}\mathbb E\left[\mathcal B_{\bar h}(\uvx',\uvx_{k-1},\uvg_{s'}^k,\umB^k,\stepsize)\mid \uvx_{k-1}\right]\leq\min\limits_{\uvx'\in\mathcal C}\mathcal B_{\bar h}(\uvx',\uvx_{k-1},\vnabla \digamma(\uvx_{k-1}),\umI_N,\frac{\stepsize}{\overline{\kappa}}).
\end{equation}
}
By using the following relations 
$$
	\mathcal D_{\bar h}^{\mathcal C}(\uvx_{k-1},\uvg_{s'}^k,\umB^k,\stepsize) = -\frac{2}{\stepsize}\min\limits_{\uvx'\in\mathcal C}\mathcal B_{\bar h}(\uvx',\uvx_{k-1},\uvg_{s'}^k,\umB^k,\stepsize),
$$ 
$$
\mathcal D_{\bar h}^{\mathcal C}(\uvx_{k-1},\vnabla \digamma(\uvx_{k-1}),\umI_N,\frac{\stepsize}{\overline{\kappa}})= -\frac{2\overline{\kappa}}{\stepsize}\min\limits_{\uvx'\in\mathcal C}\mathcal B_{\bar h}(\uvx',\uvx_{k-1},\vnabla \digamma(\uvx_{k-1}),\umI_N,\frac{\stepsize}{\overline{\kappa}}),
$$
\MinRev{
$$
\mathbb E \min_{\uvx'\in\mathcal C}  \left[
\mathcal{B}_{\bar h}(\uvx', \uvx_{k-1}, \uvg_{s'}^k, \umB^k, a_k)
\,\middle|\, \uvx_{k-1}
\right]\leq \min_{\uvx'\in\mathcal C} \mathbb E \left[
\mathcal{B}_{\bar h}(\uvx', \uvx_{k-1}, \uvg_{s'}^k, \umB^k, a_k)
\,\middle|\, \uvx_{k-1}
\right],
$$}
 and \eqref{eq:Proof:ExpEstiToGradLP}, we derive
\MinRev{
$$
\mathbb E\left[\frac{\stepsize}{2}\mathcal D_{\bar h}^{\mathcal C}(\uvx_{k-1},\uvg_{s'}^k,\umB^k,\stepsize)\mid \uvx_{k-1}\right]\geq \frac{\stepsize}{2\overline{\kappa}} \mathcal D_{\bar h}^{\mathcal C}(\uvx_{k-1},\vnabla \digamma(\uvx_{k-1}), \umI_N,\frac{\stepsize}{\overline{\kappa}}).
$$}
Substituting the above inequality into \eqref{eq:Proof:ExpPLStrII}, we reach
\MinRev{
$$
\frac{\stepsize}{2\overline{\kappa}c_k} \mathcal D_{\bar h}^{\mathcal C}(\uvx_{k-1},\vnabla \digamma(\uvx_{k-1}),\umI_N,\frac{\stepsize}{\overline{\kappa}})\leq \mathbb E\left[\Phi(\uvx_{k-1})-\Phi(\uvx_k)\mid \uvx_{k-1}\right]+\sigma^2/2.
$$}
Summing up the above inequality from $k=1$ to $K$, we obtain
\MinRev{
$$
\sum_{k=1}^K\frac{\stepsize}{2\overline{\kappa}c_k} \mathcal D_{\bar h}^{\mathcal C}(\uvx_{k-1},\vnabla \digamma(\uvx_{k-1}),\umI_N,\frac{\stepsize}{\overline{\kappa}})\leq \mathbb E\left[\Phi(\uvx_0)-\Phi^*\right]+K\sigma^2/2.
$$}
Sampling the output iterate $\uvx_{k^*}$ with probability mass function $\mathrm{Prob}\{\uvx_{k^*}\}=\frac{\stepsize}{2\overline{\kappa}c_k K}$ for any $k=1,2,\cdots,K$, we reach 
\MinRev{
$$
 \mathbb E\left[\mathcal D_{\bar h}^{\mathcal C}(\uvx_{k^*},\vnabla \digamma(\uvx_{k^*}),\umI_N,\frac{a_{k^*}}{\overline{\kappa}})\right] \leq \frac{\mathbb E\left[\Phi(\uvx_0)-\Phi^*\right]}{K}+\sigma^2/2.
$$}
Using \eqref{eq:PLIneq} and the above inequality, we establish 
\MinRev{
$$
\mathbb E\left[\Phi(\uvx_{k^*})-\Phi^*\right] \leq \frac{\mathbb E\left[\Phi(\uvx_0)-\Phi^*\right]}{2\varrho K}+\sigma^2/(4\varrho).
$$}

\section*{Acknowledgments}
\MajRev{We thank the anonymous reviewers for their constructive feedback, which has helped improve the quality of this paper.} We would also like to thank Dr.~Ahmed Ayoub, Dr.~Joowon Lim, and Prof.~Demetri Psaltis for providing us with the real data.

\bibliographystyle{siamplain}
\bibliography{ResponseSIIMS/Refs}

\end{document}

%% file: ex_shared.tex
\usepackage{lipsum}
\usepackage{amsfonts}
\usepackage{graphicx}
\usepackage{epstopdf}
\usepackage{algorithmic}
\ifpdf
  \DeclareGraphicsExtensions{.eps,.pdf,.png,.jpg}
\else
  \DeclareGraphicsExtensions{.eps}
\fi

\newsiamremark{remark}{Remark}
\newsiamremark{hypothesis}{Hypothesis}
\crefname{hypothesis}{Hypothesis}{Hypotheses}
\newsiamthm{claim}{Claim}
\newsiamremark{fact}{Fact}
\crefname{fact}{Fact}{Facts}

\headers{Mini-Batch Quasi-Newton Proximal Method Nonlinear Image Reconstruction}{T. Hong, T. Pham, I. Yavneh, and M. Unser}

\title{A Mini-Batch Quasi-Newton Proximal Method for Constrained Total Variation Nonlinear Image Reconstruction
\thanks{Submitted to the editors DATE.
}}

\author{Tao Hong\thanks{Oden Institute for Computational Engineering and Sciences, University of Texas at Austin, Austin, TX 78712, USA (\email{tao.hong@austin.utexas.edu}).}
\and Thanh-an Pham\thanks{Biomedical Imaging Group, Ecole polytechnique f\'{e}d\'{e}rale de Lausanne, Lausanne, Switzerland (\email{thanh-an.pham@proton.me}).}
\and Irad Yavneh\thanks{Department
of Computer Science, Technion-Israel Institute of Technology, Haifa, 3200003 Israel (\email{irad@cs.technion.ac.il}).}
\and Michael Unser\thanks{Biomedical Imaging Group, Ecole polytechnique f\'{e}d\'{e}rale de Lausanne, Lausanne, Switzerland (\email{michael.unser@epfl.ch}).}}
\usepackage{amsopn}


%% file: defs.tex
\usepackage{subfigure}

\usepackage{url}
\usepackage{bm}
\usepackage{amsmath,amssymb}
\usepackage{paralist}
\usepackage{algorithm}
\usepackage{algorithmic}        
\usepackage{multirow}
\renewcommand{\algorithmicrequire}{\textbf{Initialization:}}   
\renewcommand{\algorithmicensure}{\textbf{Iteration:}}
\renewcommand{\algorithmiclastcon}{\textbf{Output:}}

\usepackage[squaren,Gray]{SIunits}
\usepackage{hyperref} 
\usepackage{cleveref} 
\Crefname{figure}{Fig.}{Figs.}

\newtheorem{assumption}{Assumption}[section] 
\crefname{assumption}{assumption}{assumptions}
\Crefname{assumption}{Assumption}{Assumptions}

\newcommand{\mrm}{\mathrm}
\newcommand{\e}{\begin{equation}}
\newcommand{\ee}{\end{equation}}
\newcommand{\en}{\begin{equation*}}
\newcommand{\een}{\end{equation*}}
\newcommand{\eqn}{\begin{eqnarray}}
\newcommand{\eeqn}{\end{eqnarray}}
\newcommand{\bmat}{\begin{bmatrix}}
\newcommand{\emat}{\end{bmatrix}}
\newcommand{\BIT}{\begin{itemize}}
\newcommand{\EIT}{\end{itemize}}

\newcommand{\vr}{\bm r}

\newcommand{\vu}{\bm u}

\newcommand{\vx}{\bm x}

\newcommand{\uvd}{\mathrm{\mathbf d}}

\newcommand{\uvg}{\mathrm{\mathbf g}}
\newcommand{\uvh}{\mathrm{\mathbf h}}

\newcommand{\uvk}{\mathrm{\mathbf k}}

\newcommand{\uvm}{\mathrm{\mathbf m}}

\newcommand{\uvr}{\mathrm{\mathbf r}}
\newcommand{\uvs}{\mathrm{\mathbf s}}

\newcommand{\uvu}{\mathrm{\mathbf u}}
\newcommand{\uvv}{\mathrm{\mathbf v}}
\newcommand{\uvw}{\mathrm{\mathbf w}}
\newcommand{\uvx}{\mathrm{\mathbf x}}
\newcommand{\uvy}{\mathrm{\mathbf y}}
\newcommand{\uvz}{\mathrm{\mathbf z}}

\newcommand{\umB}{\mathrm{\mathbf B}}

\newcommand{\umD}{\mathrm{\mathbf D}}

\newcommand{\umH}{\mathrm{\mathbf H}}
\newcommand{\umI}{\mathrm{\mathbf I}}
\newcommand{\umJ}{\mathrm{\mathbf J}}

\newcommand{\umP}{\mathrm{\mathbf P}}

\newcommand{\umU}{\mathrm{\mathbf U}}

\newcommand{\umW}{\mathrm{\mathbf W}}
\newcommand{\umX}{\mathrm{\mathbf X}}

\newcounter{oursection}

\def\ours{{BQNPM}}
\def\oursF{{FBQNPM}}
\def\wangs{{SQNPM}}
\def\increm{mini-batch{}}
\def\Ksubset{S}
\newcommand{\stepsize}{a_k}
\newcommand{\Trans}{\mathsf{T}}
\newcommand{\vnabla}{\pmb{\nabla}}
\newcommand{\iterDiffk}{\bm\Delta_k}


%% file: figs/scheme3DODT/scheme3DODT.tex
\tikzset{
pics/cone/.style args={#1}{
  code={
    \draw [dotted,fill=#1!12,thick,join=round,opacity=0.5](0,0) -- (2,-.7) -- (2,.7) --cycle;
    \draw [fill=#1!12,dashed,thin,opacity=0.5](2,0) ellipse (.4 and .7);
    \draw [dashed,thin,opacity=0.5](1.2,0) ellipse (.24 and .42);
    \draw [dashed,thick] (-2.8,0) -- (2,0);
    \draw [red,opacity=0.7,-latex,thin] (2,0.7) -- (1.2,0.42);
    \draw [red,opacity=0.7,-latex,thin] (2,-0.7) -- (1.2,-0.42);
    \draw [red,-latex,thin] (1.6,0) -- (1,0);
    \draw [red,opacity=0.5,-latex,thin] (2.35,-0.25) -- (1.4,-0.15);
    \draw [red,opacity=0.5,-latex,thin] (2.35,0.25) -- (1.4,0.15);
    \draw [red,-latex,thin] (1.65,-0.35) -- (1,-0.21);
    \draw [red,-latex,thin] (1.65,0.35) -- (1,0.21);
  }
}
}

\begin{figure}[t]
    \centering
\begin{tikzpicture}
\node[yslant=-0.6944] at (4.0569,1.0870)
{\includegraphics[width=2cm]{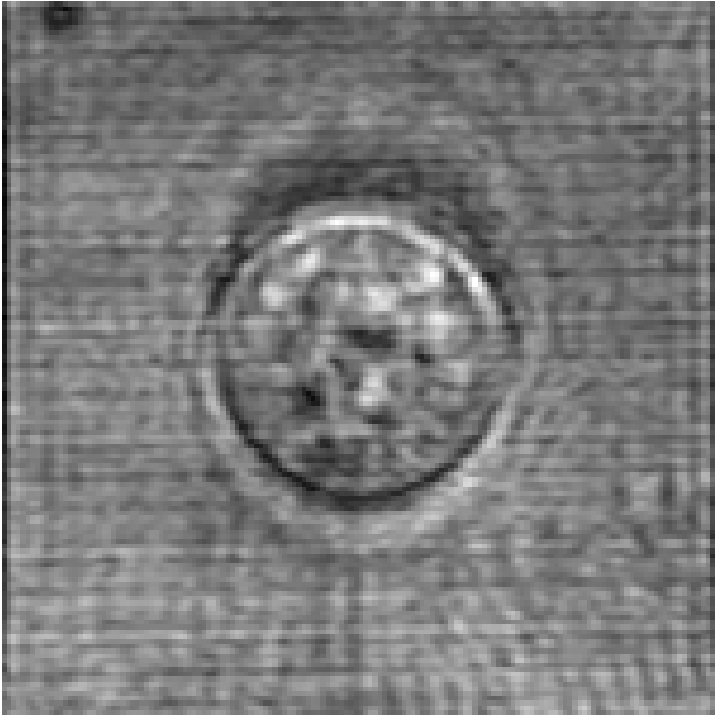}}; 
\node at (3,2.8) {$\Gamma$};
\draw [fill=blue!10,opacity=0.5,join=round](0.9,0.1) -- (-0.35,-0.3) -- (1,-1.2)  -- (2.25,-0.8)  --cycle;
\draw [fill=blue!10,opacity=0.5,join=round](0.9,0.1) -- (2.25,-0.8) -- (2.25,0.7) -- (0.9,1.6) --cycle;
\draw [join=round](0.9,0.1) -- (2.25,-0.8) -- (2.25,0.7) -- (0.9,1.6) --cycle;
\draw [fill=blue!10,opacity=0.5,join=round](0.9,0.1) -- (0.9,1.6) -- (-0.35,1.2) -- (-0.35,-0.3) --cycle;
\draw [join=round](0.9,0.1) -- (0.9,1.6) -- (-0.35,1.2) -- (-0.35,-0.3) --cycle;
\path (0,0) pic [rotate=-165,scale=1.5]{cone=gray};
\node at (1,0.3)
{\includegraphics[rotate=90,scale=0.03]{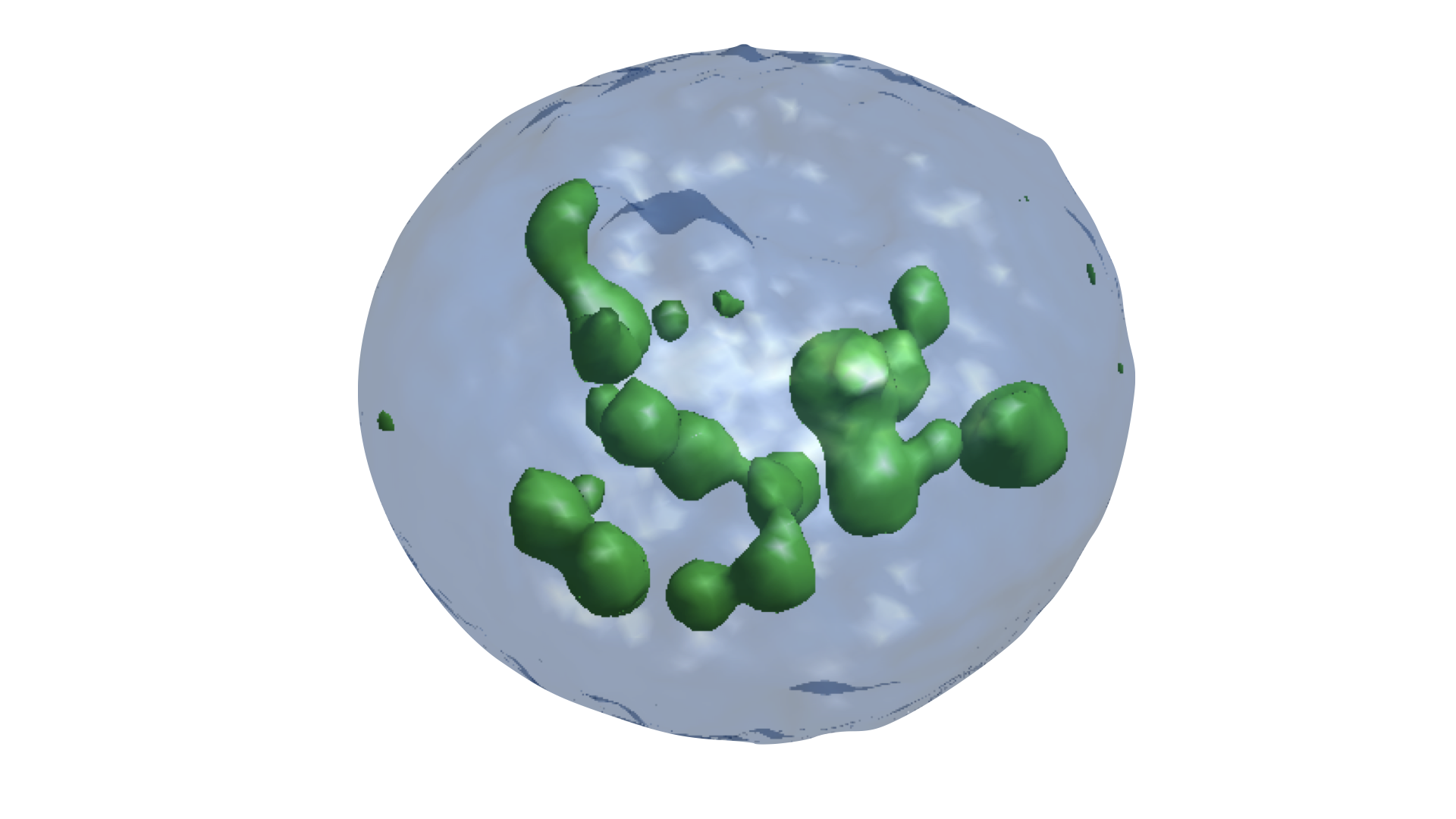}};
\draw [join=round](-0.35,-0.3) -- (1,-1.2) -- (1,0.37) -- (-0.35,1.2) --cycle;
\draw [join=round] (1,-1.2) -- (1,0.37) --  (2.25,0.7)  --(2.25,-0.8) --cycle;
\node at (-0.6,1.3) {$\Omega$};
\node at (1.55,-0.75) {$\eta(\uvr)$};
\node[red] at (-2.4,0.5) {$\uvk_l^\mathrm{in}$};
\end{tikzpicture} 
    \caption{Principle of optical diffraction tomography. The arrows represent the wave vectors  $\{\mathbf{k}^\mathrm{in}_l\in \mathbb R^3\}_{l=1}^L$ of the $L$ incident plane waves $\{u_\mathrm{in}^l\}_{l=1}^L$.
    The angles of illumination are limited to a cone around the optical axis. The refractive-index map of the sample~$\eta(\uvr)$ is embedded in the domain $\Omega\subset\mathbb{R}^3$, and the recorded domain is denoted by $\Gamma$. \MajRev{The parameter $\vr$ denotes the three-dimensional spatial coordinate.}}
    \label{fig:ODT}
\end{figure}

%% file: figs/LS_Simulated.tex
\begin{figure}[!ht]
	\centering
	\newcommand{\colorSNR}{\color{white}}
\subfigure[]{\includegraphics[width=0.495\textwidth]{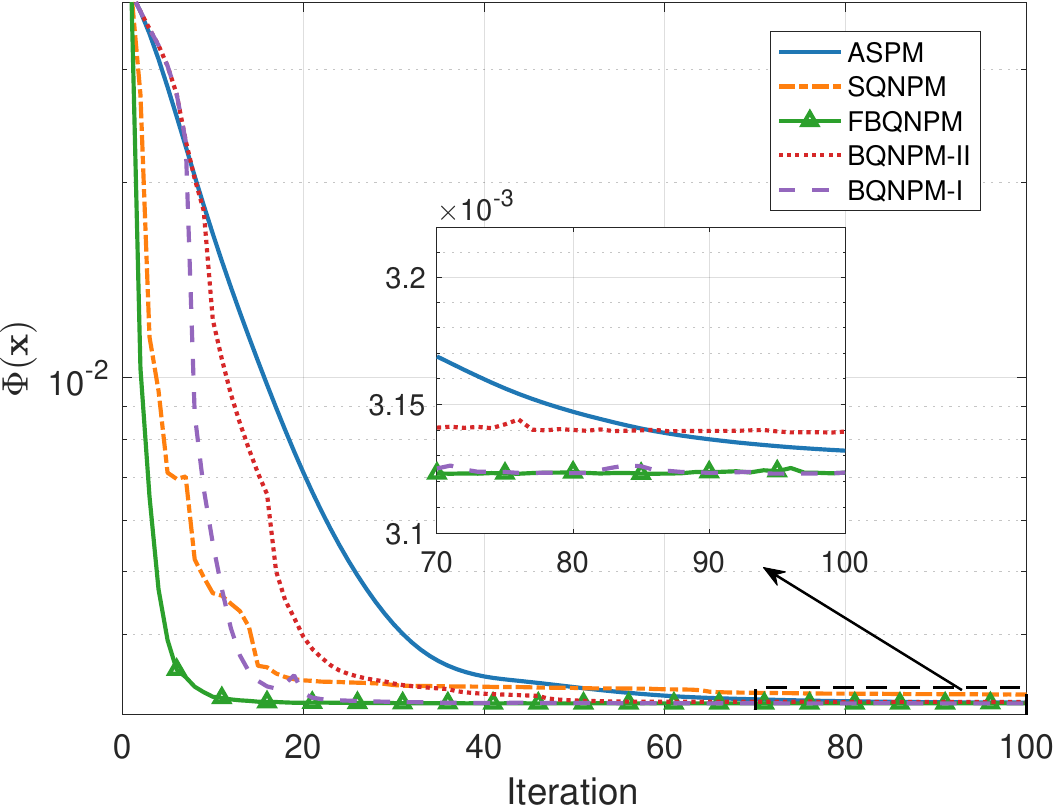}}
\subfigure[]{\includegraphics[width=0.48\textwidth]{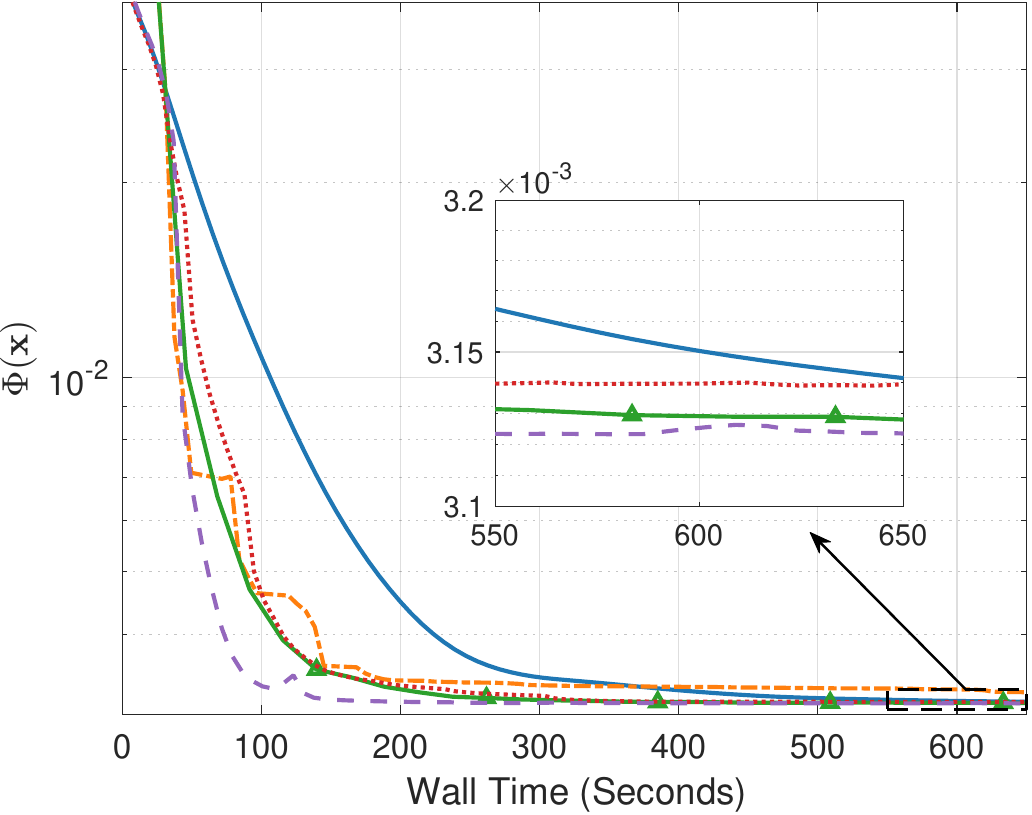}}

\subfigure[]{\includegraphics[width=0.495\textwidth]{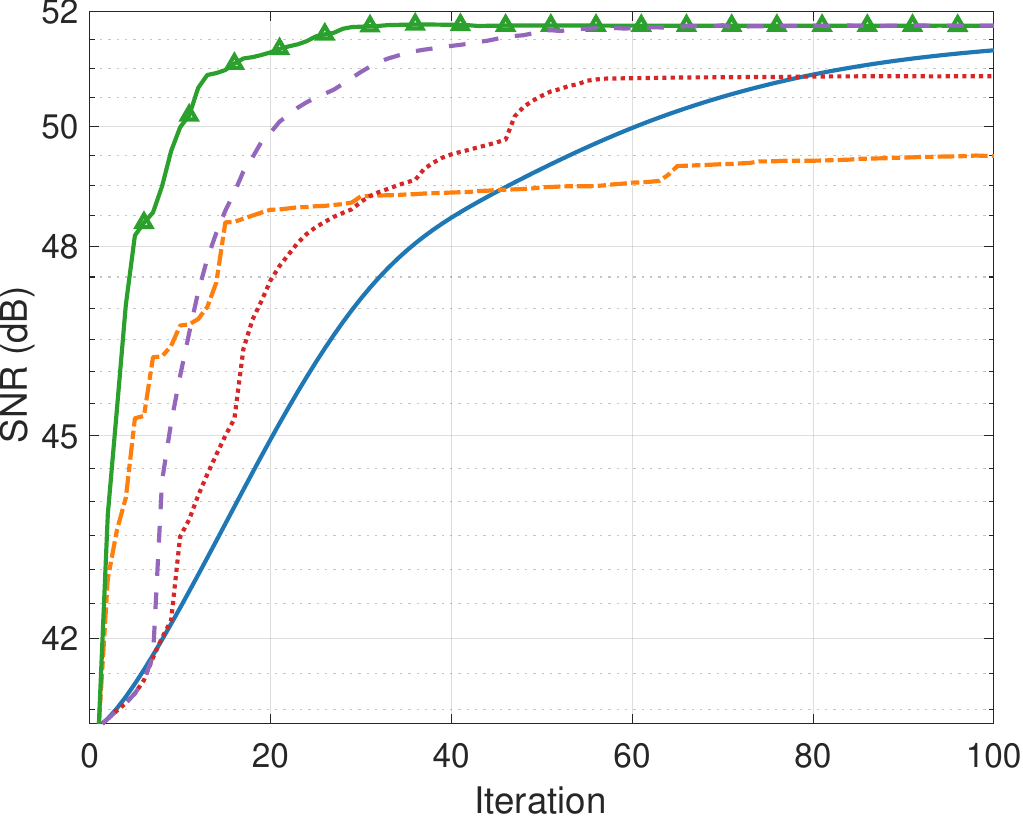}}
\subfigure[]{\includegraphics[width=0.48\textwidth]{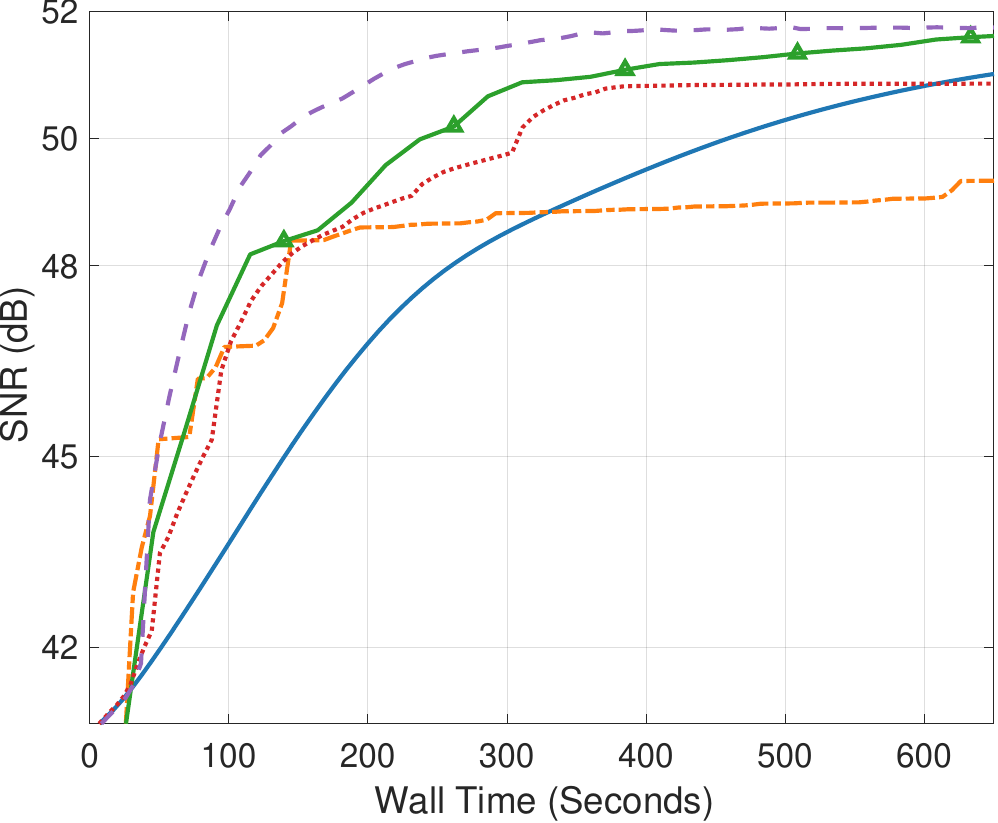}}
\caption{
Performance of ASPM, \wangs{} \cite{wang2019stochastic}, \ours{}-I/II, and \oursF{} algorithms on the strongly scattering simulated sample using the LippS model.
From top to bottom rows: Full cost and SNR values versus iteration and wall time.}
	\label{fig:LS:CostSNR:Simulated}
\end{figure}

\begin{figure}[!ht]
	\centering
	\newcommand{\colorSNR}{\color{white}}
	\hspace{1.3cm}
\begin{tikzpicture}
    \pgfplotsset{every axis y label/.append style={anchor=base,yshift = -1.5mm} }
     \begin{axis}[at={(0,0)},anchor = north west, ylabel = XZ,
    xmin = 0,xmax = 216,ymin = 0,ymax = 55, width=0.6\textwidth,
        scale only axis,
        enlargelimits=false,
        axis line style={draw=none},
        tick style={draw=none},
        axis equal image,
        xticklabels={,,},yticklabels={,,},
        ]
    \node[inner sep=0pt, anchor = south west] (GT_xz_1) at (0,0) {\includegraphics[ width=0.15\textwidth]{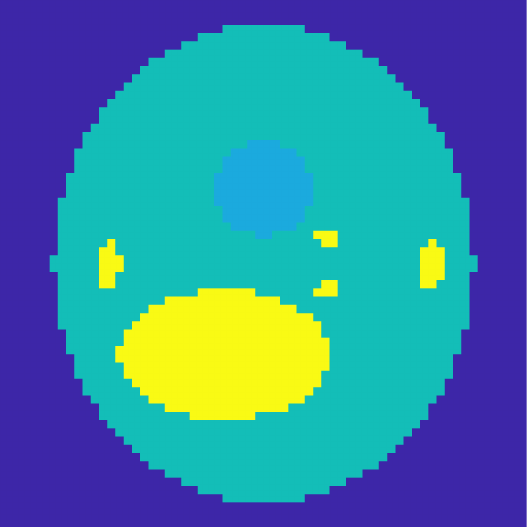}};
    \node at (6,4) {\scriptsize{\color{white} GT}};
    

  \node[inner sep=0pt, anchor = west] (FISTA_xz_1) at (GT_xz_1.east) {\includegraphics[ width=0.15\textwidth]{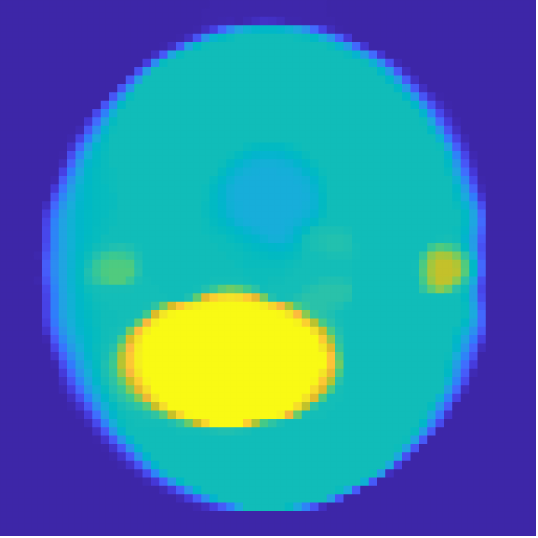}};
  \node[anchor=north west,xshift=-1mm, yshift=1mm] at (FISTA_xz_1.north west) {\colorSNR{}\footnotesize$46.55$dB};
  \node at (65,4) {\scriptsize{\color{white} ASPM}};

   \node[inner sep=0pt, anchor =  west] (QN_15_xz_1) at (FISTA_xz_1.east) {\includegraphics[width=0.15\textwidth]{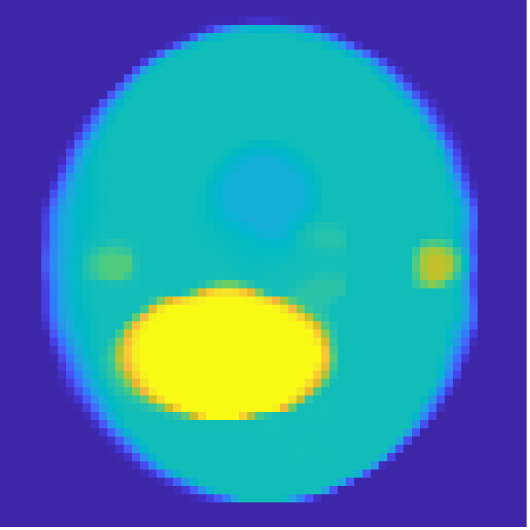}};
  \node[anchor=north west,xshift=-1mm, yshift=1mm] at (QN_15_xz_1.north west) {\colorSNR{}\footnotesize$46.63$dB};
  \node at (125,3) {\scriptsize{\color{white} \ours{}-I}};

  \node[inner sep=0pt, anchor = west] (QN_35_xz_1) at (QN_15_xz_1.east) {\includegraphics[width=0.15\textwidth]{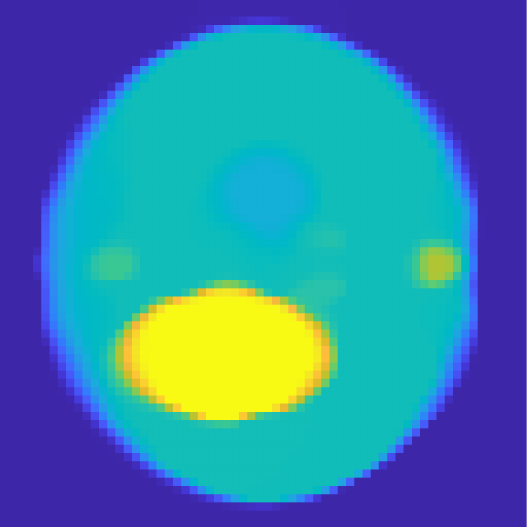}};
   \node[anchor=north west,xshift=-1mm, yshift=1mm] at (QN_35_xz_1.north west) {\colorSNR{}\small$45.98$dB};
  \node at (180,3) {\scriptsize{\color{white} \ours{}-II}};

%

 \end{axis}
    
\begin{axis}[at={(GT_xz_1.south west)},anchor = north west,ylabel style={align=center},ylabel = {XY},yshift=0.15cm,
    xmin = 0,xmax = 216,ymin = 0,ymax = 58, width=0.6\textwidth,
        scale only axis,
        enlargelimits=false,
       axis line style={draw=none},
       tick style={draw=none},
        axis equal image,
        xticklabels={,,},yticklabels={,,}
       ]
              
    \node[inner sep=0pt, anchor = south west] (GT_xy_1) at (0,0) {\includegraphics[ width=0.15\textwidth]{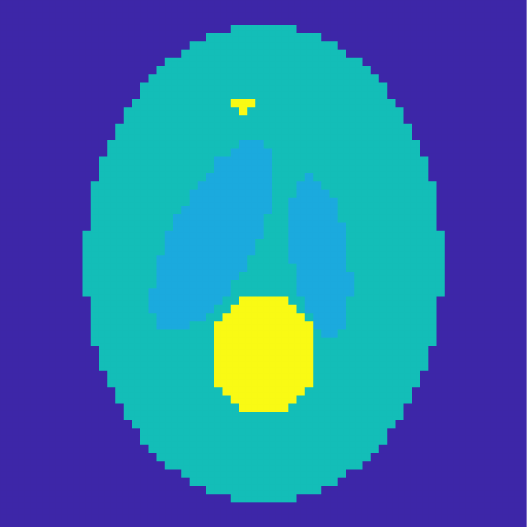}};

  \node[inner sep=0pt, anchor = west] (FISTA_xy_1) at (GT_xy_1.east) {\includegraphics[ width=0.15\textwidth]{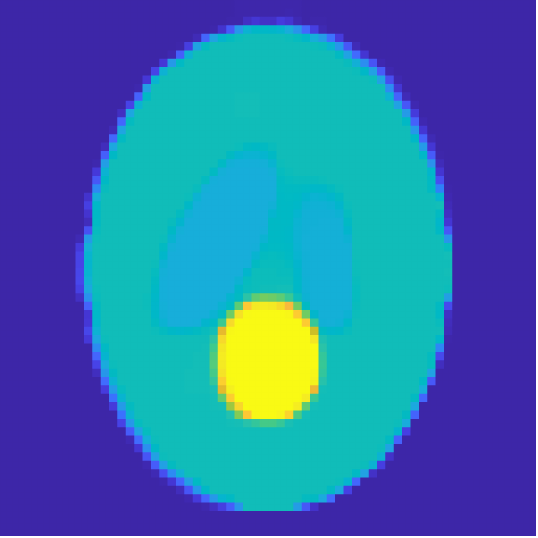}};
  \node[anchor=north west,xshift=-1mm, yshift=1mm] at (FISTA_xy_1.north west) {\colorSNR{}\footnotesize$51.49$dB};

   \node[inner sep=0pt, anchor =  west] (QN_15_xy_1) at (FISTA_xy_1.east) {\includegraphics[ width=0.15\textwidth]{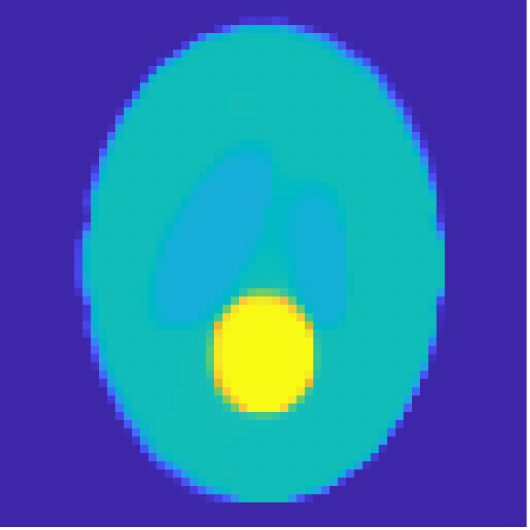}};
  \node[anchor=north west,xshift=-1mm, yshift=1mm] at (QN_15_xy_1.north west) {\colorSNR{}\footnotesize$51.51$dB};

  \node[inner sep=0pt, anchor = west] (QN_35_xy_1) at (QN_15_xy_1.east) {\includegraphics[ width=0.15\textwidth]{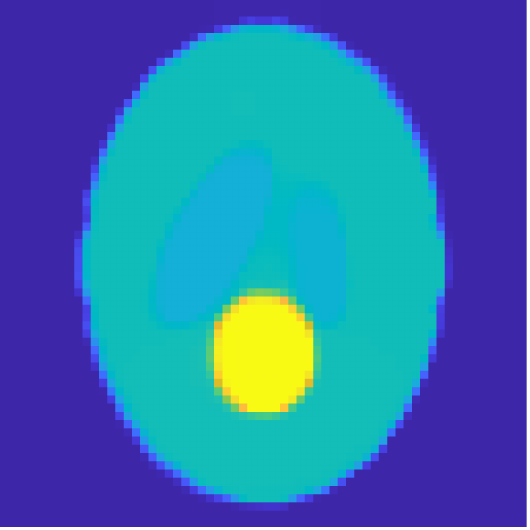}};
   \node[anchor=north west,xshift=-1mm, yshift=1mm] at (QN_35_xy_1.north west) {\colorSNR{}\small$51.09$dB};

%
%
   
 \end{axis}
 
\begin{axis}[at={(GT_xy_1.south west)},anchor = north west,ylabel style={align=center},ylabel = {YZ},yshift=0.15cm,
    xmin = 0,xmax = 216,ymin = 0,ymax = 58, width=0.6\textwidth,
        scale only axis,
        enlargelimits=false,
       axis line style={draw=none},
       tick style={draw=none},
        axis equal image,
        xticklabels={,,},yticklabels={,,}
        colormap name=parula,
       colormap={parula}{
            rgb=(0.208100000000000,0.166300000000000,0.529200000000000)
            rgb=(0.211623809523810,0.189780952380952,0.577676190476191)
            rgb=(0.212252380952381,0.213771428571429,0.626971428571429)
            rgb=(0.208100000000000,0.238600000000000,0.677085714285714)
            rgb=(0.195904761904762,0.264457142857143,0.727900000000000)
            rgb=(0.170728571428571,0.291938095238095,0.779247619047619)
            rgb=(0.125271428571429,0.324242857142857,0.830271428571429)
            rgb=(0.0591333333333334,0.359833333333333,0.868333333333333)
            rgb=(0.0116952380952381,0.387509523809524,0.881957142857143)
            rgb=(0.00595714285714286,0.408614285714286,0.882842857142857)
            rgb=(0.0165142857142857,0.426600000000000,0.878633333333333)
            rgb=(0.0328523809523810,0.443042857142857,0.871957142857143)
            rgb=(0.0498142857142857,0.458571428571429,0.864057142857143)
            rgb=(0.0629333333333333,0.473690476190476,0.855438095238095)
            rgb=(0.0722666666666667,0.488666666666667,0.846700000000000)
            rgb=(0.0779428571428572,0.503985714285714,0.838371428571429)
            rgb=(0.0793476190476190,0.520023809523810,0.831180952380952)
            rgb=(0.0749428571428571,0.537542857142857,0.826271428571429)
            rgb=(0.0640571428571428,0.556985714285714,0.823957142857143)
            rgb=(0.0487714285714286,0.577223809523810,0.822828571428572)
            rgb=(0.0343428571428572,0.596580952380952,0.819852380952381)
            rgb=(0.0265000000000000,0.613700000000000,0.813500000000000)
            rgb=(0.0238904761904762,0.628661904761905,0.803761904761905)
            rgb=(0.0230904761904762,0.641785714285714,0.791266666666667)
            rgb=(0.0227714285714286,0.653485714285714,0.776757142857143)
            rgb=(0.0266619047619048,0.664195238095238,0.760719047619048)
            rgb=(0.0383714285714286,0.674271428571429,0.743552380952381)
            rgb=(0.0589714285714286,0.683757142857143,0.725385714285714)
            rgb=(0.0843000000000000,0.692833333333333,0.706166666666667)
            rgb=(0.113295238095238,0.701500000000000,0.685857142857143)
            rgb=(0.145271428571429,0.709757142857143,0.664628571428572)
            rgb=(0.180133333333333,0.717657142857143,0.642433333333333)
            rgb=(0.217828571428571,0.725042857142857,0.619261904761905)
            rgb=(0.258642857142857,0.731714285714286,0.595428571428571)
            rgb=(0.302171428571429,0.737604761904762,0.571185714285714)
            rgb=(0.348166666666667,0.742433333333333,0.547266666666667)
            rgb=(0.395257142857143,0.745900000000000,0.524442857142857)
            rgb=(0.442009523809524,0.748080952380952,0.503314285714286)
            rgb=(0.487123809523809,0.749061904761905,0.483976190476191)
            rgb=(0.530028571428571,0.749114285714286,0.466114285714286)
            rgb=(0.570857142857143,0.748519047619048,0.449390476190476)
            rgb=(0.609852380952381,0.747314285714286,0.433685714285714)
            rgb=(0.647300000000000,0.745600000000000,0.418800000000000)
            rgb=(0.683419047619048,0.743476190476191,0.404433333333333)
            rgb=(0.718409523809524,0.741133333333333,0.390476190476190)
            rgb=(0.752485714285714,0.738400000000000,0.376814285714286)
            rgb=(0.785842857142857,0.735566666666667,0.363271428571429)
            rgb=(0.818504761904762,0.732733333333333,0.349790476190476)
            rgb=(0.850657142857143,0.729900000000000,0.336028571428571)
            rgb=(0.882433333333333,0.727433333333333,0.321700000000000)
            rgb=(0.913933333333333,0.725785714285714,0.306276190476191)
            rgb=(0.944957142857143,0.726114285714286,0.288642857142857)
            rgb=(0.973895238095238,0.731395238095238,0.266647619047619)
            rgb=(0.993771428571429,0.745457142857143,0.240347619047619)
            rgb=(0.999042857142857,0.765314285714286,0.216414285714286)
            rgb=(0.995533333333333,0.786057142857143,0.196652380952381)
            rgb=(0.988000000000000,0.806600000000000,0.179366666666667)
            rgb=(0.978857142857143,0.827142857142857,0.163314285714286)
            rgb=(0.969700000000000,0.848138095238095,0.147452380952381)
            rgb=(0.962585714285714,0.870514285714286,0.130900000000000)
            rgb=(0.958871428571429,0.894900000000000,0.113242857142857)
            rgb=(0.959823809523810,0.921833333333333,0.0948380952380953)
            rgb=(0.966100000000000,0.951442857142857,0.0755333333333333)
            rgb=(0.976300000000000,0.983100000000000,0.0538000000000000)
        }, colorbar,
    colorbar style={
    	width=0.3cm,          
        height=5cm,                         
        xshift=0.1cm,  
      yshift = 3.5cm,
        ytick={1.333,1.43},
        yticklabels = {1.333,1.43},
        yticklabel style={font=\small},
    },   point meta min=1.333,
          point meta max=1.43]
          
    \node[inner sep=0pt, anchor = south west] (GT_yz_1) at (0,0) {\includegraphics[ width=0.15\textwidth]{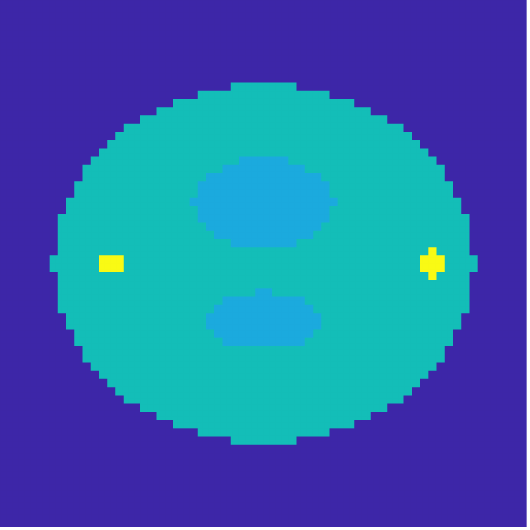}};

   \node[inner sep=0pt, anchor = west] (FISTA_yz_1) at (GT_yz_1.east) {\includegraphics[ width=0.15\textwidth]{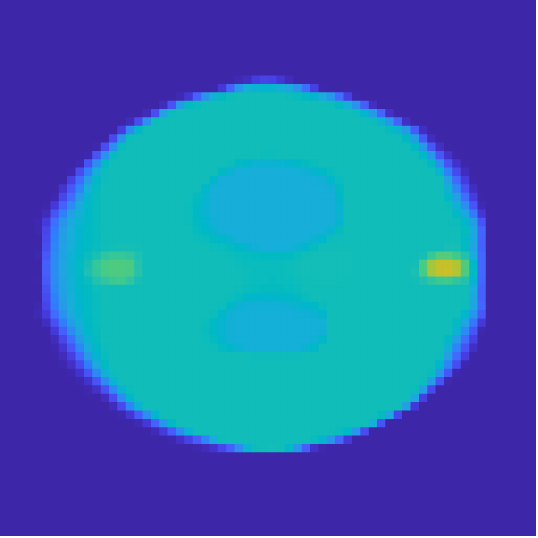}};
  \node[anchor=north west,xshift=-1mm, yshift=1mm] at (FISTA_yz_1.north west) {\colorSNR{}\footnotesize$49.44$dB};

   \node[inner sep=0pt, anchor =  west] (QN_15_yz_1) at (FISTA_yz_1.east) {\includegraphics[ width=0.15\textwidth]{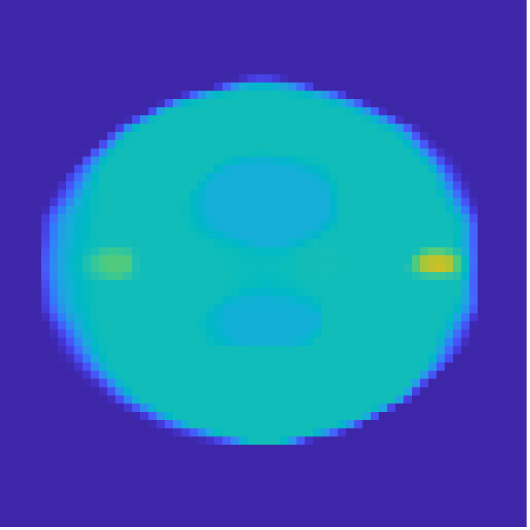}};
  \node[anchor=north west,xshift=-1mm, yshift=1mm] at (QN_15_yz_1.north west) {\colorSNR{}\footnotesize$49.50$dB};

  \node[inner sep=0pt, anchor = west] (QN_35_yz_1) at (QN_15_yz_1.east) {\includegraphics[ width=0.15\textwidth]{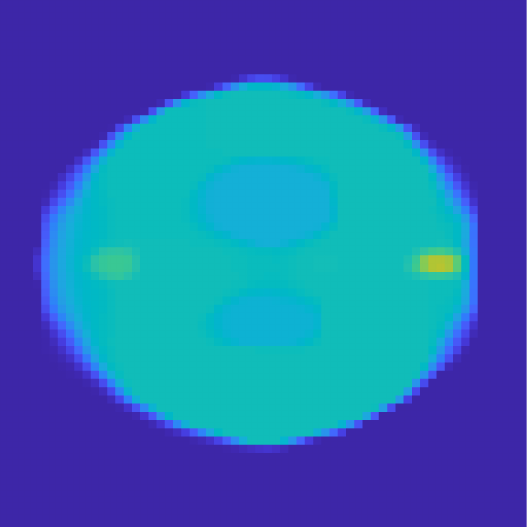}};
   \node[anchor=north west,xshift=-1mm, yshift=1mm] at (QN_35_yz_1.north west){\colorSNR{}\small$48.79$dB};

%
%
   
   \end{axis}
   
   \end{tikzpicture} 
\caption{
Orthoviews of the 3D refractive-index maps obtained by ASPM~(iter. $k=100$) and \ours{}-I/II~(iter. $k=38/100$) algorithms on the strongly scattering simulated sample using the Lippmann-Schwinger model.
The SNR for each slice is displayed in the top-left corner of each image.
 \wangs{} yielded the worst SNR, which we did not present here.}
	\label{fig:LS:Ortho:Simulated}
\end{figure}

%% file: figs/ChoiceLGamma.tex
\begin{figure}[t]
	\centering
\subfigure[\ours{}-I]{\includegraphics[width=0.495\textwidth]{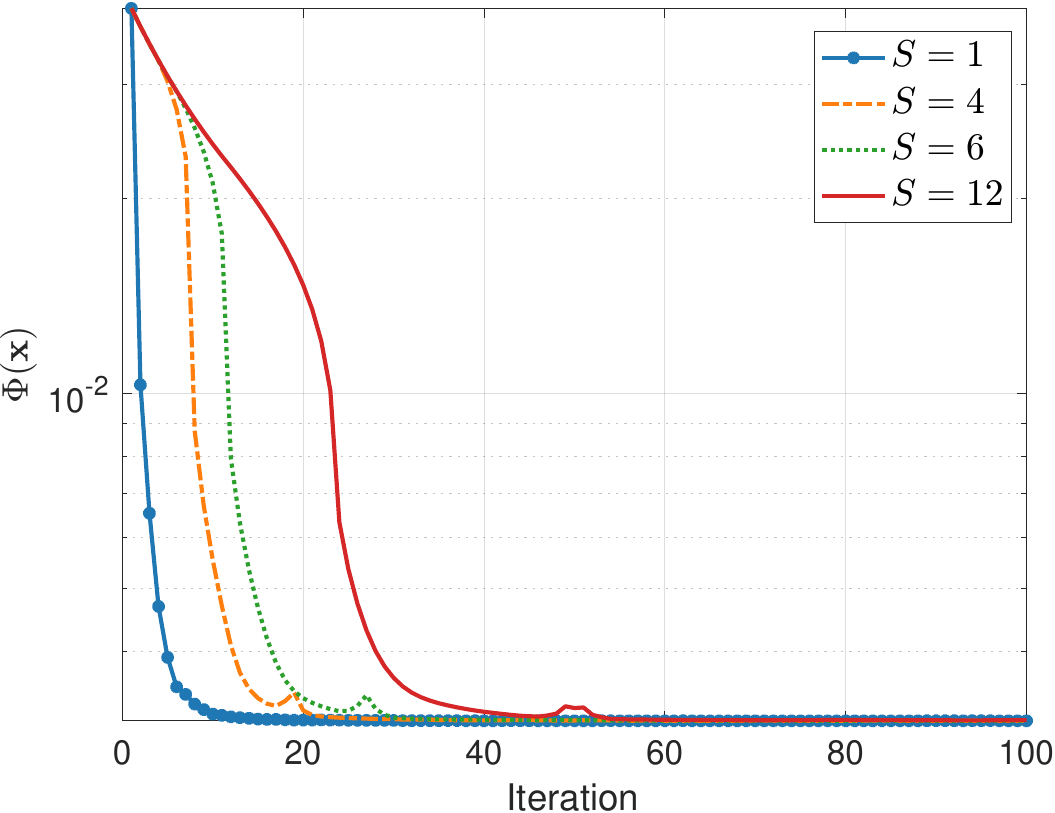}\label{fig:LSSimLSReco:Subset:Iter:I}}
\subfigure[\ours{}-I]{\includegraphics[width=0.48\textwidth]{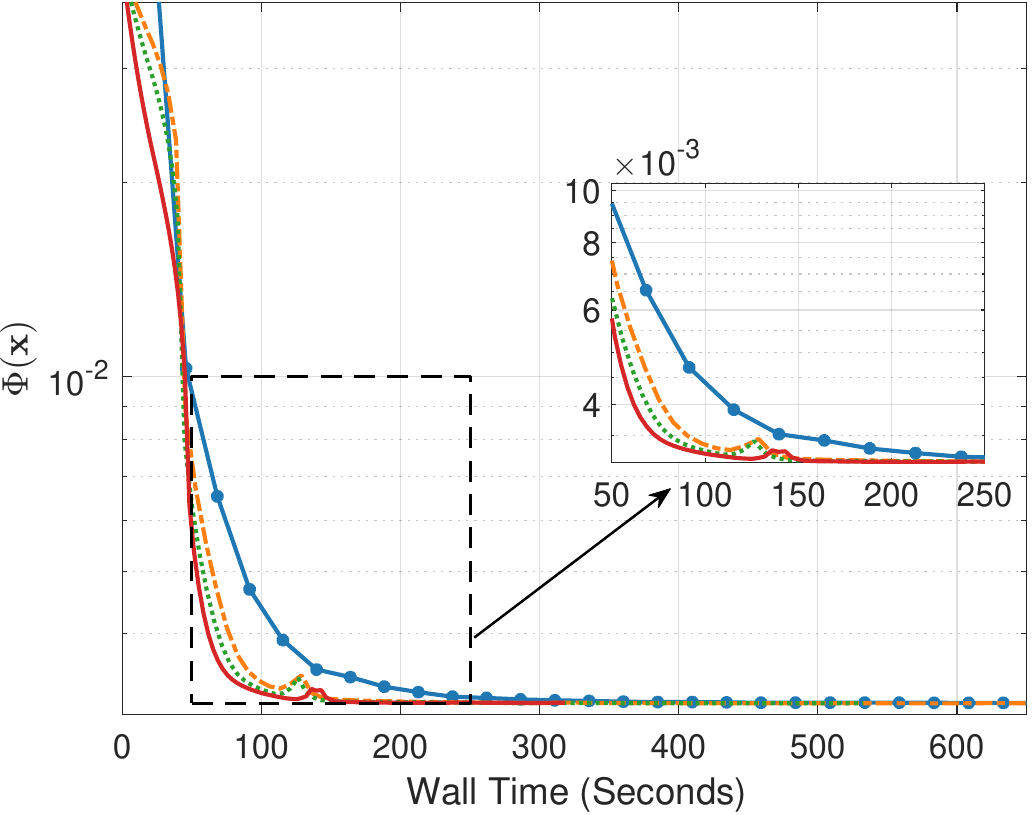}\label{fig:LSSimLSReco:Subset:Time:I}}

\subfigure[\ours{}-II]{\includegraphics[width=0.5\textwidth]{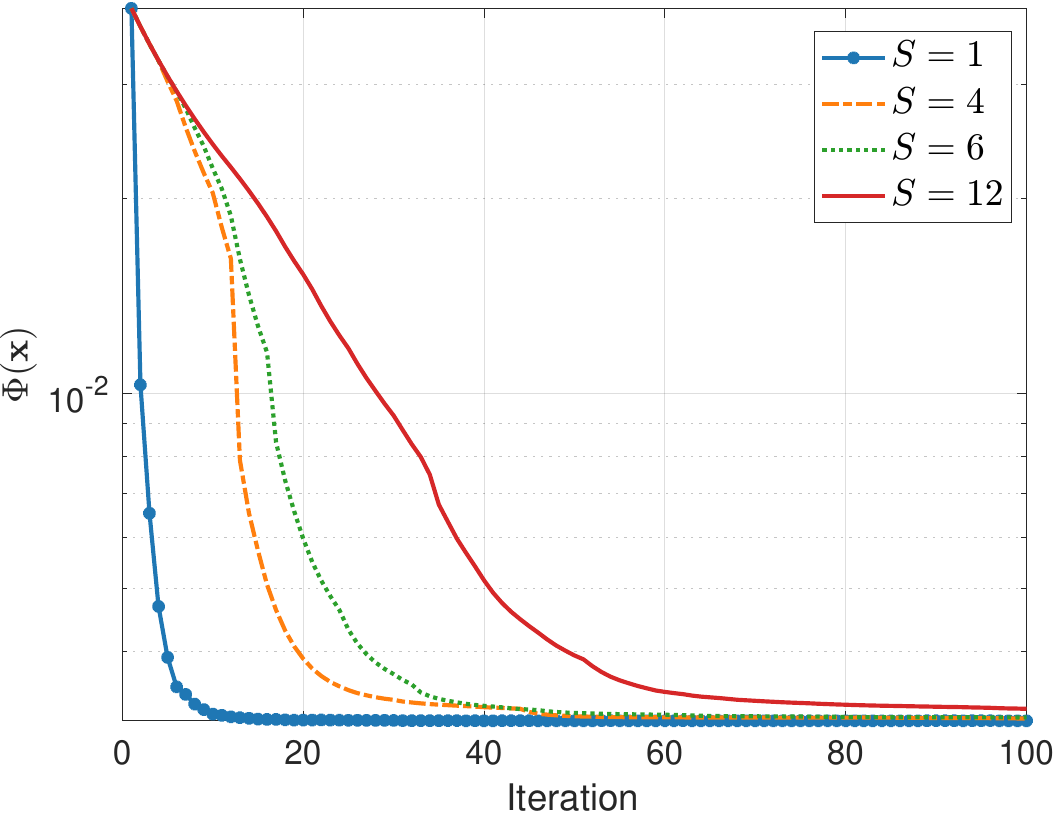}\label{fig:LSSimLSReco:Subset:Iter:II}}
\subfigure[\ours{}-II]{\includegraphics[width=0.48\textwidth]{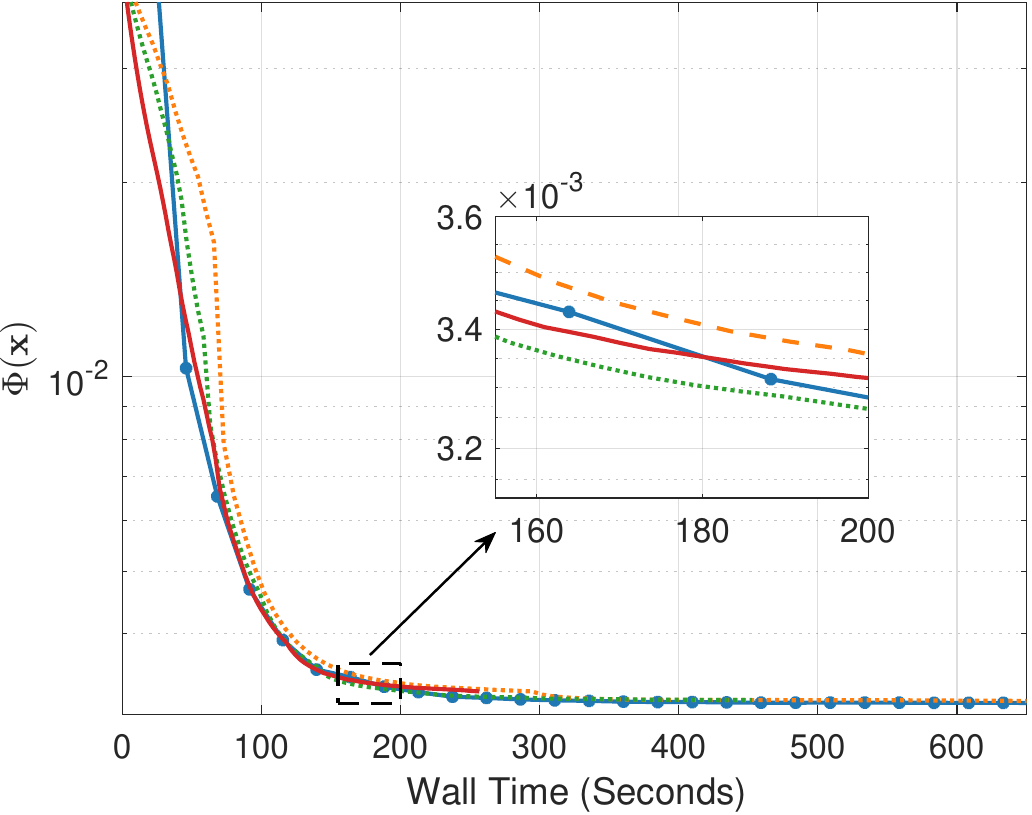}\label{fig:LSSimLSReco:Subset:Time:II}}

\caption{
Effect of $\Ksubset$ on the convergence behavior of \ours{}-I/II with $\gamma=0.8$.}
 \label{fig:LSSimLSReco:Subset}
\end{figure}

\begin{figure}[t]
	\centering

\subfigure[\ours{}-I]{\includegraphics[width=0.45\textwidth]{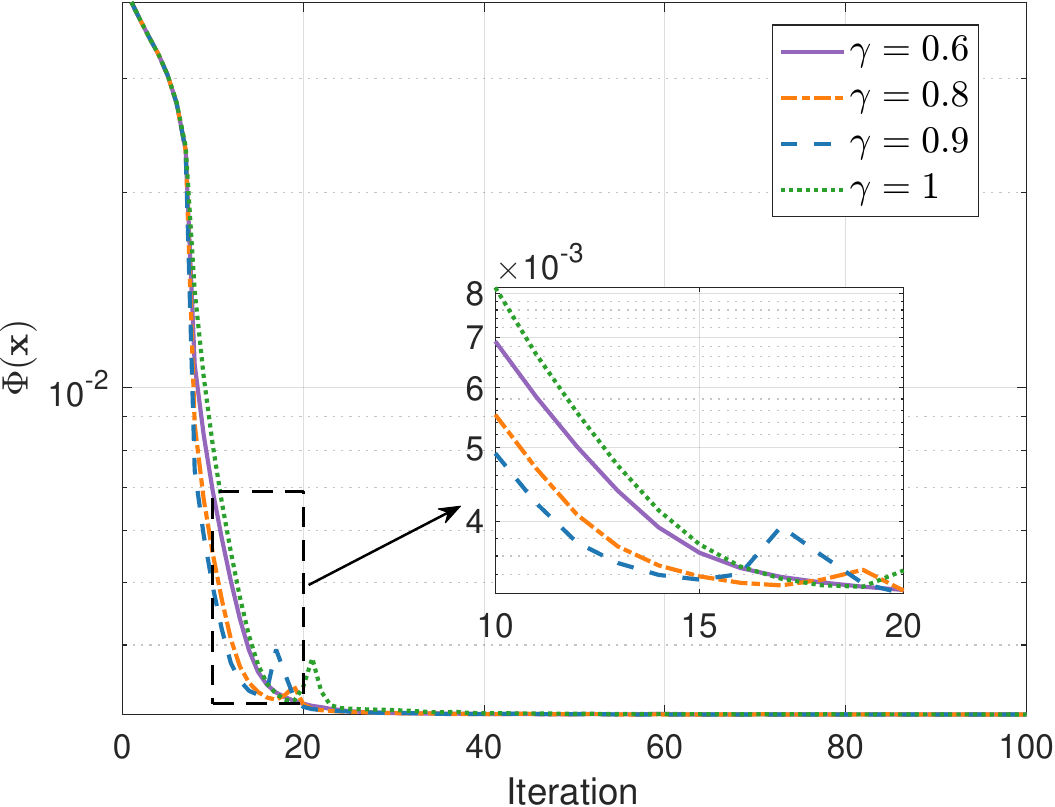}
\label{fig:LSSimLSReco:Gamma:I}}
\subfigure[\ours{}-II]{\includegraphics[width=0.45\textwidth]{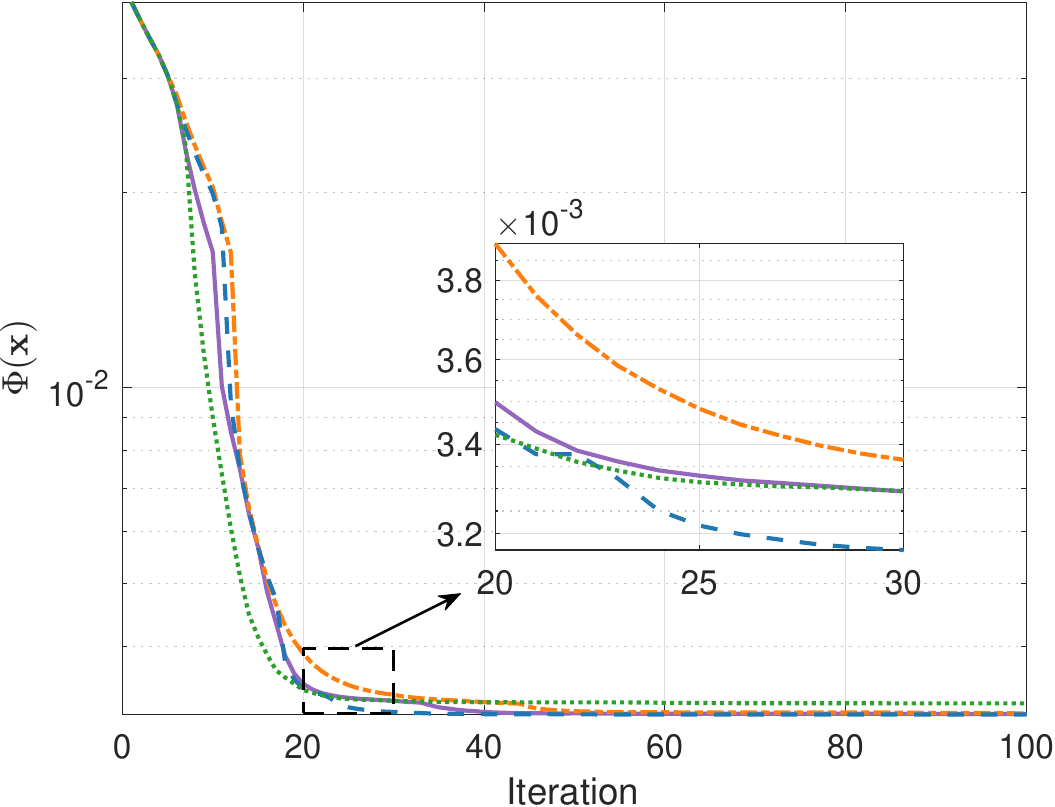}
\label{fig:LSSimLSReco:Gamma:II}}
\caption{
Effect of $\gamma$ on the convergence behavior of \ours{}-I/II with $S=4$.
}
	\label{fig:LSSimLSReco:Gamma}
\end{figure}

%% file: figs/LSReal.tex
\begin{figure}[t]
	\centering
\subfigure[]{\includegraphics[width=0.48\textwidth]{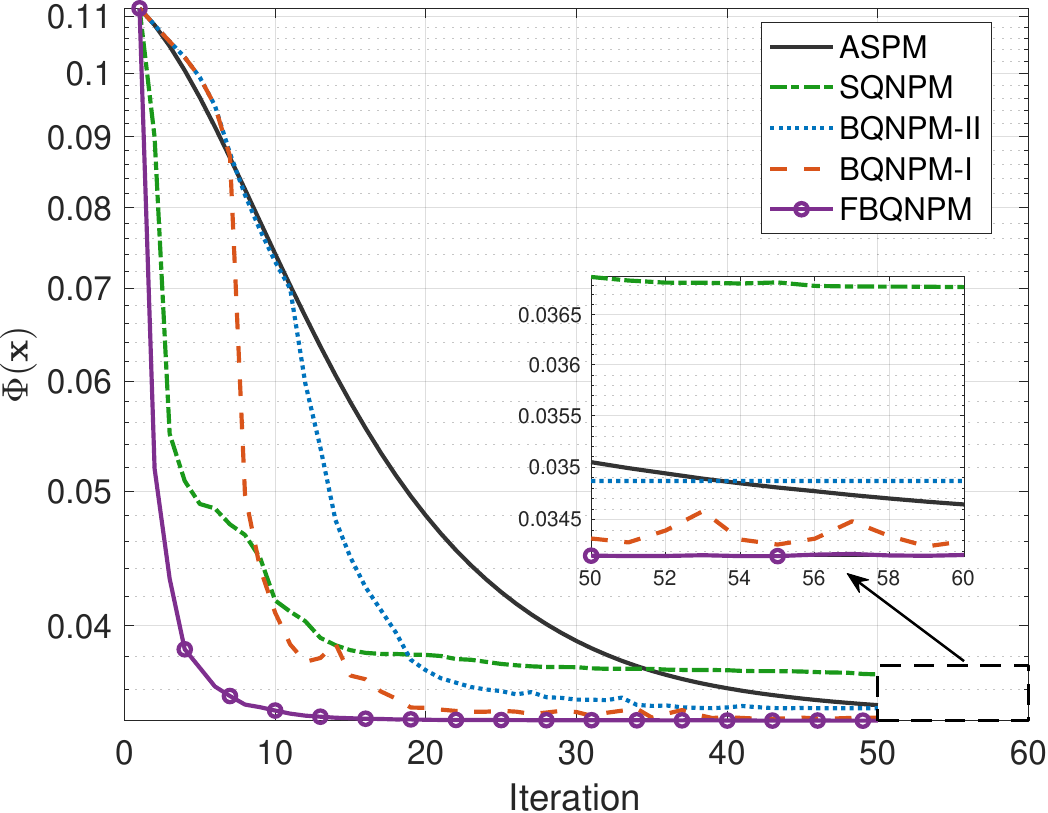}}
\subfigure[]{\includegraphics[width=0.48\textwidth]{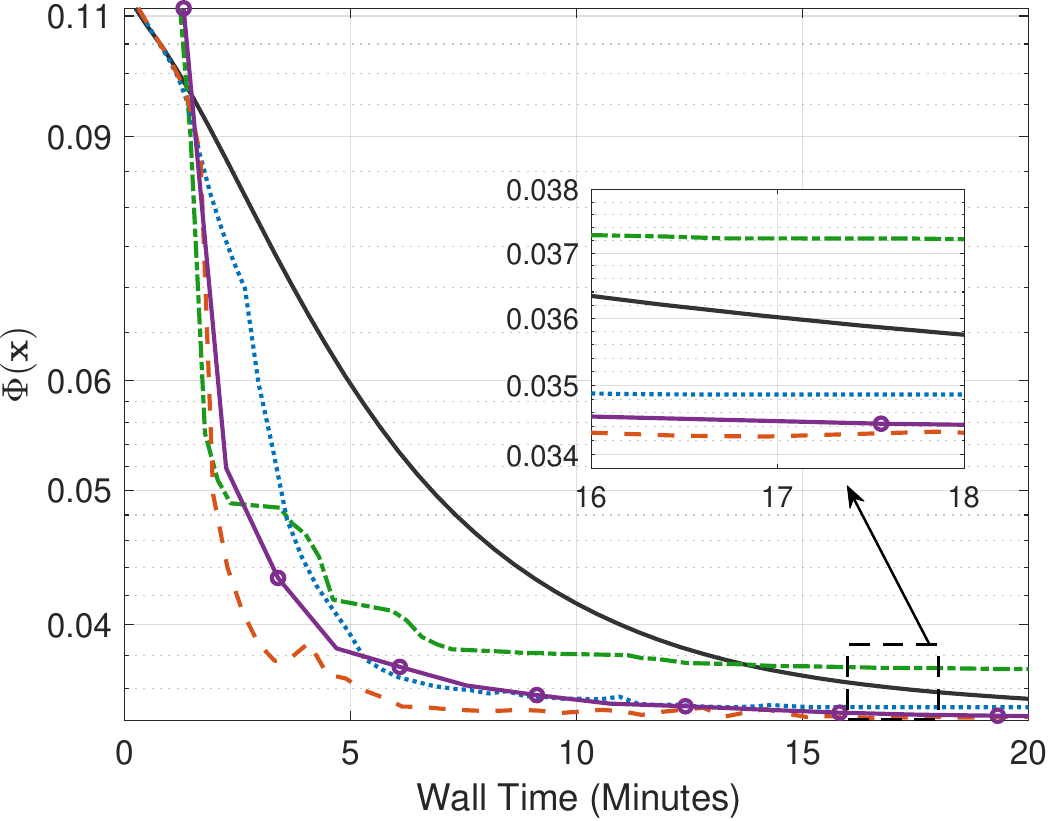}}
\caption{\MajRev{Full cost versus iterations and wall time  for ASPM, \wangs{}, \ours{}-I/II, and \oursF{} on real data~(yeast cell) using the LippS model}.}
\label{fig:LS:CostSNR:real}
\end{figure}

\begin{figure*}[p]%
\centering

\begin{tikzpicture}[scale=0.92]
\pgfplotsset{every axis y label/.append style={anchor=base,yshift = -1.5mm} }
    
\begin{axis}[name={ASPM1},at={(0,0)},anchor = north west,ylabel ={ASPM},ylabel style={anchor=base,align=center,yshift=0cm,xshift={0cm}},
    xmin = 0,xmax = 432, 
    ymin = 0,ymax = 72, width=0.95\textwidth,
        scale only axis,
        enlargelimits=false,
        axis on top,
       axis line style={draw=none},
       tick style={draw=none},
        axis equal image,
        xticklabels={,,},yticklabels={,,}
       ]
       \draw[very thick,dashed,black] (axis cs:216,72) -- (216,0);
   \addplot[] graphics[xmin=0,ymin=0,xmax=72,ymax=72] {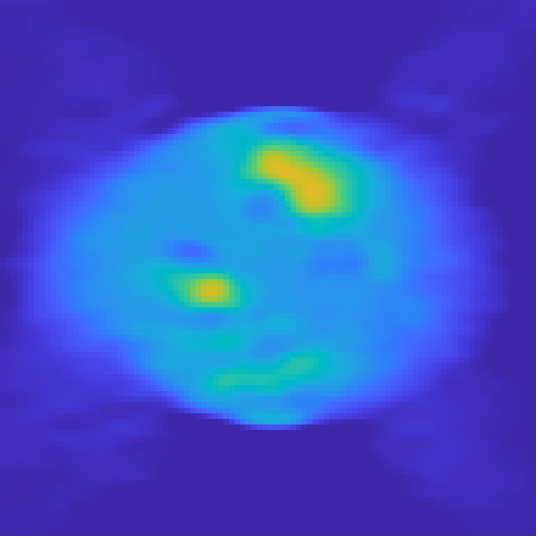};
   
    \node[anchor=north] at (axis cs:36,72) 
    {\color{white} XZ};

    \addplot[] graphics[xmin=72,ymin=0,xmax=144,ymax=72] {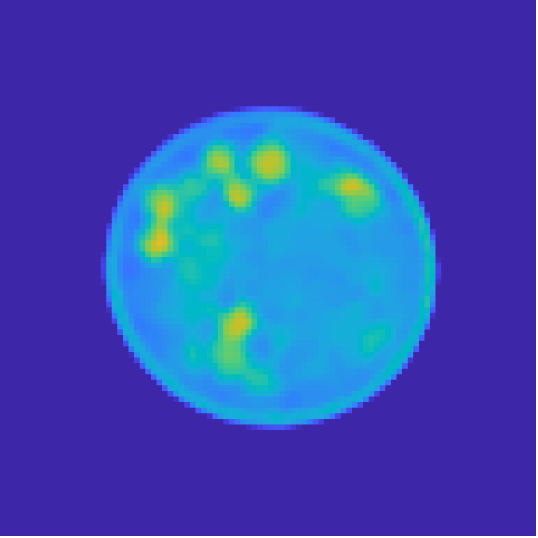};
    \node[anchor=north] at (axis cs:108,72) {\color{white} XY};
 
     \node[anchor=base] at (axis cs:108,2) {\color{white} $\text{iter.}=10$};

     \addplot[] graphics[xmin=144,ymin=0,xmax=216,ymax=72] {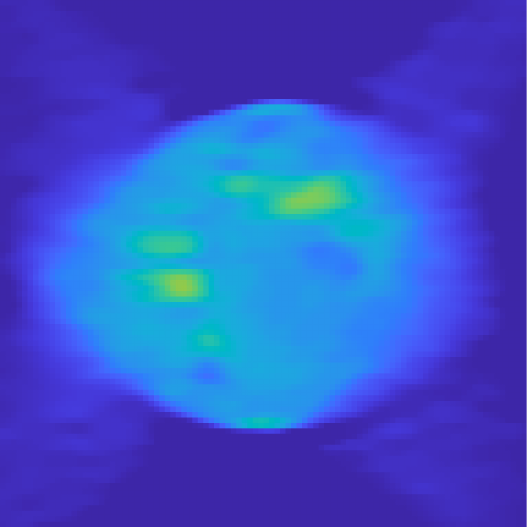};
    
    \node[anchor=north] at (axis cs:180,72) {\color{white} YZ};

   \addplot[] graphics[xmin=216,ymin=0,xmax=288,ymax=72] {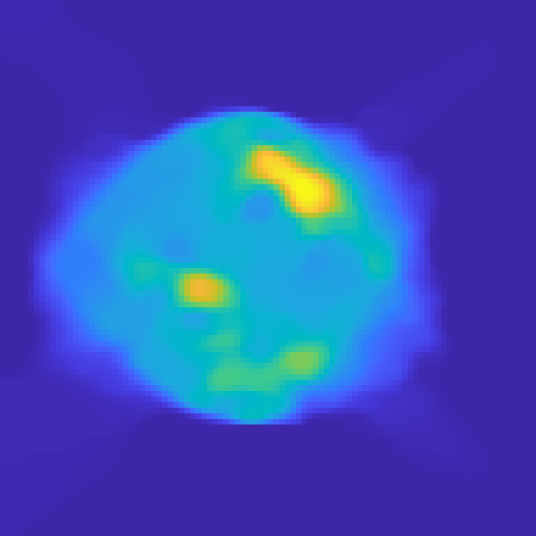};
     
    \node[anchor=north] at (axis cs:252,72) {\color{white} XZ}; 
   
   \addplot[] graphics[xmin=288,ymin=0,xmax=360,ymax=72] {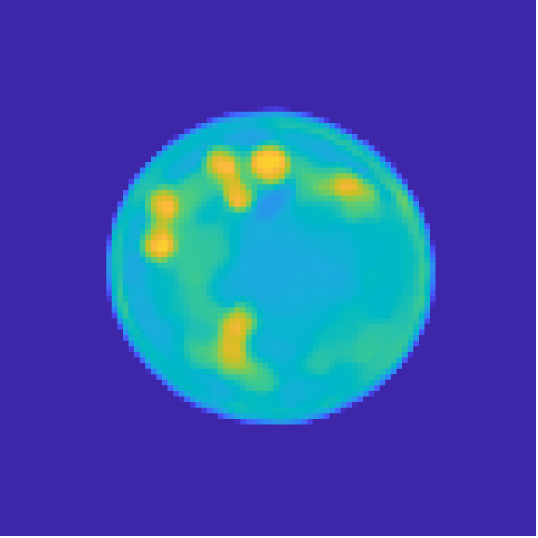};    
 
    \node[anchor=north] at (axis cs:324,72) {\color{white} XY};
    
  \node[anchor=base] at (axis cs:324,2) {\color{white} $\text{iter.}=30$};
  
   \addplot[] graphics[xmin=360,ymin=0,xmax=432,ymax=72] {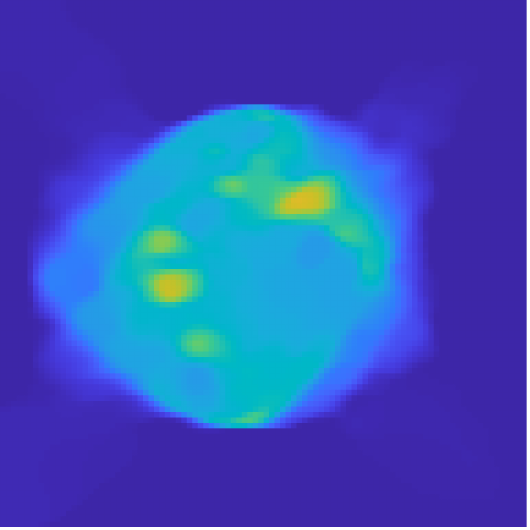};
    \node[anchor=north] at (axis cs:396,72) {\color{white} YZ};

    \draw[->, line width=1pt, draw=white] (267,3.2) -- (260,2.5);

\draw[->, line width=1pt, draw=white] (380,5.6) -- (388,4.8);
 \end{axis}

\begin{axis}[name={wang1},at={(ASPM1.south west)},anchor = north west,ylabel = \wangs{},ylabel style={yshift=0cm,xshift=0cm,anchor=base,align=center},
    xmin = 0,xmax = 432,ymin = 0,ymax = 72, width=0.95\textwidth,
        scale only axis,
        enlargelimits=false,
       axis line style={draw=none},
       tick style={draw=none},
        axis equal image,
        xticklabels={,,},yticklabels={,,}
       ]
    \draw[very thick,dashed,black] (axis cs:216,72) -- (216,0);
    \addplot[] graphics[xmin=0,ymin=0,xmax=72,ymax=72] {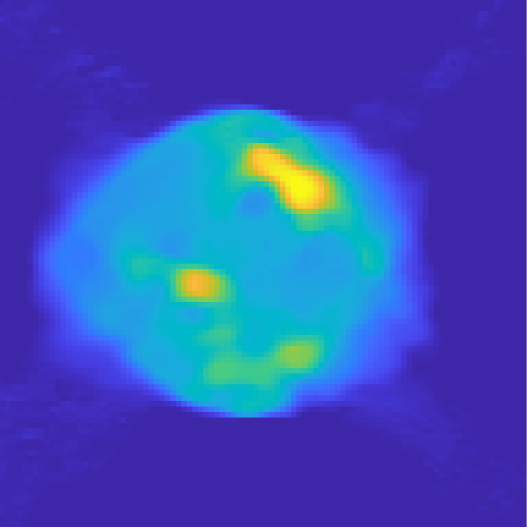};

    \addplot[] graphics[xmin=72,ymin=0,xmax=144,ymax=72] {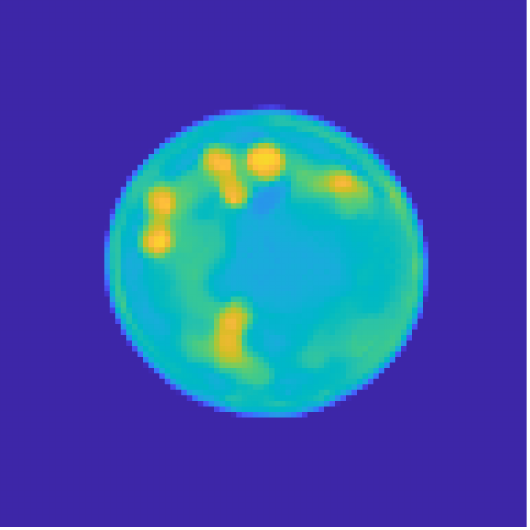};
    

     \node[anchor=base] at (axis cs:108,2) {\color{white} $\text{iter.}=10$};
     
    \addplot[] graphics[xmin=144,ymin=0,xmax=216,ymax=72] {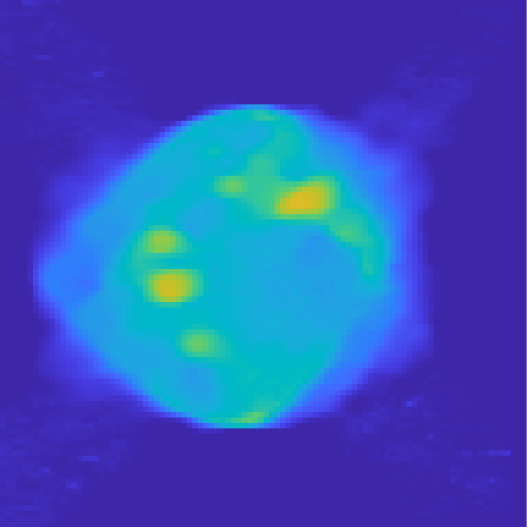};
    
    \addplot[] graphics[xmin=216,ymin=0,xmax=288,ymax=72] {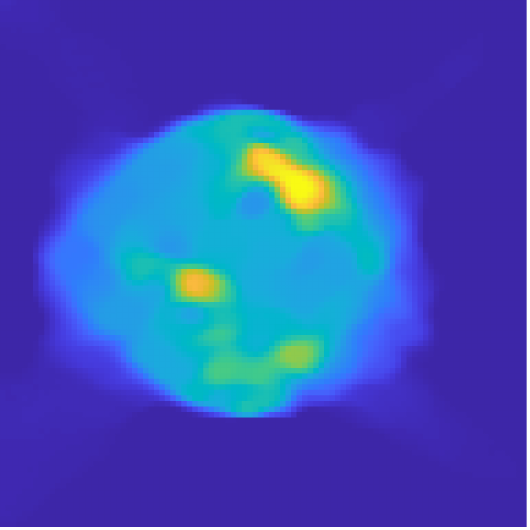};
     
    \addplot[] graphics[xmin=288,ymin=0,xmax=360,ymax=72] {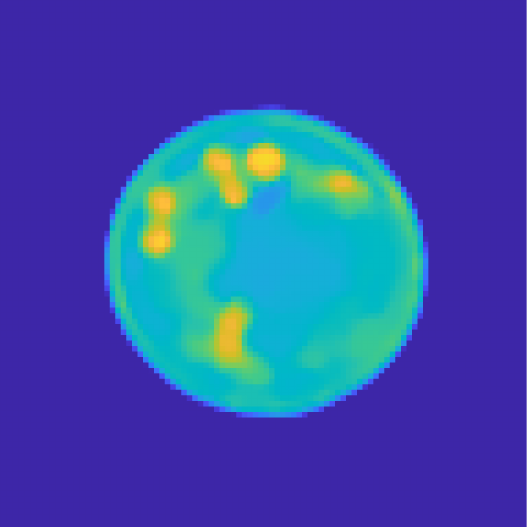};   
 
  \node[anchor=base] at (axis cs:324,2) {\color{white} $\text{iter.}=30$};

  \addplot[] graphics[xmin=360,ymin=0,xmax=432,ymax=72] {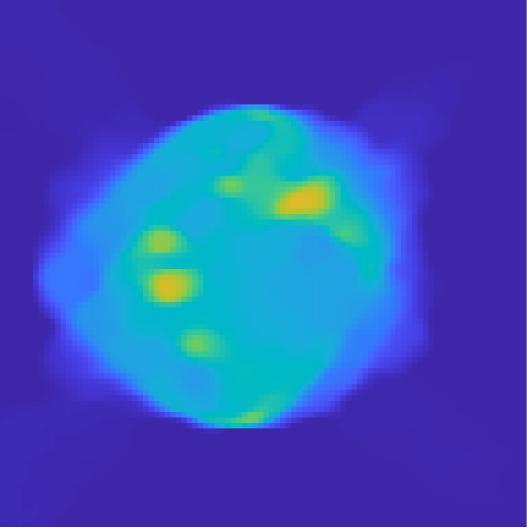};   
   
     \draw[->, line width=1pt, draw=white] (267,3.2) -- (260,2.5);
  
     \draw[->, line width=1pt, draw=white] (380,5.6) -- (388,4.8); 
 \end{axis}

\begin{axis}[name={QNRU1},at={(wang1.south west)},anchor = north west,ylabel = \ours{}-II,ylabel style={yshift=0cm,xshift=0cm,anchor=base,align=center},
    xmin = 0,xmax = 432,ymin = 0,ymax = 72, width=0.95\textwidth,
        scale only axis,
        enlargelimits=false,
       axis line style={draw=none},
       tick style={draw=none},
        axis equal image,
        xticklabels={,,},yticklabels={,,}
       ]
    \draw[very thick,dashed,black] (axis cs:216,72) -- (216,0);
    \addplot[] graphics[xmin=0,ymin=0,xmax=72,ymax=72] {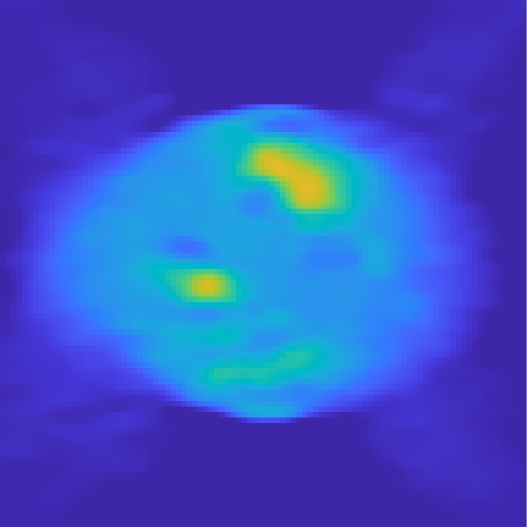}; 
    

  \addplot[] graphics[xmin=72,ymin=0,xmax=144,ymax=72] {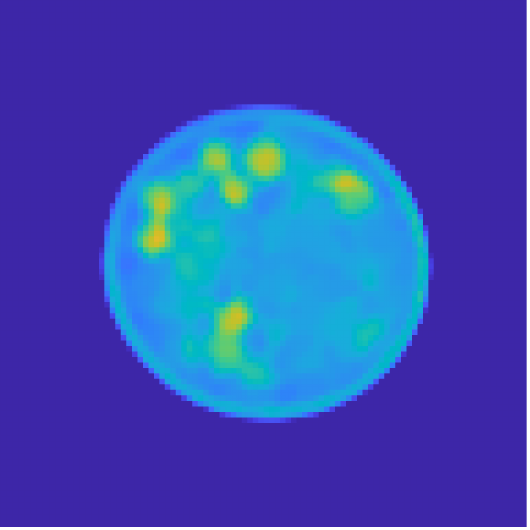}; 

     \node[anchor=base] at (axis cs:108,2) {\color{white} $\text{iter.}=10$};

     \addplot[] graphics[xmin=144,ymin=0,xmax=216,ymax=72] {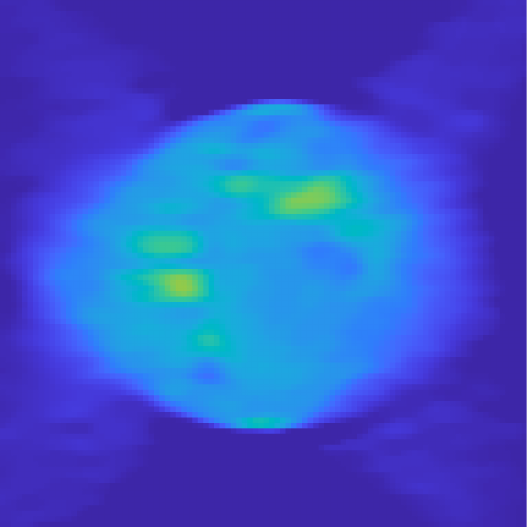}; 

    \addplot[] graphics[xmin=216,ymin=0,xmax=288,ymax=72] {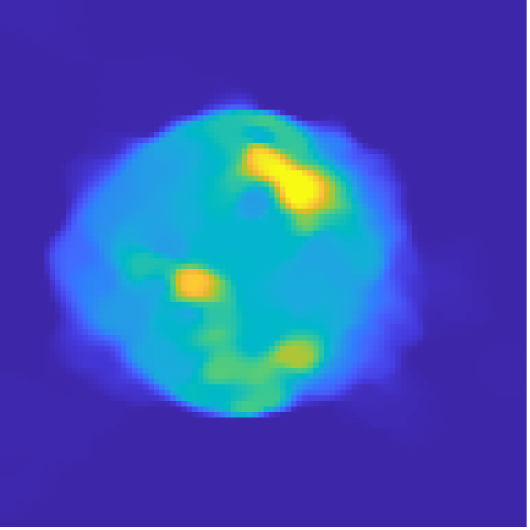}; 
     
    \addplot[] graphics[xmin=288,ymin=0,xmax=360,ymax=72] {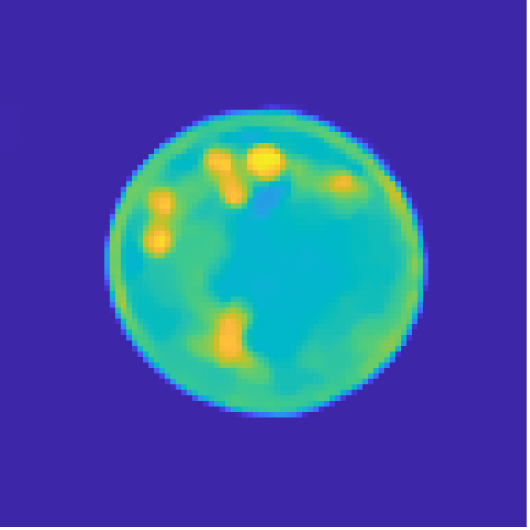};
 
  \node[anchor=base] at (axis cs:324,2) {\color{white} $\text{iter.}=30$};

  \addplot[] graphics[xmin=360,ymin=0,xmax=432,ymax=72] {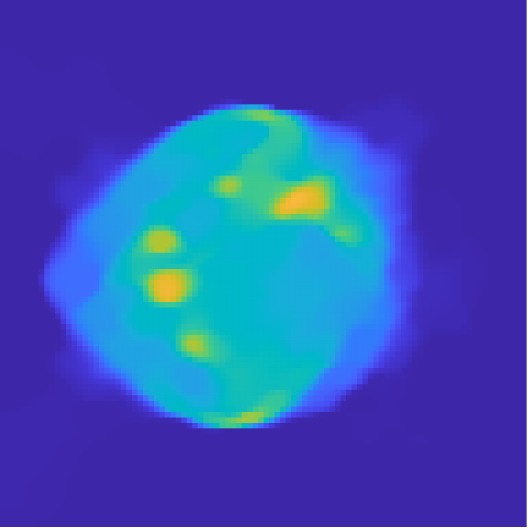};
     
     \draw[->, line width=1pt, draw=white] (267,3.2) -- (260,2.5);

     \draw[->, line width=1pt, draw=white] (380,5.6) -- (388,4.8); 
 \end{axis}

\begin{axis}[name={QN1},at={(QNRU1.south west)},anchor = north west,ylabel = \ours{}-I,ylabel style={yshift=0cm,xshift=0cm,anchor=base,align=center},
    xmin = 0,xmax = 432,ymin = 0,ymax = 72, width=0.95\textwidth,
        scale only axis,
        enlargelimits=false,
       axis line style={draw=none},
       tick style={draw=none},
        axis equal image,
        xticklabels={,,},yticklabels={,,}
       ]
    \draw[very thick,dashed,black] (axis cs:216,72) -- (216,0);
    \addplot[] graphics[xmin=0,ymin=0,xmax=72,ymax=72] {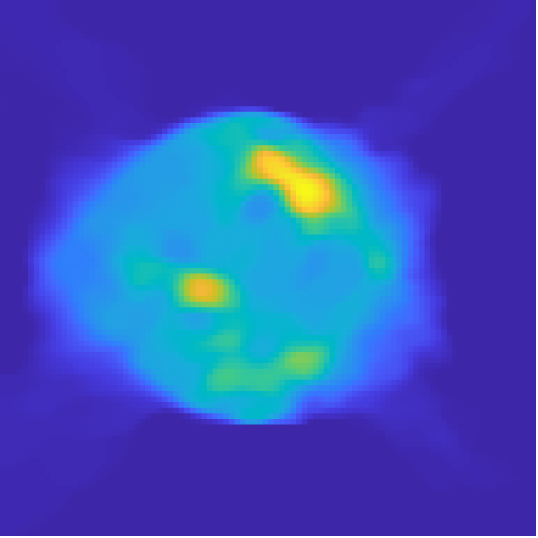}; 

  \addplot[] graphics[xmin=72,ymin=0,xmax=144,ymax=72] {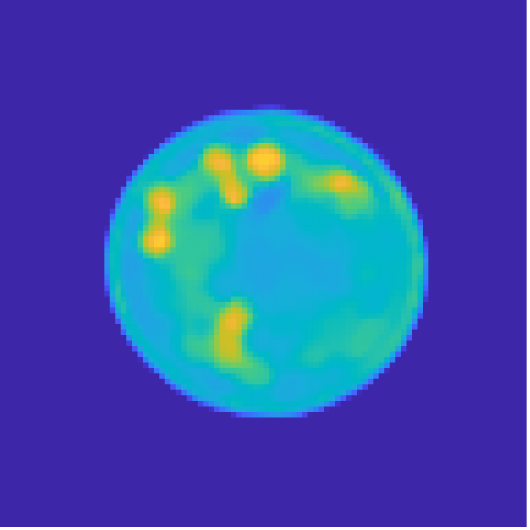}; 

     \node[anchor=base] at (axis cs:108,2) {\color{white} $\text{iter.}=10$};

     \addplot[] graphics[xmin=144,ymin=0,xmax=216,ymax=72] {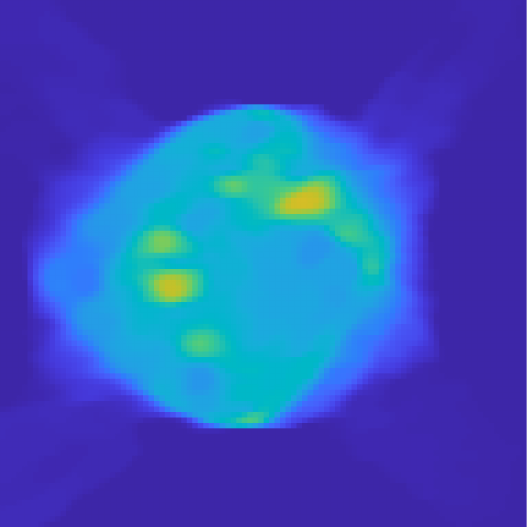}; 

    \addplot[] graphics[xmin=216,ymin=0,xmax=288,ymax=72] {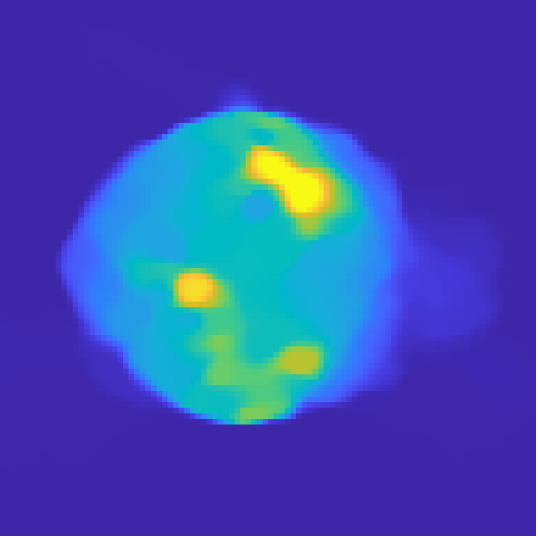}; 
     
    \addplot[] graphics[xmin=288,ymin=0,xmax=360,ymax=72] {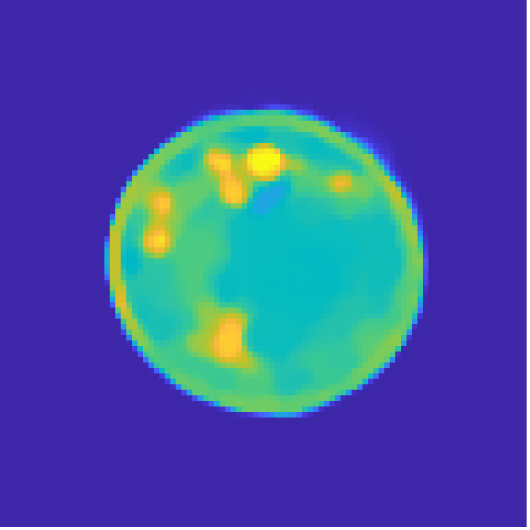};
 
  \node[anchor=base] at (axis cs:324,2) {\color{white} $\text{iter.}=30$};

  \addplot[] graphics[xmin=360,ymin=0,xmax=432,ymax=72] {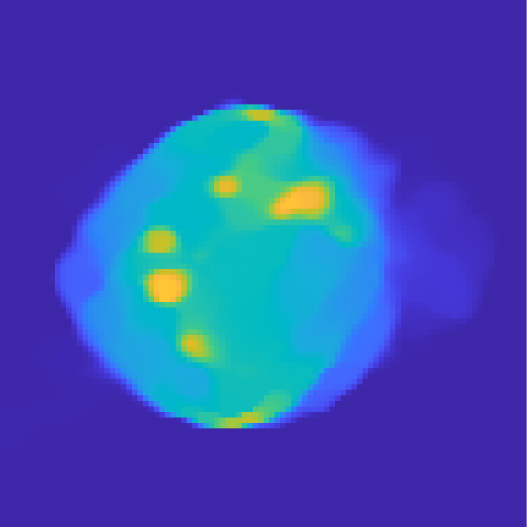};
      
      \draw[->, line width=1pt, draw=white] (267,3.2) -- (260,2.5);

      \draw[->, line width=1pt, draw=white] (380,5.6) -- (388,4.8);
 \end{axis}
 
 \begin{axis}[name={ASPM2},at={(QN1.south west)},anchor = north west,ylabel = ASPM,ylabel style={yshift=0cm,xshift=-0.cm,anchor=base,align=center},
    xmin = 0,xmax = 432,ymin = 0,ymax = 72, width=0.95\textwidth,
        scale only axis,
        enlargelimits=false,
       axis line style={draw=none},
       tick style={draw=none},
        axis equal image,
        xticklabels={,,},yticklabels={,,}]
  \draw[very thick,dashed,black] (axis cs:216,72) -- (216,0);

    \addplot[] graphics[xmin=0,ymin=0,xmax=72,ymax=72] {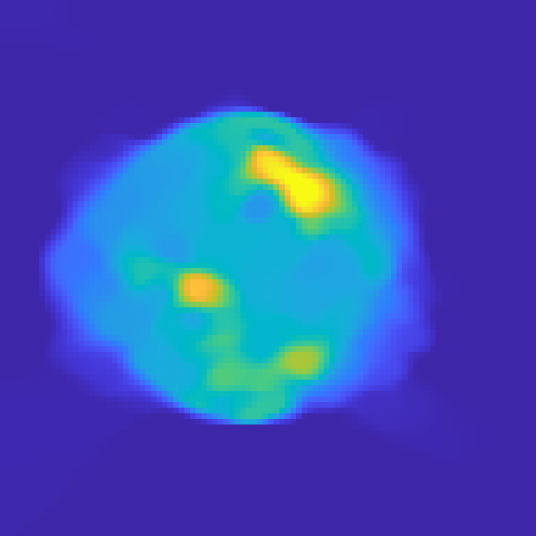};

    \addplot[] graphics[xmin=72,ymin=0,xmax=144,ymax=72] {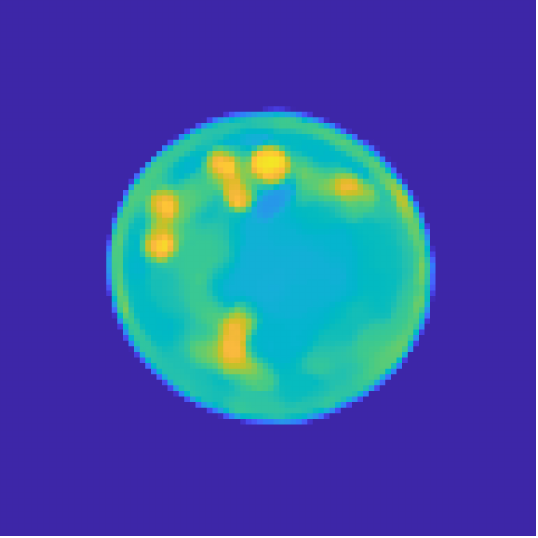};
    

     \node[anchor=base] at (axis cs:108,2) {\color{white} $\text{iter.}=40$};

     \addplot[] graphics[xmin=144,ymin=0,xmax=216,ymax=72] {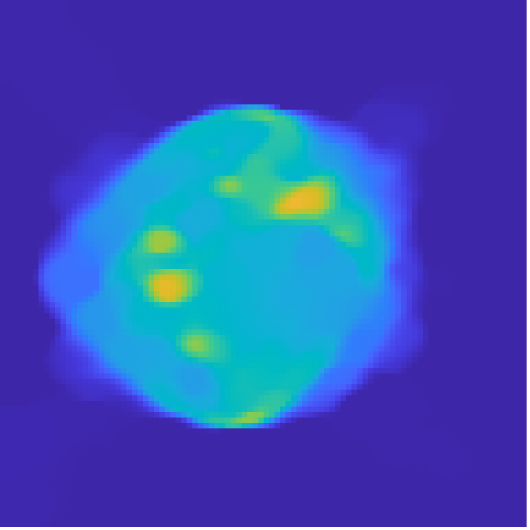};
    
    \addplot[] graphics[xmin=216,ymin=0,xmax=288,ymax=72] {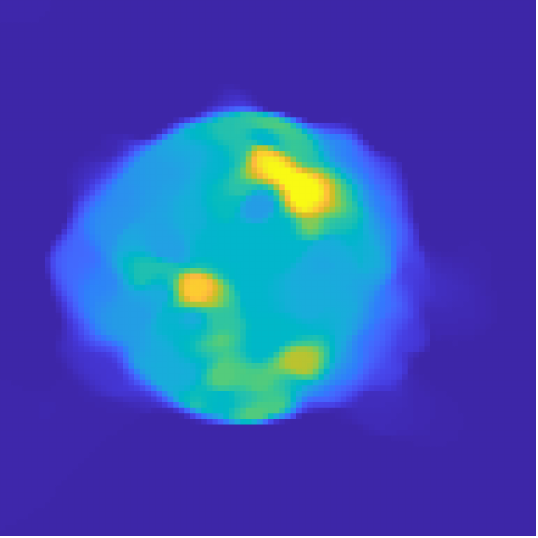};   
     
     \addplot[] graphics[xmin=288,ymin=0,xmax=360,ymax=72] {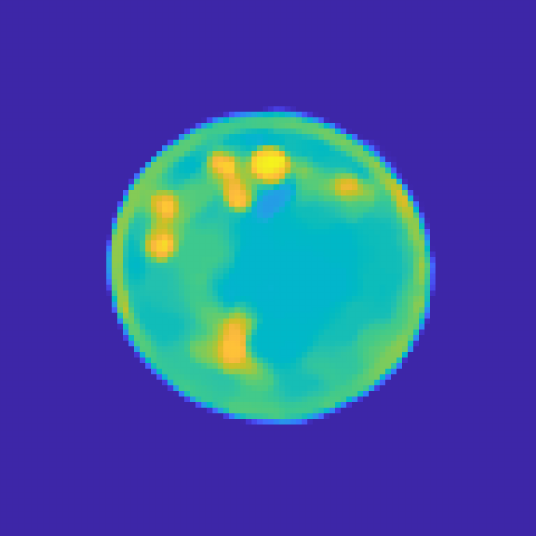};     
 
  \node[anchor=base] at (axis cs:324,2) {\color{white} $\text{iter.}=50$};
  
  \addplot[] graphics[xmin=360,ymin=0,xmax=432,ymax=72] {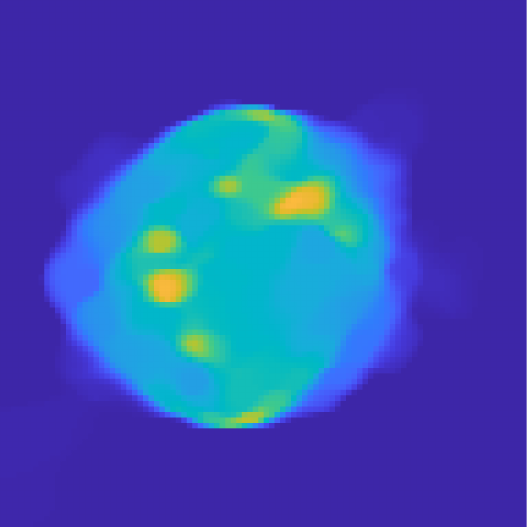};
    \draw[->, line width=1pt, draw=white] (267,3.2) -- (260,2.5);

    \draw[->, line width=1pt, draw=white] (380,5.6) -- (388,4.8);
 \end{axis}

  \begin{axis}[name=wang2,at={(ASPM2.south west)},anchor = north west,ylabel = \wangs{}, ylabel style={yshift=0cm,xshift=-0.cm,anchor=base,align=center},   xmin = 0,xmax = 432,ymin = 0,ymax = 72, width=0.95\textwidth,
        scale only axis,
        enlargelimits=false,
       axis line style={draw=none},
       tick style={draw=none},
        axis equal image,
        xticklabels={,,},yticklabels={,,}]

    \draw[very thick,dashed,black] (axis cs:216,72) -- (216,0);
        \addplot[] graphics[xmin=0,ymin=0,xmax=72,ymax=72] {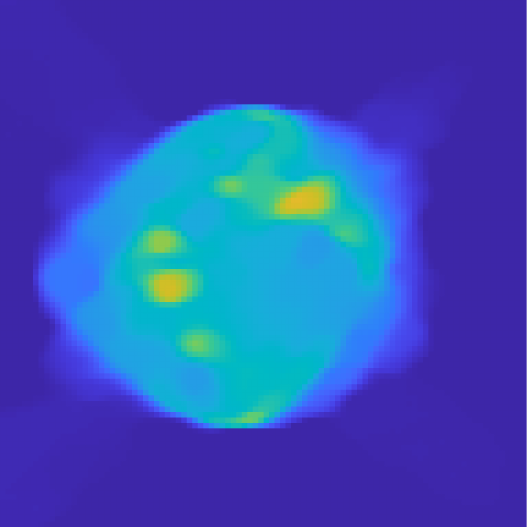};
    \addplot[] graphics[xmin=72,ymin=0,xmax=144,ymax=72] {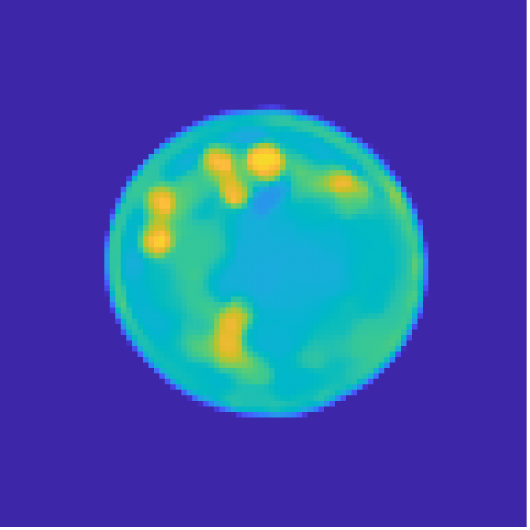};

     \node[anchor=base] at (axis cs:108,2) {\color{white} $\text{iter.}=40$};
     \addplot[] graphics[xmin=144,ymin=0,xmax=216,ymax=72] {figs/figures/RealReco/QN_wang_YZ_iter_40.pdf};
    
    \addplot[] graphics[xmin=216,ymin=0,xmax=288,ymax=72] {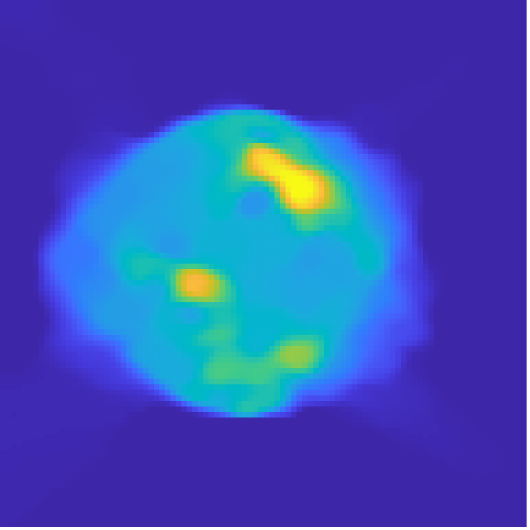};
     
    \addplot[] graphics[xmin=288,ymin=0,xmax=360,ymax=72] {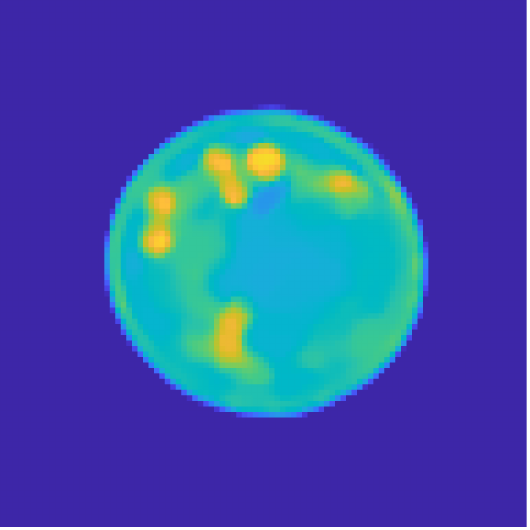};
 
    \node[anchor=base] at (axis cs:324,2) {\color{white} $\text{iter.}=50$};
    \addplot[] graphics[xmin=360,ymin=0,xmax=432,ymax=72] {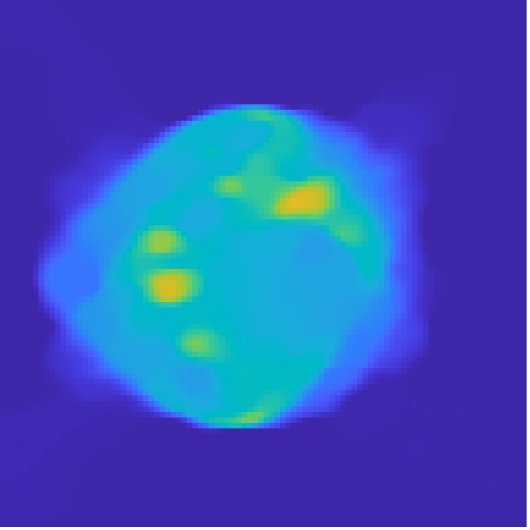};
   
   \draw[->, line width=1pt, draw=white] (267,3.2) -- (260,2.5);

   \draw[->, line width=1pt, draw=white] (380,5.6) -- (388,4.8);
 \end{axis}

\begin{axis}[name=QNUR2,at={(wang2.south west)},anchor = north west,ylabel = \ours{}-II,ylabel style={yshift=0cm,xshift=-0cm,anchor=base,align=center}, 
    xmin = 0,xmax = 432,ymin = 0,ymax = 72, width=0.95\textwidth,
         scale only axis,
        enlargelimits=false,
       axis line style={draw=none},
       tick style={draw=none},
        axis equal image,
        xticklabels={,,},yticklabels={,,}]%
\coordinate (colorbarc) at (axis cs:216,0);
\draw[very thick,dashed,black] (axis cs:216,72) -- (216,0);

\addplot[] graphics[xmin=0,ymin=0,xmax=72,ymax=72] {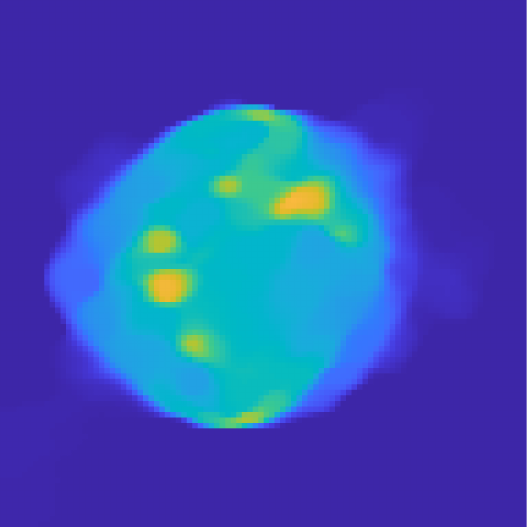}; 
   
  \addplot[] graphics[xmin=72,ymin=0,xmax=144,ymax=72] {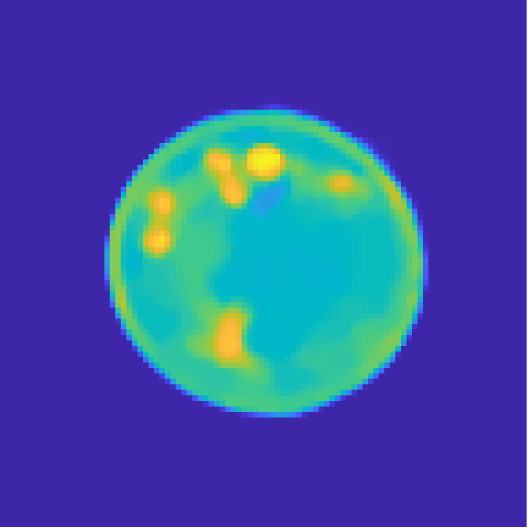}; 
   
     \node[anchor=base] at (axis cs:108,2) {\color{white} $\text{iter.}=40$};

     \addplot[] graphics[xmin=144,ymin=0,xmax=216,ymax=72] {figs/figures/RealReco/QNUR_YZ_iter_40.pdf};

    \addplot[] graphics[xmin=216,ymin=0,xmax=288,ymax=72] {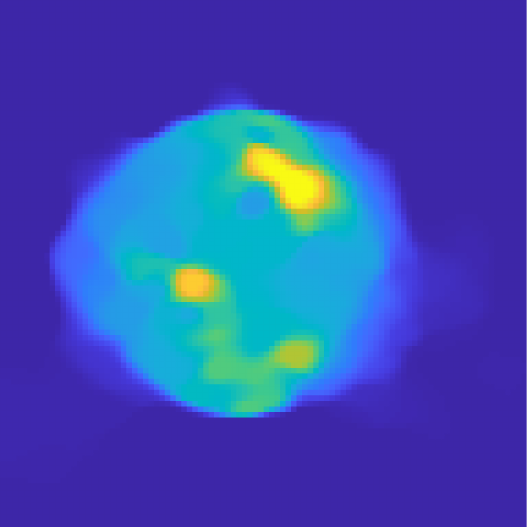}; 
     
    \addplot[] graphics[xmin=288,ymin=0,xmax=360,ymax=72] {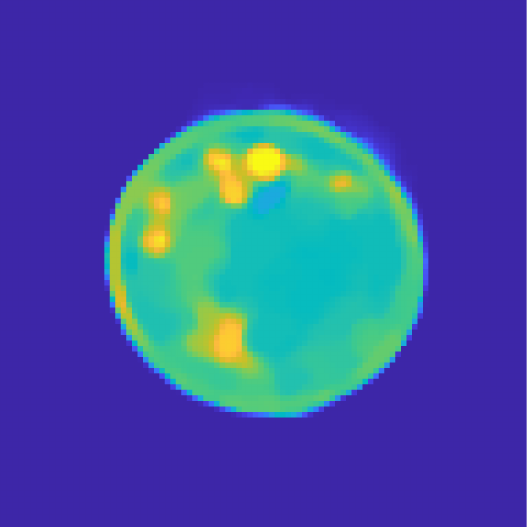};
   
  \node[anchor=base] at (axis cs:324,2) {\color{white} $\text{iter.}=50$};

  \addplot[] graphics[xmin=360,ymin=0,xmax=432,ymax=72] {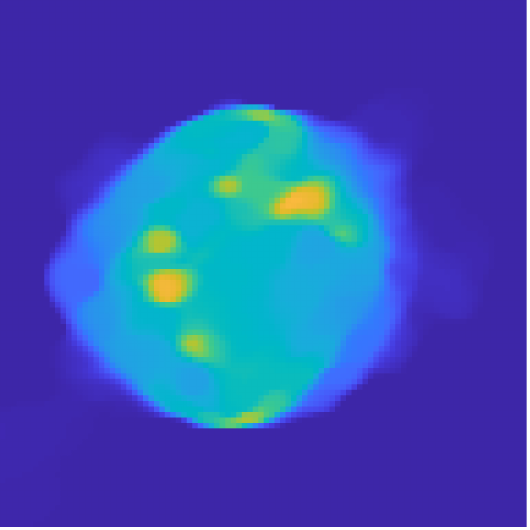};
  
  \draw[->, line width=1pt, draw=white] (267,3.2) -- (260,2.5);

  \draw[->, line width=1pt, draw=white] (380,5.6) -- (388,4.8);
 \end{axis}

\begin{axis}[name=QN2,at={(QNUR2.south west)},anchor = north west,ylabel = \ours{}-I,ylabel style={yshift=0cm,xshift=-0cm,anchor=base,align=center}, 
    xmin = 0,xmax = 432,ymin = 0,ymax = 72, width=0.95\textwidth,
        scale only axis,
        enlargelimits=false,
       axis line style={draw=none},
       tick style={draw=none},
        axis equal image,
        xticklabels={,,},yticklabels={,,},
        ,     
        colormap={parula}{
            rgb=(0.208100000000000,0.166300000000000,0.529200000000000)
            rgb=(0.211623809523810,0.189780952380952,0.577676190476191)
            rgb=(0.212252380952381,0.213771428571429,0.626971428571429)
            rgb=(0.208100000000000,0.238600000000000,0.677085714285714)
            rgb=(0.195904761904762,0.264457142857143,0.727900000000000)
            rgb=(0.170728571428571,0.291938095238095,0.779247619047619)
            rgb=(0.125271428571429,0.324242857142857,0.830271428571429)
            rgb=(0.0591333333333334,0.359833333333333,0.868333333333333)
            rgb=(0.0116952380952381,0.387509523809524,0.881957142857143)
            rgb=(0.00595714285714286,0.408614285714286,0.882842857142857)
            rgb=(0.0165142857142857,0.426600000000000,0.878633333333333)
            rgb=(0.0328523809523810,0.443042857142857,0.871957142857143)
            rgb=(0.0498142857142857,0.458571428571429,0.864057142857143)
            rgb=(0.0629333333333333,0.473690476190476,0.855438095238095)
            rgb=(0.0722666666666667,0.488666666666667,0.846700000000000)
            rgb=(0.0779428571428572,0.503985714285714,0.838371428571429)
            rgb=(0.0793476190476190,0.520023809523810,0.831180952380952)
            rgb=(0.0749428571428571,0.537542857142857,0.826271428571429)
            rgb=(0.0640571428571428,0.556985714285714,0.823957142857143)
            rgb=(0.0487714285714286,0.577223809523810,0.822828571428572)
            rgb=(0.0343428571428572,0.596580952380952,0.819852380952381)
            rgb=(0.0265000000000000,0.613700000000000,0.813500000000000)
            rgb=(0.0238904761904762,0.628661904761905,0.803761904761905)
            rgb=(0.0230904761904762,0.641785714285714,0.791266666666667)
            rgb=(0.0227714285714286,0.653485714285714,0.776757142857143)
            rgb=(0.0266619047619048,0.664195238095238,0.760719047619048)
            rgb=(0.0383714285714286,0.674271428571429,0.743552380952381)
            rgb=(0.0589714285714286,0.683757142857143,0.725385714285714)
            rgb=(0.0843000000000000,0.692833333333333,0.706166666666667)
            rgb=(0.113295238095238,0.701500000000000,0.685857142857143)
            rgb=(0.145271428571429,0.709757142857143,0.664628571428572)
            rgb=(0.180133333333333,0.717657142857143,0.642433333333333)
            rgb=(0.217828571428571,0.725042857142857,0.619261904761905)
            rgb=(0.258642857142857,0.731714285714286,0.595428571428571)
            rgb=(0.302171428571429,0.737604761904762,0.571185714285714)
            rgb=(0.348166666666667,0.742433333333333,0.547266666666667)
            rgb=(0.395257142857143,0.745900000000000,0.524442857142857)
            rgb=(0.442009523809524,0.748080952380952,0.503314285714286)
            rgb=(0.487123809523809,0.749061904761905,0.483976190476191)
            rgb=(0.530028571428571,0.749114285714286,0.466114285714286)
            rgb=(0.570857142857143,0.748519047619048,0.449390476190476)
            rgb=(0.609852380952381,0.747314285714286,0.433685714285714)
            rgb=(0.647300000000000,0.745600000000000,0.418800000000000)
            rgb=(0.683419047619048,0.743476190476191,0.404433333333333)
            rgb=(0.718409523809524,0.741133333333333,0.390476190476190)
            rgb=(0.752485714285714,0.738400000000000,0.376814285714286)
            rgb=(0.785842857142857,0.735566666666667,0.363271428571429)
            rgb=(0.818504761904762,0.732733333333333,0.349790476190476)
            rgb=(0.850657142857143,0.729900000000000,0.336028571428571)
            rgb=(0.882433333333333,0.727433333333333,0.321700000000000)
            rgb=(0.913933333333333,0.725785714285714,0.306276190476191)
            rgb=(0.944957142857143,0.726114285714286,0.288642857142857)
            rgb=(0.973895238095238,0.731395238095238,0.266647619047619)
            rgb=(0.993771428571429,0.745457142857143,0.240347619047619)
            rgb=(0.999042857142857,0.765314285714286,0.216414285714286)
            rgb=(0.995533333333333,0.786057142857143,0.196652380952381)
            rgb=(0.988000000000000,0.806600000000000,0.179366666666667)
            rgb=(0.978857142857143,0.827142857142857,0.163314285714286)
            rgb=(0.969700000000000,0.848138095238095,0.147452380952381)
            rgb=(0.962585714285714,0.870514285714286,0.130900000000000)
            rgb=(0.958871428571429,0.894900000000000,0.113242857142857)
            rgb=(0.959823809523810,0.921833333333333,0.0948380952380953)
            rgb=(0.966100000000000,0.951442857142857,0.0755333333333333)
            rgb=(0.976300000000000,0.983100000000000,0.0538000000000000)
        }, colorbar horizontal,
    colorbar style={
        at = (colorbarc),
        anchor={outer north},
        width= 13cm,
        height=0.3cm,
        yshift = 0.33cm,
        xshift = -0.45cm,
        x tick label style ={color=black,anchor=base,yshift=-0.mm,xshift = (-4mm + (\ticknum)*8mm)},
        xtick={1.338,1.48},
    },   point meta min=1.338,
           point meta max=1.48]%
\coordinate (colorbarc) at (axis cs:216,0);
\draw[very thick,dashed,black] (axis cs:216,72) -- (216,0);
\addplot[] graphics[xmin=0,ymin=0,xmax=72,ymax=72] {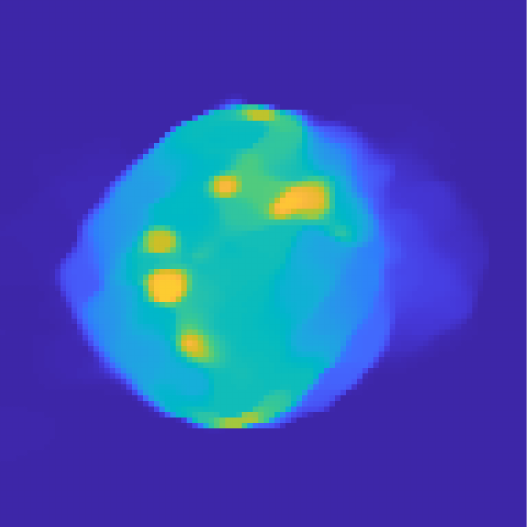}; 
   
  \addplot[] graphics[xmin=72,ymin=0,xmax=144,ymax=72] {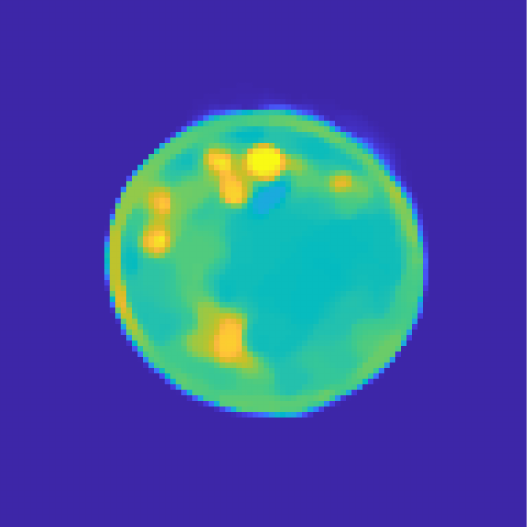}; 
   
     \node[anchor=base] at (axis cs:108,2) {\color{white} $\text{iter.}=40$};

     \addplot[] graphics[xmin=144,ymin=0,xmax=216,ymax=72] {figs/figures/RealReco/QN_YZ_iter_40.pdf}; 
   
    \addplot[] graphics[xmin=216,ymin=0,xmax=288,ymax=72] {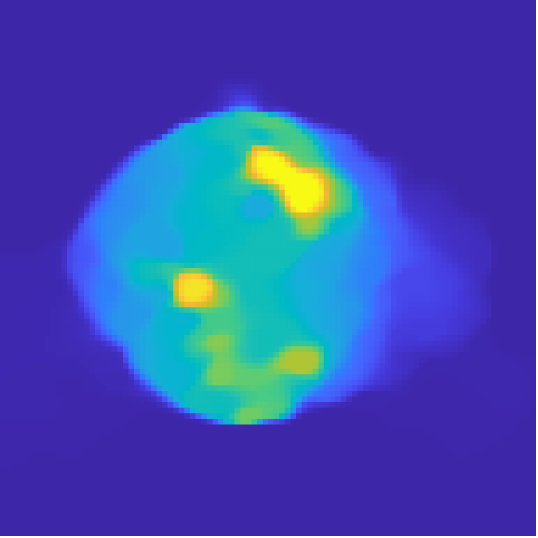}; 
     
    \addplot[] graphics[xmin=288,ymin=0,xmax=360,ymax=72] {figs/figures/RealReco/QN_XY_iter_50.pdf};
   
  \node[anchor=base] at (axis cs:324,2) {\color{white} $\text{iter.}=50$};

  \addplot[] graphics[xmin=360,ymin=0,xmax=432,ymax=72] {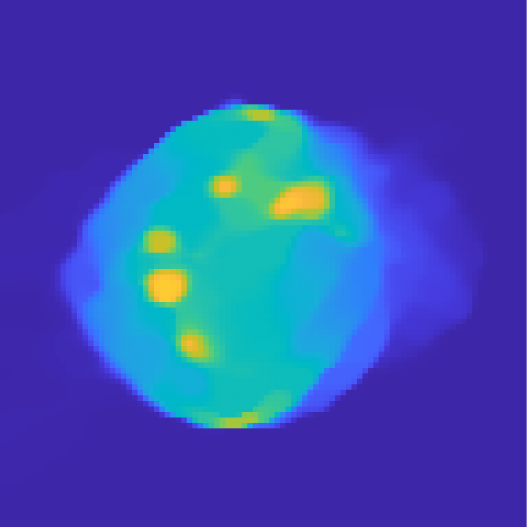};

\draw[->, line width=1pt, draw=white] (267,3.2) -- (260,2.5);  
\draw[->, line width=1pt, draw=white] (380,5.6) -- (388,4.8);
 \end{axis}
\end{tikzpicture}

\caption{Orthoviews of the reconstructed 3D refractive-index maps obtained using ASPM, \wangs{}, and \ours{}-I/II algorithms on real data (yeast cell) with the Lippmann-Schwinger model at the $10$th, $30$th, $40$th, and $50$th iterations.
}
\label{fig:LS:recon:real}
\end{figure*}